\documentclass[onefignum,onetabnum]{siamart220329}

\usepackage{soul}
\usepackage{lipsum}
\usepackage{amsfonts}
\usepackage{amsmath}
\usepackage{graphicx}
\usepackage{epstopdf}
\usepackage{algorithmic}
\ifpdf
  \DeclareGraphicsExtensions{.eps,.pdf,.png,.jpg}
\else
  \DeclareGraphicsExtensions{.eps}
\fi

\newsiamremark{example}{Example}
\newsiamremark{remark}{Remark}
\newsiamremark{hypothesis}{Hypothesis}
\crefname{hypothesis}{Hypothesis}{Hypotheses}
\newsiamthm{claim}{Claim}

\headers{Optimality Conditions for Multivariate Chebyshev Approximation}{A. Goldsztejn}

\title{Optimality Conditions for Multivariate Chebyshev Approximation: A Survey
\thanks{{\bf Funding:} This author was funded by the CNRS grant Emergence AMEGO.
}}

\author{Alexandre Goldsztejn\thanks{Nantes Université, École Centrale Nantes, IMT Atlantique, CNRS, LS2N, UMR 6004, F-44000 Nantes, France (\email{alexandre.goldsztejn@cnrs.fr}%
  ).}
}

\ifpdf
\hypersetup{
  pdftitle={Optimality Conditions for Multivariate Chebyshev Approximation},
  pdfauthor={Alexandre Goldsztejn}
}
\fi

\usepackage[caption=false]{subfig}
\usepackage{soul} % for \ul{}, underline with linebreak
\newcommand{\lb}[1]{\underline{#1}}
\newcommand{\ub}[1]{\overline{#1}}
\newcommand{\lx}{\lb{x}}
\newcommand{\ux}{\ub{x}}
\newcommand{\ba}{{\ub{a}}}
\newcommand{\bp}{\ub{p}}

\newcommand{\bg}{\ub{g}}
\newcommand{\pLP}{\tilde{p}}
\newcommand{\aLP}{\tilde{a}}
\newcommand{\xLP}{\tilde{x}}
\newcommand{\sLP}{\tilde{s}}
\newcommand{\lLP}{\tilde{\lambda}}
\newcommand{\SLP}{\tilde{S}}
\newcommand{\aNEW}{\ub{a}}
\newcommand{\xNEW}{\ub{x}}
\newcommand{\SNEW}{\ub{S}}
\newcommand{\lNEW}{\ub{\lambda}}
\DeclareMathOperator{\one}{\boldsymbol{1}}

\DeclareMathOperator{\abs}{abs}

\DeclareMathOperator{\conv}{conv}
\DeclareMathOperator{\sign}{sign}
\DeclareMathOperator{\Sign}{Sign}
\DeclareMathOperator{\act}{act}
\DeclareMathOperator{\ext}{ext}
\DeclareMathOperator{\sig}{\Sigma}

\DeclareMathOperator{\interior}{int}

\newcommand{\N}{\mathbb{N}}
\newcommand{\R}{\mathbb{R}}
\newcommand{\B}{\mathbb{B}}
\newcommand{\Sp}{\mathbb{S}}

\newcommand{\C}{\mathcal C}
\newcommand{\V}{\mathcal P}
\newcommand{\F}{\mathcal F}
\newcommand{\tX}{\tilde X}

\usepackage{xcolor}

\begin{document}

\maketitle

% REQUIRED
\begin{abstract}
Uniform polynomial approximation, also called minimax approximation or Chebyshev approximation, consists in searching polynomial approximation that minimizes the worst case error. Optimality conditions for the uniform approximation of univariate functions defined in an interval are governed by the equioscillation theorem, which is also a key ingredient in algorithms for computing best uniform approximation, like Remez's algorithm and the two-step approach. Multivariate polynomial approximation is more complicated, and several optimality conditions for uniform multivariate polynomial approximation generalize the equioscillation theorem. We review these conditions, including, from oldest to newest, Kirchberger's kernel condition, Kolmogorov criteria, Rivlin and Shapiro's annihilating measures. An emphasis is given to conditions for strong optimality, which have some strong theoretical and practical importance, including Bartelt's and Smarzewsky's conditions. Optimality conditions related to more general relative Chebyshev centers are also presented, including Tanimoto's and Levis et al.'s conditions. In a second step, conditions obtained by standard convex analysis, subdifferential and directional derivative, applied to uniform approximation are formulated. Their relationship to previous conditions is investigated, providing sometimes enlightening interpretations of the laters, e.g., relating Kolmogorov criterion with directional derivative, and strong uniqueness with sharp minimizers. Finally, numerical applications of the two-step approach to three uniform approximation problems are presented, namely the approximation of the multidimensional Runge function, the approximation of the two dimensional inverse model of the DexTAR parallel robot, and the approximation problem consisting in minimizing the sum of both the polynomial approximation error and the polynomial evaluation error in Horner form.
\end{abstract}

% REQUIRED
\begin{keywords}
Chebyshev approximation problem, multivariate polynomial approximation, optimality conditions, strong uniqueness, convex optimization
\end{keywords}

% REQUIRED
\begin{MSCcodes}
41A50 (best approximation, Chebyshev systems), 41A63 (multidimensional problems), 41A52 (uniqueness of best approximation), 90C25 (convex programming), 41-02 (research exposition, survey article)
\end{MSCcodes}

%%%%%%%%%%%%%%%%%%%%%%%%%%%%%%%%%%%%%%%%%%%%%%%%%%%%%%%%%%%%%%%%%%%%%%
%%%%%%%%%%%%%%%%%%%%%%%%%%%%%%%%%%%%%%%%%%%%%%%%%%%%%%%%%%%%%%%%%%%%%%

%\newpage
\vspace{0.2cm}
\tableofcontents
%\newpage

%%%%%%%%%%%%%%%%%%%%%%%%%%%%%%%%%%%%%%%%%%%%%%%%%%%%%%%%%%%%%%%%%%%%%%
%%%%%%%%%%%%%%%%%%%%%%%%%%%%%%%%%%%%%%%%%%%%%%%%%%%%%%%%%%%%%%%%%%%%%%

\section{Introduction}\label{s:intro}

Polynomial approximation ranges from very practical considerations, arising from the need of having simple polynomial expressions to model complex phenomena, to deep theoretical quantitative and qualitative information on approximability of certain classes of functions. A cornerstone in the theory of polynomial approximation of continuous univariate real valued functions of a compact interval is Weierstrass approximation theorem, which simply but accurately states that such functions can be approximated by polynomials uniformly with arbitrary precision. Approximation in the sense of the uniform norm corresponds to minimizing the worst case error in the domain, and is also called minimax approximation, or Chebyshev approximation. Using the usual notation $E_n(f)$ for the error of the best uniform approximation of $f$ by polynomial of degree less or equal to $n$, Weierstrass Approximation Theorem says that $E_n(f)$ converges to zero as $n$ goes to infinity. Providing some quantitative information about the asymptotic of $E_n(f)$ has been thoroughly investigated, starting by the concomitant works of Bernstein and Jackson, who discovered direct and converse theorems relating properties of functions to be approximated, e.g., Lipschitz continuity or differentiability, with the rate of convergence of their uniform polynomial best approximations. Among direct theorems, if $f$ is $p$ times differentiable with $p^\mathrm{th}$ derivative Lipschitz continuous ($p=0$ meaning $f$ is simply Lipschitz) then $E_n(f)=O(\tfrac{1}{n^{p+1}})$, formalizing the intuitive idea that the smoother a function the better it is approximated by polynomials. A remarkable converse theorem states that if $E_n(f)$ converges quick enough to zero so that the series $\sum_{n=1}^\infty E_n(f)\, n^{p-1}$ converges then $f$ has to be $p$ times continuously differentiable.
% This is Corollary 5.9 page 185 of~\cite{Steffens2006}
The interested reader is referred to the passionating historical book~\cite{Steffens2006} and to~\cite{Meinardus1967,Holland1981,Powel1981,Cheney1982,Trefethen2019} (\cite[Jackson's Theorem V page 147]{Cheney1982} and~\cite[Chapter 16]{Powel1981} provide some statements with explicit constants and elementary proofs).
%
%\begin{equation}
%	\sum_{n=1}^\infty E_n(f)\,n^{p-1}<\infty \ \Longrightarrow f\in C^p([a,b]) \ \Longrightarrow \lim_{n\rightarrow\infty} E_n(f)\,n^p=0.
%\end{equation}

\begin{figure}[t!]
	\centering
	\includegraphics[width=0.45\linewidth]{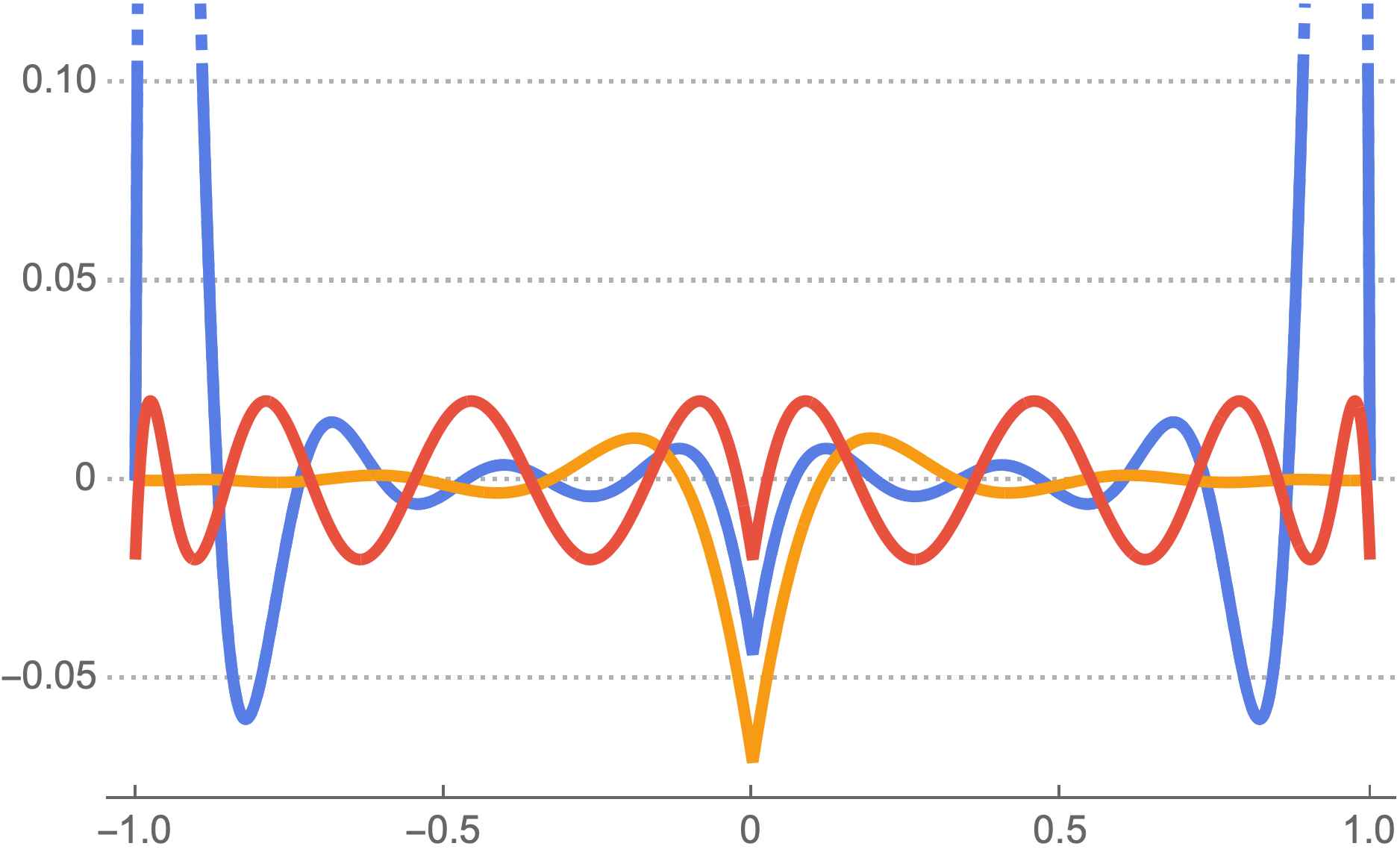}\hspace{1cm}\includegraphics[width=0.45\linewidth]{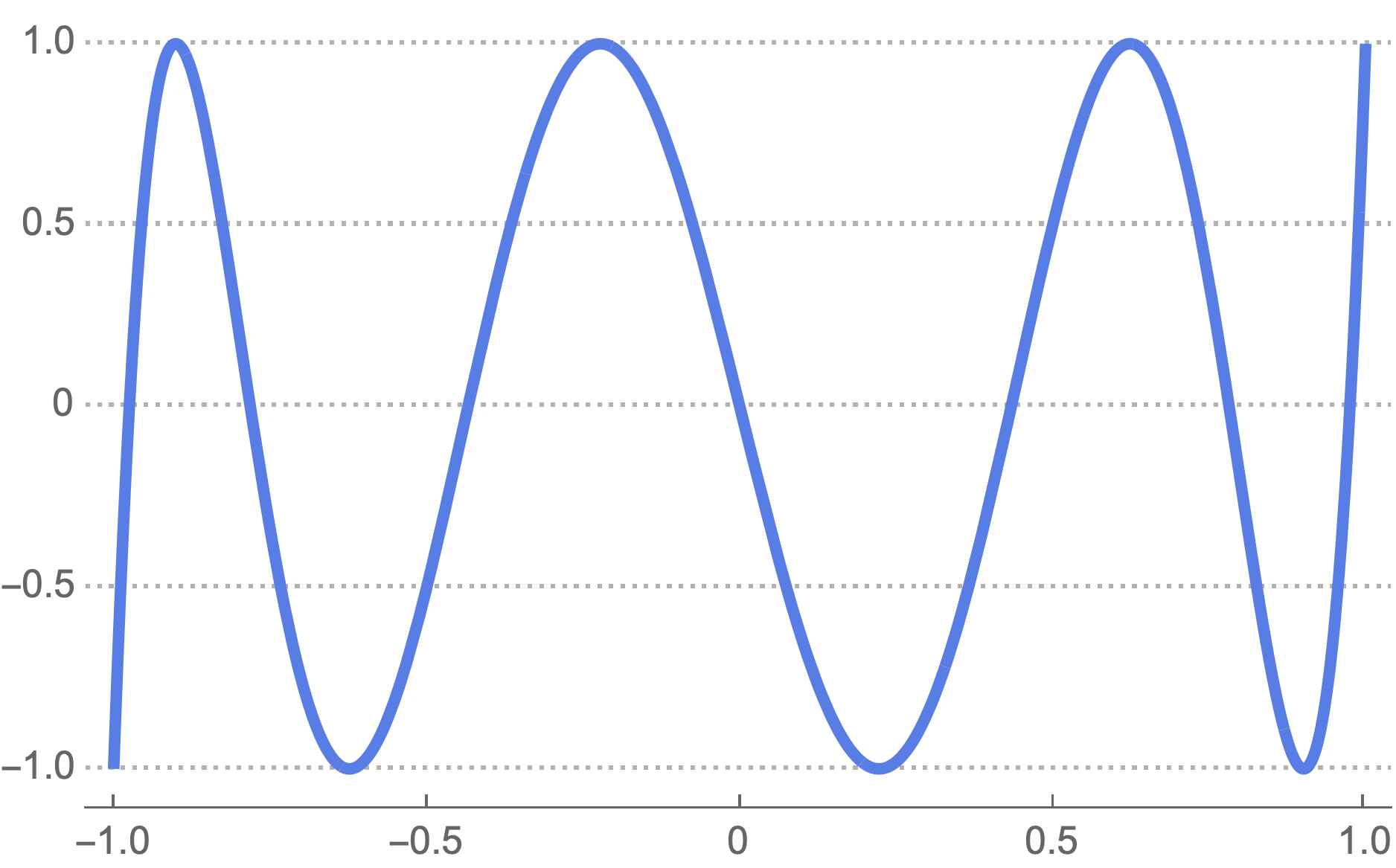}
	\caption{Left plot: error functions for three degree $15$ polynomials approximating $\abs$ in the interval $[-1,1]$ (blue, orange and red corresponding to equidistant nodes interpolation, Chebyshev nodes interpolation and best uniform). Right plot: $T_7(x)=\cos(7\arccos(x))=64 x^7- 112 x^5+ 56 x^3-7 x  $, which has $8$ oscillating extrema. As a consequence $p(x)=112 x^5- 56 x^3+7 x $ is a best degree $6$ uniform approximation of $f(x)=64x^7$.\label{fig:univariate-polynomial-approximtion}}
\end{figure}

The more practical problem of actually building polynomial approximations has been tackled in several ways. The simplest way is to interpolate a degree $n$ polynomial at given $n+1$ points, which offers cheap and accurate polynomial approximations provided that interpolation points are chosen correctly. Using equidistant interpolation points leads to the Runge phenomenon~\cite{Platte2011,Corless2020} and to poor quality approximations, e.g., the left plot of Figure~\ref{fig:univariate-polynomial-approximtion} shows in blue the degree $15$ polynomial interpolated at equidistant nodes for approximating the absolute value function, which shows a high error near the interval endpoints. Using Chebyshev interpolation nodes leads to accurate polynomial approximations, e.g., the left plot of Figure~\ref{fig:univariate-polynomial-approximtion} shows in orange the degree $15$ polynomial interpolated at Chebyshev nodes for the absolute value function, which shows a very good error, although this approximation is seen to be more accurate at the endpoints where more Chebyshev nodes are located than at the center of the domain. Furthermore, this latter enjoys a close to best uniform error: Bernstein~\cite[Section 4 page 6]{Bernstein1918}\footnote{This old paper of Bernstein is available on the editor's website.
%, as well as the paper~\cite{Bernstein1914} where Bernstein proves the tight bounds of $\tfrac{1}{4(1+\sqrt{2})}\tfrac{1}{2n-1}\leq E_{\abs}(n)\leq \tfrac{1}{2n+1}$ on the error of the polynomial uniform approximation of the absolute value function inside $[-1,1]$.
} proved that the uniform approximation error of the degree $n$ polynomial interpolated at Chebyshev nodes is asymptotically $O(E_f(n)\,\log(n))$, hence at a factor $\log(n)$ of the best uniform approximation.
%This is usually proved using the Lebesgue constant of the interpolation process at given nodes, which is $O(\log(n))$ for $n$ Chebyshev nodes.
%\begin{remark}
%The uniform error of degree $n$ polynomials interpolating nodes $A=\{x_0,\ldots,x_n\}$ is usually quantified using the Lebesgue constant $\Delta(A)$, which depends only on the set $A$ of approximation nodes: its uniform approximation error is not greater than $(1+\Delta(A))E_f(n)$. It is well known that for $A_n$ containing $n$ equidistant nodes we have $\Delta(A_n)=O(\tfrac{2^{n+1}}{n\log(n)})$. Bernstein proved that in general for any sequence of $A_n$ containing $n$ nodes we have $\liminf_{n\rightarrow\infty}\tfrac{\Delta(A_n)}{\log(n)}=0$, and that for $A_n$ containing the Chebyshev nodes $\Delta(A_n)=O(\log(n+1))$.
%\end{remark}
Interpolation at Chebyshev nodes is often preferred to the slightly more accurate projection on the orthogonal basis Chebyshev polynomials, which requires computing integrals, see~\cite[Chapter 7]{Trefethen2019}. Searching for the actual best polynomial that minimizes the uniform error started from practical engineering considerations with the works on Poncelet and Chebyshev on parallelogram mechanisms, as explained in details in the historical texts~\cite{Goncharov2000,Steffens2006}. The best uniform approximation is also commonly called Chebyshev approximation. Very soon after appeared the celebrated optimality condition for Chebyshev approximation, called here the equioscillation theorem, to be pronounced with a charming Ukrainian accent: a degree $n$ polynomial $p$ is a best uniform approximation of a continuous function $f$ inside $[a,b]$ if and only if the approximation error $p-f$ has $n+2$ (global) extrema $a\leq x_0<\cdots<x_{n+1}\leq b$ with oscillating errors, so that extrema are minimizers and maximizers of the error with the same magnitude. This best uniform approximation is now known to be furthermore unique, and even strongly unique\footnote{Strong uniqueness is defined in Section~\ref{ss:strong-uniqueness}. It is a crucial property that has been thoroughly investigated in the context of uniform approximation, which is in particular related to the success of discretization-based algorithms~\cite[Section 12 page 74]{Kroo2010}.}. It seems to be agreed that the theorem was known by Chebyshev and finally proved by Borel, see again~\cite{Goncharov2000,Steffens2006}. The approximation error of the best degree $15$ polynomial uniform approximation of $\abs$ is shown in red in the left plot of Figure~\ref{fig:univariate-polynomial-approximtion}, where we see $17$ oscillating extrema. Such optimality conditions bring some understanding on the problem of uniform approximation, and have some strong impact on practical methods for computing them either formally (e.g., the Chebyshev polynomials $T_n(x)=\cos(n\arccos(x))=2^{n-1}x^n+\cdots$ is a degree $n$ polynomial, which equioscillates, since $n\arccos(x)$ ranges over $[0,n\pi]$ when $x\in[-1,1]$, see the right plot of Figure~\ref{fig:univariate-polynomial-approximtion} for $n=7$, hence the degree $n-1$ polynomial $p(x)=2^{n-1}x^n-T_n(x)$ is a best uniform approximation of $f(x)=2^{n-1}x^n$ with error $p(x)-f(x)=-T_n(x)$), or numerically (e.g., Remez' exchange algorithm and the simple two phase method both rely on the equioscillation theorem, see~\cite{Watson2000}).

We now turn our attention to the general case of approximating continuous real-valued functions defined in an arbitrary compact Hausdorff\footnote{Spaces that are compact but not Hausdorff are exotic spaces, where continuity and limit of sequences are not well defined.} set $X$. The domain $X$ is typically a compact subset of $\R^m$, and we speak about approximating a function of $m$ variables, and of multivariate approximation if $m\geq2$. The space of continuous functions $C(X)$ is endowed with the uniform norm, and we consider a finite dimensional subspace $\V$ of $C(X)$. Elements of $\V$ are called generalized polynomials. The best uniform approximation of $f\in C(X)$ by generalized polynomials from $\V$ is the projection of $f$ onto $\V$, i.e., a generalized polynomial with minimal uniform distance to $f$. Given a basis $\{\phi_1,\ldots,\phi_n\}\subseteq\V$ of $\V$, the best uniform approximation problem becomes the minimax optimization problem
\begin{equation}
	\min_{a\in\R^n} \ \max_{x\in X} \ \Bigl|\sum_{i=1}^na_i\phi_i(x)-f(x)\Bigr|.
\end{equation}
For example, approximation of two variables functions by polynomials of degree two can be performed with the basis $\{1,x_1,x_2,x_1x_2,x_1^2,x_2^2\}$. In this general framework, the Stone–Weierstrass theorem generalizes Weierstrass theorem, and proves in particular that multivariate polynomials are dense in $C(X)$. Jackson type theorems are known for multivariate approximation, relating smoothness of the multivariate function $f$ to be approximated to the asymptotic of $E_n(f)$, e.g., the work of Paul Montel~\cite{Montel1918} at the beginning of the twentieth century for box domains,~\cite{Newman1964} for some ball domains,~\cite{Ganzburg1981} for some more general convex domains, and~\cite{Plesniak2009,Totik2020} for more general non-convex domains. Actually building multivariate polynomial approximations is much more complicated than the univariate case. Distributions of nodes in domains of multivariate functions may lead to singular interpolation matrices, see~\cite{Neidinger2019}, while explicit generalizations of Chebyshev nodes are known only for the two variable square domains $X=[a,b]\times[c,d]$, so-called Fekete and Padua points, leading to uniform approximation error $O(E_f(n)\,\log(n)^2)$, see~\cite[Theorem 3]{Bos2006} and~\cite{Bos2010}. On the other hand, current computing capabilities of computers encourage building multivariate polynomial approximations using multivariate least square regression, which intrinsically removes the singularity issue of multivariate interpolation by using more nodes than necessary. It can produce good polynomial approximations but asks for efficient sampling strategies~\cite{Guo2020}.
%Finally, although not discussed in detailed in this survey, high degree univariate polynomial require efficient representations and evaluation algorithms, like Horner form, representation in the Chebyshev basis, Clenshaw evaluation algorithm.

While in the context of univariate approximation, the compromise between complexity of computation and accuracy favors the interpolation at Chebyshev nodes, multivariate best uniform approximation challenges other techniques in the multivariate case, where their complexities all sensibly increase. Multivariate best uniform approximation also comes with its own burden of complications. Even in the favorable situation of approximation with multivariate polynomials, uniqueness or strong uniqueness of best uniform approximation are not granted. In fact, for a fixed subspace $\V$ of dimension $n$, strong uniqueness of best uniform approximation independently of the function to be approximated is well-known to be equivalent to the so-called Haar condition~\cite[Theorem 3]{Newman1963}, that each non null element of $\V$ has at most $n-1$ zeros inside $X$.
%This condition is independent of the basis of $V$, a useful equivalent condition using a basis $\{\phi_1,\ldots,\phi_n\}\subseteq\V$ is that every square matrix with entries $\phi_i(x_j)$, $1\leq i,j\leq n$, called here Haar matrices, is nonsingular for any choice of different $x_1,\ldots,x_n\in X$. In the case of univariate polynomial approximation of degree less or equal to $n-1$, with $\phi_i(x)=x^{i-1}$, Haar matrices are Vandermonde matrices, and the Haar condition is satisfied, i.e., no non zero such polynomial has $n$ zeros, or equivalently, square Vandermonde matrices are non-singular.
Subspaces satisfying the Haar condition are often called Chebyshev sets. Univariate polynomials of degree $n$ form a $n+1$ dimensional Chebyshev set since they cannot have more than $n$ roots unless being identically zero.
However, Marhuber's theorem~\cite{Mairhuber1956} shows that if $X\subseteq \R^m$ satisfies the Haar condition then it is homeomorphic to a segment or a circle, that is, multivariate approximation cannot satisfy the Haar condition. See~\cite[Chapter 2]{Alimov2021} for more details. Non-strong uniqueness is actually common in multivariate approximation, such approximation problems being called singular by several authors~\cite{Osborne1969,Reemtsen1990,Watson2000} because the linear problems arising in discretization methods tends to become singular if strong uniqueness is lacking. It is also common that multivariate approximation problems have several solutions, hence infinitely many solutions since uniform approximation is convex by nature, the optimal solutions form a convex set. This was recently analyzed in~\cite{Roshchina2024}, where the authors proved that the degree $6$ bivariate polynomial $x_1^6 + x_2^6 + 3x_1^4x_2^2 + 3x_1^2x_2^4 + 6x_1x_2^2-2x_1^3$ has infinitely many best uniform quadratic approximations. Some multivariate Chebyshev polynomials were computed in~\cite{Dressler2024}, using a specific algorithm exploiting the moment-SOS, aka Lasserre hierarchy~\cite{Lasserre2001}, that allows handling non-uniqueness of the corresponding uniform approximation problems.

\begin{figure}[t!]
	\includegraphics[width=0.28\linewidth]{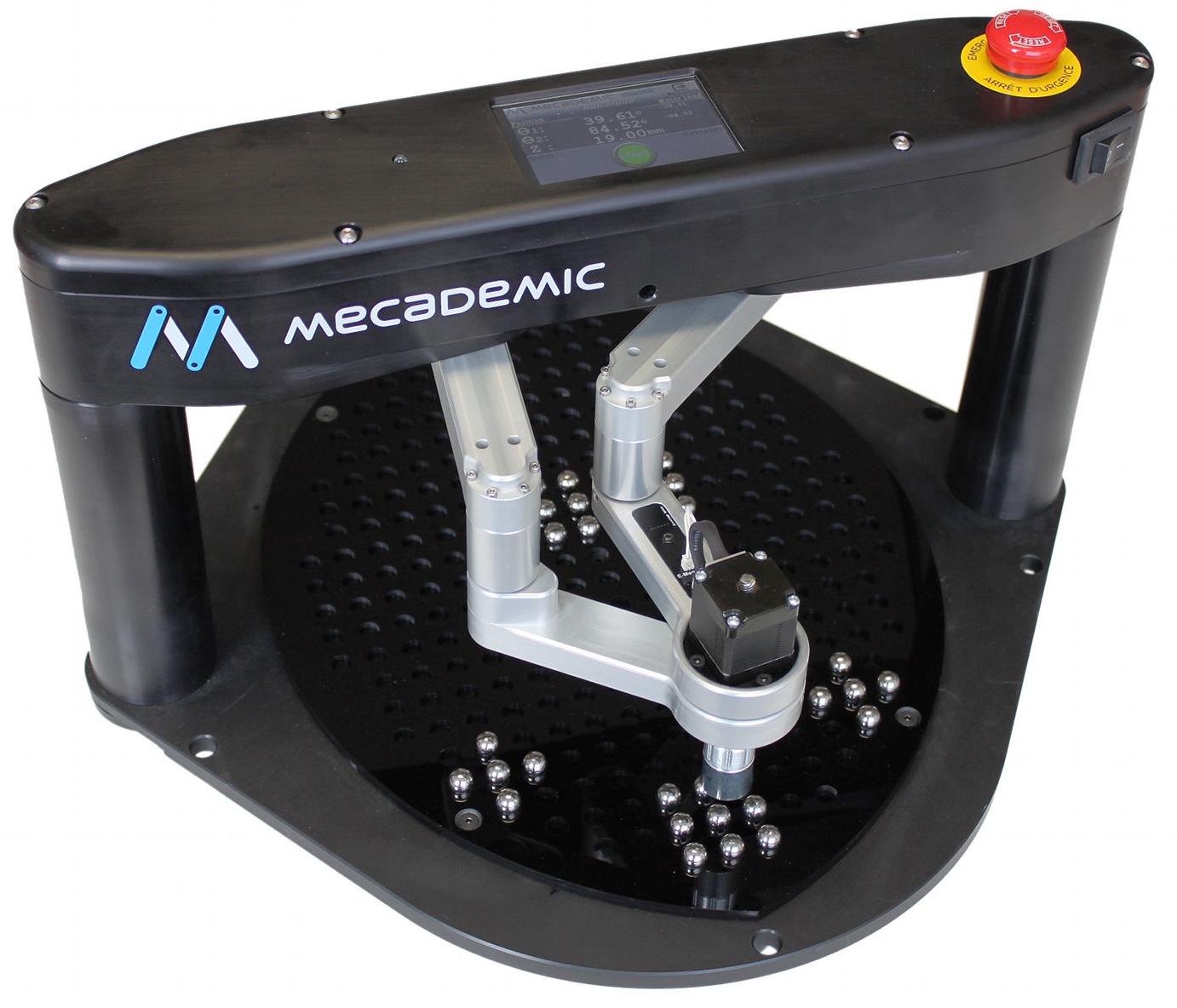}\hfill\includegraphics[width=0.34\linewidth]{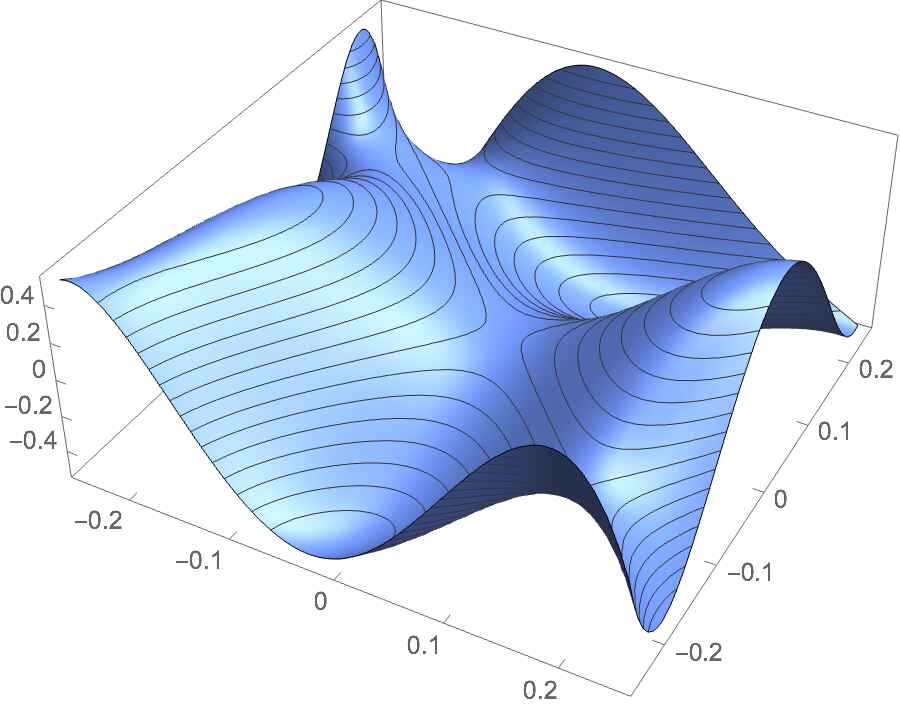}\hfill\includegraphics[width=0.34\linewidth]{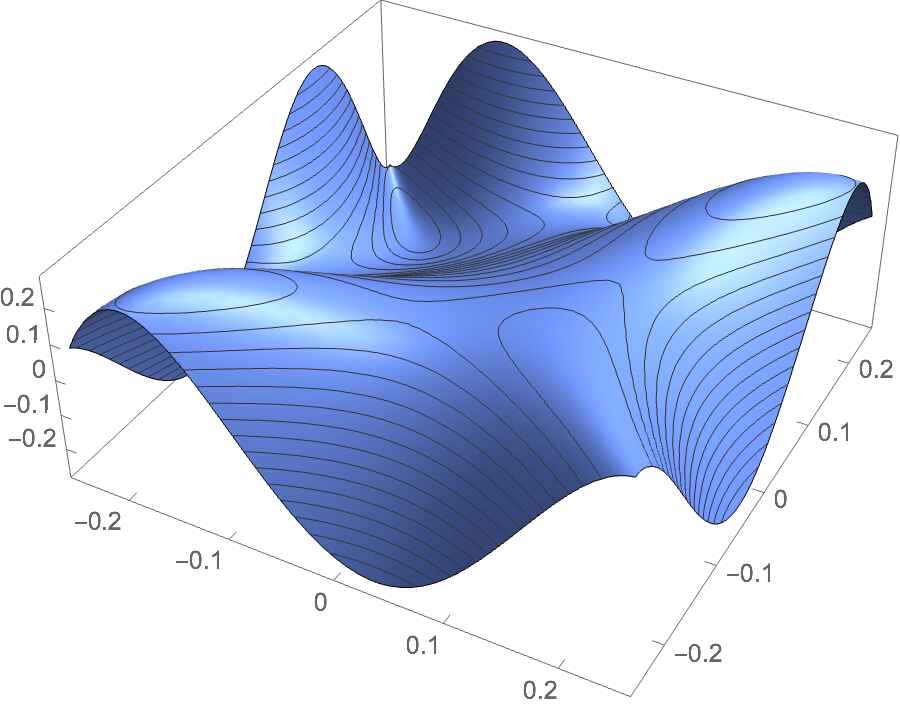}
	\caption{On the left, the DexTAR parallel robot, whose arm lengths are $90$ millimeters. Center and right plots show the error functions of best uniform degree three polynomial approximations of each coordinate $x_1$ and $x_2$ the position of the end effector for a range of actuated joint angles inside $[-0.25,0.25]\times[-0.25,0.25]$ (in radians), showing a maximal error of $0.53$ millimeter for $x_1$ and $0.28$ millimeter for $x_2$.\label{fig:DexTAR}}
\end{figure}
As a consequence, uniqueness and strong uniqueness have to be characterized depending on the function to be approximated, even in the favorable case of multivariate polynomial approximation, and optimality conditions for multivariate approximation usually have variants that include uniqueness and strong uniqueness. Best uniform multivariate approximations are also characterized through the extrema of their error, similarly to the equioscillation theorem, but multivariate error functions and their extrema are much richer. Figure~\ref{fig:DexTAR} shows the DexTAR parallel robot~\cite{Koessler2020}, and the two dimensional error functions of the best degree three polynomial approximations $p_1(\theta)$ and $p_2(\theta)$ of its direct geometric models $x_1(\theta)$ and $x_2(\theta)$, which maps actuated joint angular coordinates $\theta=(\theta_1,\theta_2)$ to each end-effector coordinates $x_1$ and $x_2$. Such polynomial approximations of direct geometric models are useful for example in the context of continuum parallel robots inside some singularity free workspace~\cite{Briot2022,Boyer2023}, where the evaluation of the direct geometric model requires the resolution  of a boundary valued problem, which needs to be solved at the frequency of the controller (usually one KHz). Comparing univariate and multivariate error functions in Figure~\ref{fig:univariate-polynomial-approximtion} and Figure~\ref{fig:DexTAR}, it is obvious that the characterization of multivariate best uniform approximations cannot be as simple as the equioscillation theorem. The first formulation of the optimality condition for best uniform multivariate approximation is due to Kirchberger~\cite{Kirchberger1902-PhD,Kirchberger1903} at the beginning of the twentieth century. Kirchberger's condition was sharpened and other conditions were later discovered. These conditions are presented in Section~\ref{s:classical-optimality-conditions}, where conditions applying to relative Chebyshev centers are presented homogeneously with conditions for uniform approximation (relative Chebyshev centers define worst case best uniform approximation for a set of target functions, instead of a single target function, somehow introducing some kind of robustness in uniform approximation, see Section~\ref{ss:def-center}). The focus given here to the approximation of real valued functions allows applying standard optimality condition of convex optimization to uniform approximation. They are presented in Section~\ref{s:convex-optimality-conditions}, together with basic techniques for subdifferential computation tailored to uniform approximation, homogenizing and providing new lights on classical optimality conditions for uniform approximation. Three numerical examples are presented in Section~\ref{s:numerical-applications}: the first shows the approximation of the 2D Runge function. The second shows the polynomial approximation of the direct geometric model of the DexTAR. The last shows how subdifferential calculus can help finding optimality conditions and numerical algorithms for a non-standard approximation problem that consists in minimizing the evaluation error with the uniform approximation error, here in the univariate case for simplicity.

%\subsection*{Todo}~
%
%\noindent Additional references:
%\begin{itemize}
%	\item Lasserre's group: Unions of intervals~\cite{Foucart2019};
%\end{itemize}

%%%%%%%%%%%%%%%%%%%%%%%%%%%%%%%%%%%%%%%%%%%%%%%%%%%%%%%%%%%%%%%%%%%%%%
%%%%%%%%%%%%%%%%%%%%%%%%%%%%%%%%%%%%%%%%%%%%%%%%%%%%%%%%%%%%%%%%%%%%%%

\section{Formal definition of the problems}

Chebyshev approximation problems and relative Chebyshev centers are defined in the following two subsections. The notion of strong uniqueness of an optimal generalized polynomial is presented in the last subsection.

\subsection{Chebyshev approximation problems}

The space $\C(X)$ of real-valued continuous functions defined on the compact Hausdorf set $X$ is endowed with the uniform norm $\|f\|_\infty=\max_{x\in X}|f(x)|$, where the domain of the function is implicitly $X$ throughout the paper. Typically, $X$ is a compact subset of $\R^m$, and we speak about approximating functions of $m$ variables. We consider a $n$-dimensional subspace $\V$ and $n$ basis functions $\phi_i:X\rightarrow\R$, which are stacked in the basis function vector $\phi:X\rightarrow\R^n$ defined by $\phi(x)=(\phi_1(x),\ldots,\phi_n(x))$. Elements $p(x)=\phi(x)^Ta=a_1\,\phi_1(x)+\cdots+a_n\,\phi_n(x)$ of $\V$, where $a\in\R^n$ are the coordinates of $p$ in the basis, are called generalized polynomials. 
\begin{remark}
Generalized polynomials $p\in\V$ are considered as functions with domain $X$, and the dimension of $\V$ may depend on this domain. As an extreme example, it is easy to see that with $X=\{0\}$ and $\phi(x)=(1,x,x^2,\ldots,x^n)$ the canonical basis of degree $n$ polynomials, the dimension of $\V$ is actually $1$. In this case, $\phi(x)^Ta$ is to be understood as a function from $\{0\}$ to $\R$, therefore $1$ and $1+a_1x+a_2x^2+\cdots+a_nx^n$ are two representations of the same function, and the dimension of $\V$ is one. The requirement that the dimension of $\V$ is equal to the the number of basis function is equivalent to requiring that there exists $x_1,\ldots,x_n\in X$ such that the matrix $\bigl(\phi(x_1)\cdots\phi(x_n)\bigr)\in\R^{n\times n}$ is nonsingular, so that two generalized polynomials are equal to each other if and only if their coordinates in the basis agree. It is also equivalent to $\|\phi(\cdot)^Ta\|_\infty$ is a norm on the space $\R^n$ of coefficients in the basis of $\V$. This is the typical situation, e.g., for the canonical polynomial basis $\phi(x)=(1,x,x^2,\ldots,x^n)$, the matrix $\bigl(\phi(x_1)\cdots\phi(x_{n+1})\bigr)$ is a Vandermonde matrix which is regular hence the requirement is met as soon as $X$ contains $n+1$ distinct elements. The failure of this requirement leads to non practical situations where the functions $\phi_i(x)$ span $\V$ but are not a basis.
\end{remark}
We also consider one continuous function $f:X\rightarrow \R$ that does not belong to $\V$. The uniform approximation problem, or minimax approximation problem, or Chebyshev approximation problem, consists in computing the generalized polynomials that are closest to $f$ for the uniform norm on $X$:
\begin{equation}\label{eq:chebyshev}
	\min_{p\in \V}\|p-f\|_{\infty}= \min_{a\in\R^n} \ \max_{x\in X} \ \bigl| \ \phi(x)^Ta-f(x) \ \bigr|.
\end{equation}
It is well known that this problem has at least one optimal solution. Since we assume $f\notin \V$, the above minimum is strictly positive.

The extremal points of a function $e:X\rightarrow \R$ are denoted by $\ext(e)=\{x\in X:|e(x)|=\|e\|_X\}$, i.e., the maximizers of $|e|$. The extremal points of the error function, i.e., $\ext(p-f)$, are of central importance for optimality conditions. A signature is a subset of $X\times\{-1,1\}$, which associates signs to some elements $X$. The support of the signature is its projection onto $X$. A signature is naturally associated to a non identically zero continuous function $e:X\rightarrow\R$, which maps its extremal points to their sign:
\begin{equation}\label{eq:def-sig}
\sig(e)=\bigl\{(x,s)\in X\!\times\!\{\pm1\}:e(x)=s\,\|e\|_{\infty}\bigr\}.
\end{equation}
The signature of the error function $\sig(p-f)$ associates the sign of the error to its extremal points. It is also of central importance for optimality conditions. The following example will be used throughout the next sections to illustrate optimality conditions.

\begin{figure}
	\centering
	\includegraphics[width=0.35\linewidth]{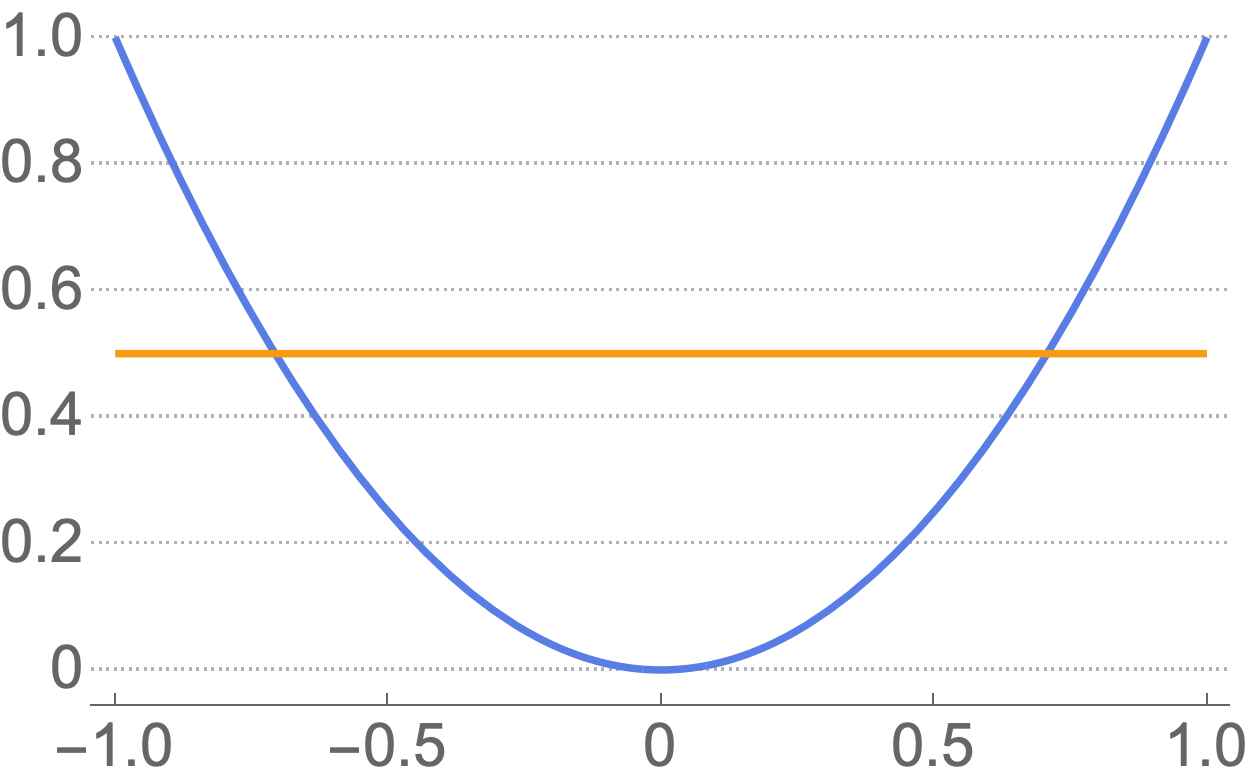}\hspace{1cm}\includegraphics[width=0.48\linewidth]{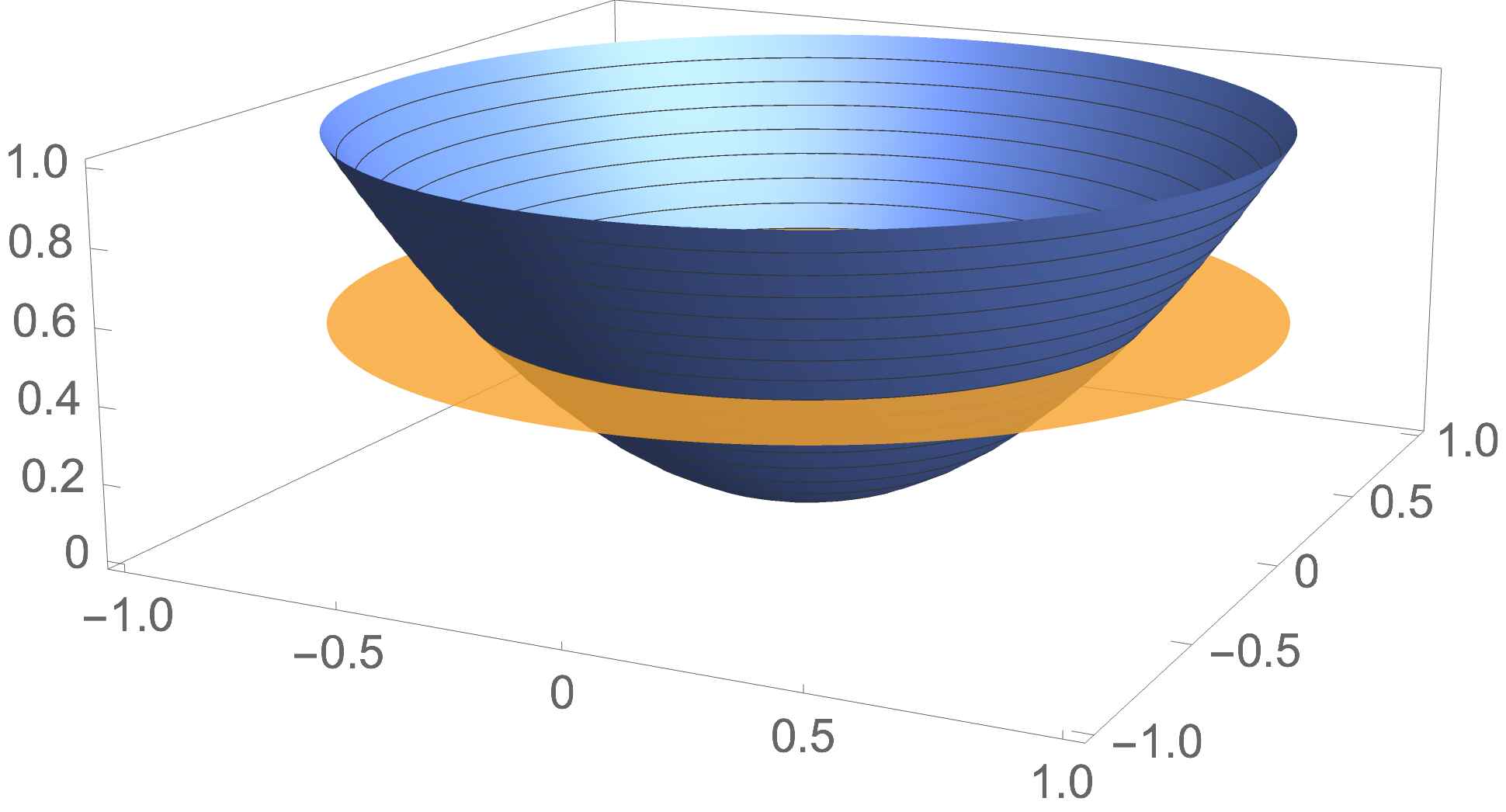}
	\caption{Left and right: the Chebyshev approximation problems of Example~\ref{ex:example-1} consisting in approximating the blue functions by an affine function. Best uniform affine approximations are represented in orange.\label{fig:example-1}}
\end{figure}
\newcommand{\ab}{a_{\one}}
\begin{example}\label{ex:example-1}
	We consider the problem of approximating the quadratic function $f(x)=x^Tx$ by an affine function $p(x)=\ab^T\,x+a_0$ uniformly inside $n$-ball $X=\B^{m}=\{x\in\R^m:x^Tx\leq1\}$, where for convenience $a=(a_0,a_1,\ldots,a_m)^T\in\R^{n}$, hence $n=m+1$, and $\ab=(a_1,\ldots,a_m)^T\in\R^{n-1}$. The optimal solution is $p(x)=\tfrac{1}{2}$, which is represented in orange in Figure~\ref{fig:example-1} for $n=1$ and $n=2$. The extremal points of the error $p-f$ are here the solutions of the Karush-Kuhn-Tucker system~\cite{bertsekas1997,Nocedal2006}:
    \begin{subequations}\label{eq:KKT}
    \begin{align}
        \ab-2x+2\mu x&=0
        \\ \mu(x^Tx-1)&=0
        \\ x^Tx&\leq1,
    \end{align}
    \end{subequations}
    where the absence of any sign restriction on the multiplier allows computing minimizers and maximizers of the error. After basic computations, the following extremal points are found: with $\ub\alpha=\max\{2,\|\ab\|\}$, $\lb\alpha=\min\{2,\|\ab\|\}$ and $\alpha_0=-\tfrac{1}{8}\lb\alpha^2+\tfrac{1}{2}\lb\alpha+\tfrac{1}{2}$ the extremal points are
    \begin{equation}\label{eq:example-1-extreme}
    	\ext(p-f)=\left\{\begin{array}{ll} \{\tfrac{-\ab}{\|\ab\|_2}\} & \text{if } a_0<\alpha_0 \\ \{\tfrac{\ab}{\ub\alpha}\} & \text{if } a_0>\alpha_0 \\  \{\tfrac{-\ab}{\|\ab\|_2}\}\cup\{\tfrac{\ab}{\ub\alpha}\} & \text{if } a_0=\alpha_0\end{array}\right.,
    \end{equation}
    where $\{\tfrac{-\ab}{\|\ab\|_2}\}$ is the $(m-1)$-sphere $\Sp^{{m-1}}$ if $\ab=0$. The extremal points $\{\tfrac{-\ab}{\|\ab\|_2}\}$ are minimizers of the error, including $\Sp^{{m-1}}$ when $\ab=0$, while the extremal point $\{\tfrac{\ab}{\ub\alpha}\}$ is a maximizer. For the optimal solution $\bp(x)=\tfrac{1}{2}$ we have $\ub\alpha=2$, $\lb\alpha=0$, $\alpha_0=\tfrac{1}{2}$, hence $a_0=\alpha_0$, and $\ext(\bp-f)=\Sp^{{m-1}}\cup\{0\}$, which can be seen on the two plots of Figure~\ref{fig:example-1}. The corresponding signature is $\Sigma(\bp-f)=(\Sp^{{m-1}}\!\times\!\{-1\})\cup\{(0,1)\}$.
\end{example}
In this example, the set of extremal points of the best approximation error is infinite. This is an atypical situation caused by some strong symmetry in the problem. Nevertheless, this is a good test case for optimality conditions presented in the next sections.

\subsection{Relative Chebyshev centers}\label{ss:def-center}

We follow here~\cite{Levis2023,Levis2024}. The Chebyshev center of a uniformly bounded set of target functions $\F\subseteq \C(X)$ relative to the subspace $\V$ is the generalized polynomial that minimizes the worst case distance to functions of $\F$. It is computed by solving
\begin{equation}\label{eq:crc-1}
	\inf_{p\in \V}\sup_{f\in \F}\|p-f\|_X.
\end{equation}
In the sequel, Chebyshev centers will be relative to $\V$ and simply called relative Chebyshev centers. This problem is also called simultaneous uniform approximation of the target functions inside $\F$. The classical Chebyshev approximation is the special case $\F=\{f\}$. The theory is restricted to so-called totally complete sets of target functions, for which both $f^-(x):=\inf_{f\in\F}f(x)=\min_{f\in\F}f(x)$ and $f^+(x):=\sup_{f\in\F}f(x)=\max_{f\in\F}f(x)$, and both functions are continuous. In this case, the relative Chebyshev center of $\F$ is the same as the one of $\{f^-,f^+\}$, see~\cite[Theorem 2.3]{Levis2023}. We assume this condition throughout the paper.
\begin{figure}
	\centering
	\includegraphics[width=0.32\linewidth]{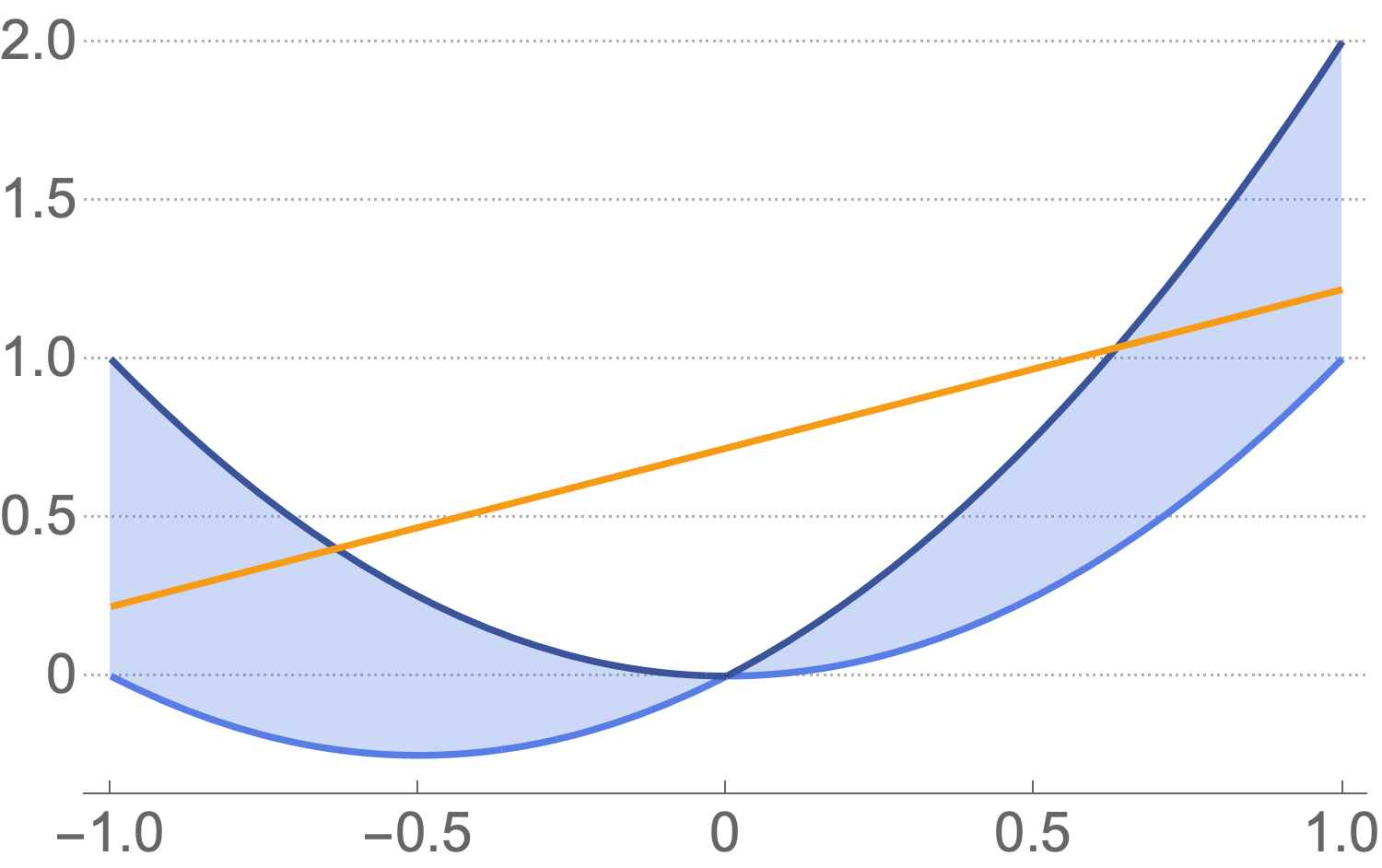}\hspace{0.2cm}\includegraphics[width=0.32\linewidth]{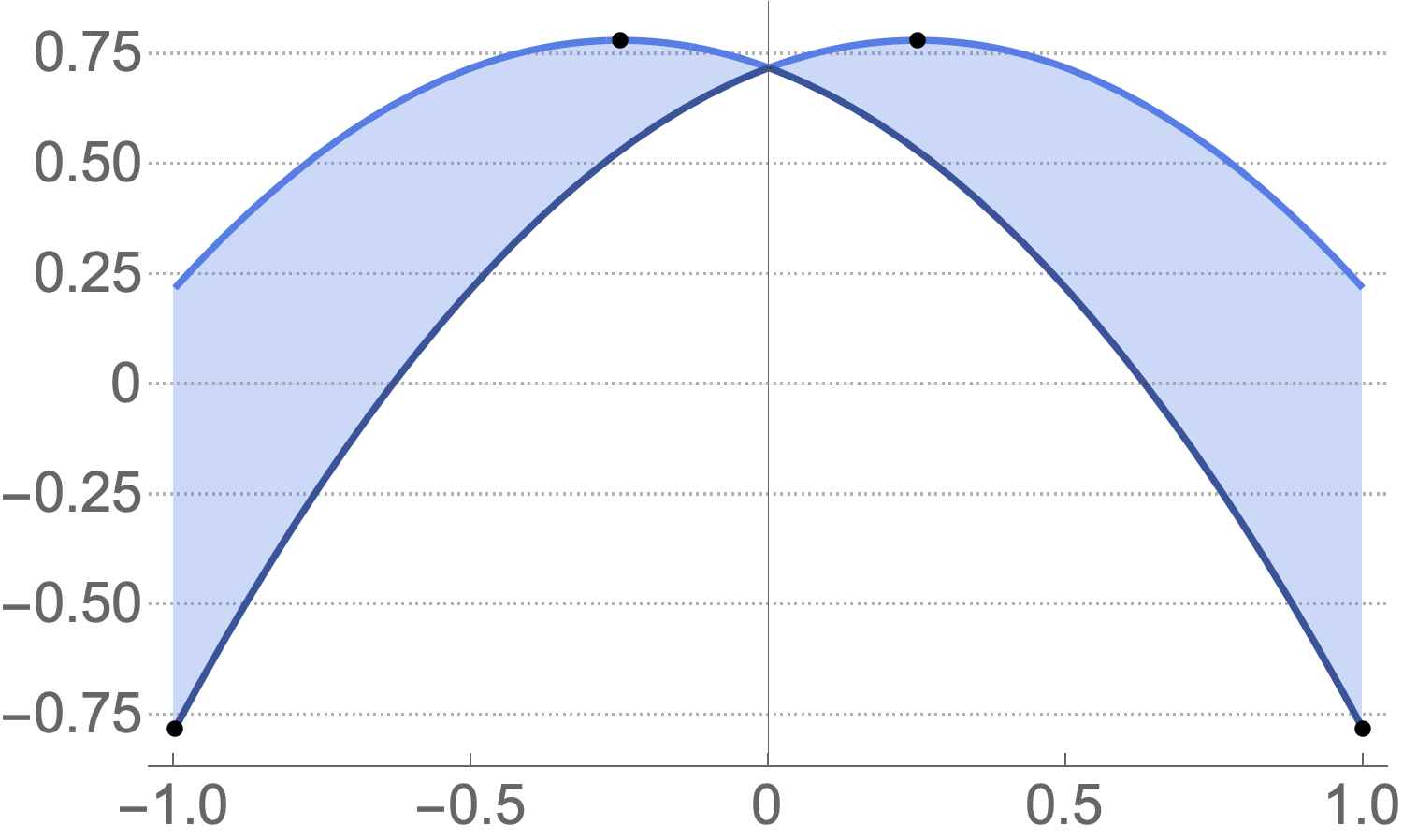}\hspace{0.2cm}\includegraphics[width=0.32\linewidth]{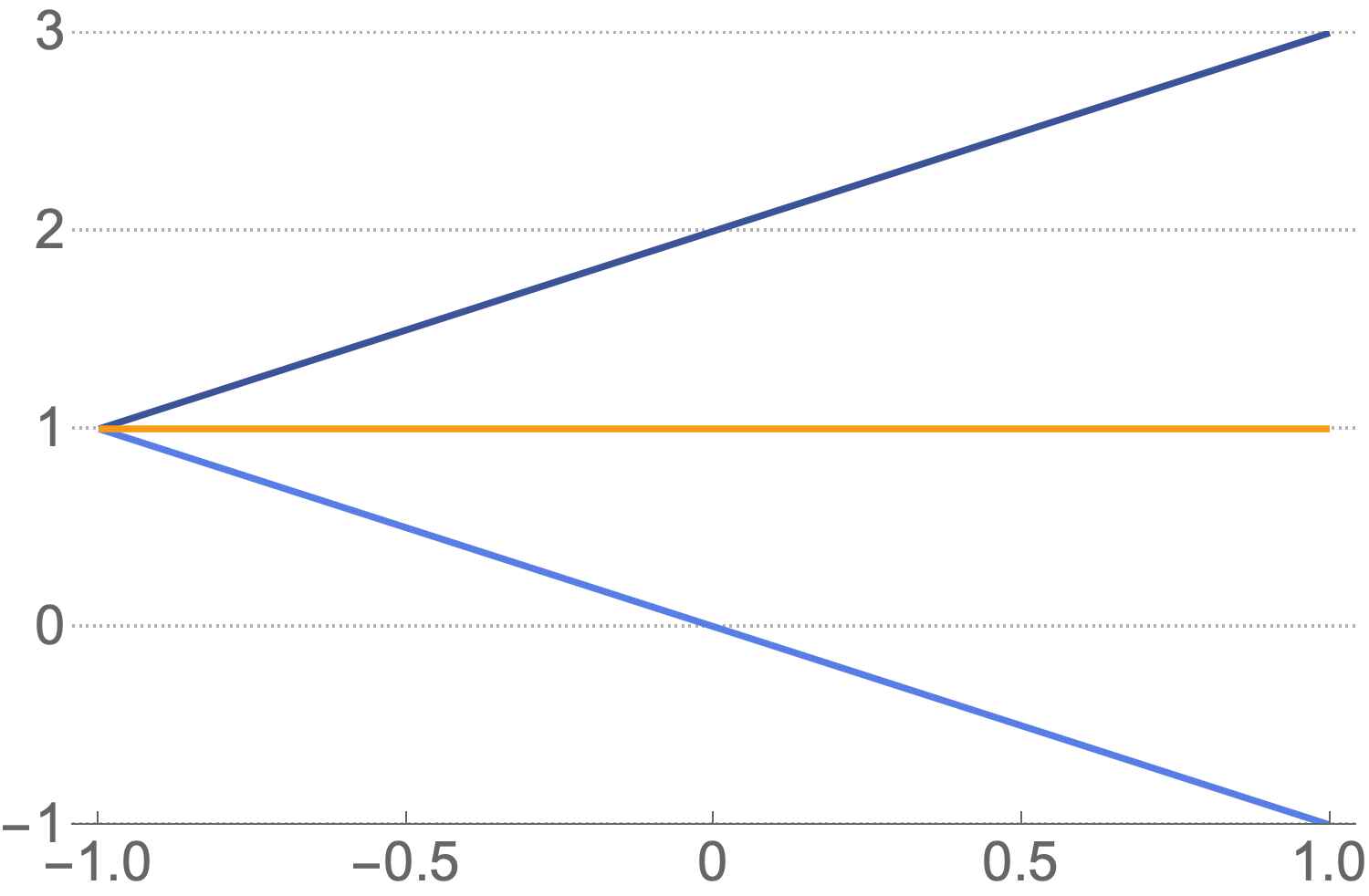}
	\caption{Left: the uncertain quadratic function of Example~\ref{ex:example-2-bis} and it best uniform affine approximation in orange. Right: the relative Chebyshev center problem of Example~\ref{ex:example-2}, consisting in approximating simultaneously the two blue functions by a constant function.\label{fig:example-2}}
\end{figure}
\begin{example}\label{ex:example-2-bis}
	We consider the uncertain quadratic function $f_{\delta}(x)=x^2+\delta x$ with $\delta\in[0,1]$, so $\F=\{f_\delta:\delta\in[0,1]\}$. We have $f^+(x)=x^2$ and $f^-(x)=x^2+x$ if $x\leq0$, and $f^+(x)=x^2+x$ and $f^-(x)=x^2$ otherwise. The set of functions is represented in light blue in the left plot of Figure~\ref{fig:example-2}, the functions $f^-(x)$ and $f^+(x)$ being represented in blue and darker blue respectively. We approximate it uniformly by affine functions $p(x)=a_1x+a_0$, with best uniform approximation $\bp(x)=\tfrac{1}{2}x+\tfrac{23}{32}$ and uniform error $\tfrac{25}{32}$, which is represented in orange in the left plot of Figure~\ref{fig:example-2}.
\end{example}

Aiming at an homogeneous presentation of optimality conditions, we define here set-valued error functions and their signatures. First of all, the set of target functions is interpreted as a set-valued function $\F:X\rightarrow 2^\R$ with $\F(x)=\{f(x):f\in\F\}$. Real operations are naturally extended to set-valued operations or to mixed real-valued and set-valued operations, e.g., $p-\F$ is the set-valued function $x\mapsto \{p(x)-f(x):f\in\F\}$. The uniform norm of the set-valued function $e:X\rightarrow2^\R$ is $\|e\|_\infty=\max\{|y|:y\in e(x),x\in X\}$, its extremal points being
\begin{equation}
\ext(e)=\{x\in X:e(x)\ni -\|e\|_\infty\text{ or }e(x)\ni \|e\|_\infty\}.
\end{equation}
With these definitions, the Chebyshev center of $\F\subseteq \C(X)$ relative to $\V$ is the uniform approximation of the set-valued function $X\ni x\mapsto \F(x)$.
\begin{example}\label{ex:example-2-ter}
Revisiting Example~\ref{ex:example-2-bis}, the uncertain quadratic function is now the set-valued function $x\mapsto\F(x)=\{f_\delta(x):\delta\in[0,1]\}=[f^-(x),f^+(x)]$, which is an interval in this case. The optimal approximation set-valued error $\bp(x)-\F(x)$ is represented in the center plot of Figure~\ref{fig:example-2}, together with its four extremal points.
\end{example}

Some non-emptiness, compactness and hemicontinuity requirements would typically be introduced here for characterizing the set-valued function $\F$, aiming at applying Berge Maximum Theorem, but the total completeness assumption simplifies the treatment of the set-valued functions involved here: we say that the set-valued function $\F$ is totally complete if it comes from a totally complete set of functions, in which case we have the following properties
\begin{subequations}
\begin{align}
\conv\F(x)&=[f^-(x),f^+(x)]
%\\\max\{|y|:y\in\F(x)\}&=\max\{|f^-(x)|,|f^+(x)|\}
\\\|\F\|_\infty&=
%\sup_{x\in X}\{|y|:y\in \F(x)\}=
\max\{\|f^-\|_\infty,\|f^+\|_\infty\}
\\
\|p-\F\|_\infty&=
%\sup_{\begin{subarray}{c}x\in X\\f\in\F\end{subarray}}|p-f(x)|=
\max\{\|p-f^-\|_\infty,\|p-f^+\|_\infty\}.\label{eq:prop-set-valued-signature}
\end{align}
\end{subequations}
With these definitions, the relative Chebyshev center problem consists is solving $\min_{p\in \V}\|p-\F\|_\infty$. The definition of signatures is naturally extended to the set-valued function $e:X\rightarrow2^\R$ as follows:
\begin{equation}\label{eq:sig-rcc}
	\sig(e)=\{(x,s)\in X\!\times\!\{\pm1\}:e(x)\ni s\,\|e\|_{\infty}\bigr\}.
\end{equation}
The definition~\eqref{eq:sig-rcc} matches the definition for real-valued function if we identify the real valued function $e:X\rightarrow\R$ with the set-valued function $x\mapsto \{e(x)\}$.
\begin{example}
As seen on the center plot of Figure~\ref{fig:example-2}, the signature of the set-valued error function of Example~\ref{ex:example-2-ter} is $\sig(p-\F)=\{(-1,-1),(-\tfrac{1}{4},1),(\tfrac{1}{4},1),(1,-1)\}$.
\end{example}
As we will see, the signature $\sig(p-\F)$ will be of central importance in optimality conditions of relative Chebyshev centers: it will allow for a convenient expression of Tanimoto's corrected condition in Section~\ref{sss:Tanimoto}, and an easy derivation of the associated subdifferential in Section~\ref{sss:subdifferentitial-center}. Using~\eqref{eq:prop-set-valued-signature}, the signature~\eqref{eq:sig-rcc} can be expressed directly using the functions $f^-$ and $f^+$:
\begin{subequations}
\begin{align}
%	\sig(p-\F)&=\bigl\{ \ (x,s)\in X\!\times\!\{\pm1\}:p(x)-\F(x)\ni s\max\{\|p-f^-\|_\infty,\|p-f^+\|_\infty\} \ \bigr\}\\%
\begin{split}
	\sig(p-\F)=\bigl\{ & (x,s)\in X\!\times\!\{\pm1\}:\\&p(x)-\F(x)\ni s\max\{\|p-f^-\|_\infty,\|p-f^+\|_\infty\} \bigr\}
\end{split}
\\
\begin{split}
	=\bigl\{& (x,s)\in X\!\times\!\{\pm1\}:\\&\{p(x)-\!f^-(x),p(x)-\!f^+(x)\}\!\ni\! s\max\{\|p-\!f^-\|_\infty,\|p-\!f^+\|_\infty\} \bigr\}
\end{split}
\end{align}
\end{subequations}
The following more explicit case by case expression is a simple consequence of the previous one, and will be useful in the sequel:
\begin{equation}\label{eq:sig-center-explicit}
\sig(p-\F)=\left\{\begin{array}{cl}
\sig(p-f^-)&\text{if }\|p-f^-\|_\infty>\|p-f^+\|_\infty
\\\sig(p-f^+)&\text{if }\|p-f^-\|_\infty<\|p-f^+\|_\infty
\\\sig(p-f^-)\cup\sig(p-f^+)&\text{otherwise.}
\end{array}\right.
\end{equation}
Informally, if $\|p-f^-\|_\infty>\|p-f^+\|_\infty$ then the extremal points in the signature $\sig(p-\F)$ appear only with respect to $f^-$. Although not shown explicitly in~\eqref{eq:sig-center-explicit} for clarity, in this case the sign of the error needs to be positive, i.e., $\sig(p-f^-)=\ext(p-f^-)\!\times\!\{1\}$.
%However, extremal points of $f^+$ can correspond to positive or negative signs but are not in the signature.
When $\|p-f^-\|_\infty<\|p-f^+\|_\infty$, the reverse situation applies to $f^+$ and $\sig(p-f^+)=\ext(p-f^+)\!\times\!\{-1\}$. Finally, when $\|p-f^-\|_\infty=\|p-f^+\|_\infty$ extremal points come from both $f^-$ and $f^+$.

The following example is taken from~\cite{Levis2024} and is a counter example to the wrong original condition of Tanimoto~\cite[Theorem 2.4]{Levis2023} and to the strong uniqueness condition initially presented by Levis et al. in~\cite[Theorem 4.5]{Levis2023} (which were both corrected by Levis et al. in~\cite{Levis2024}). It will also be used in the next section to illustrate optimality conditions.
\begin{example}\label{ex:example-2}
	We consider the problem of approximating simultaneously the two univariate affine functions $f_1(x)=-x$ and $f_2(x)=x+2$ by constant functions $p(x)=a_0$ uniformly inside $X=[-1,1]$. The optimal solution is $p(x)=1$, as can be seen on the right plot of Figure~\ref{fig:example-2} where the two functions $f_1(x)$ and $f_2(x)$ are represented together with the optimal approximation. The uniform norm of the error function $p-\F$ is
	\begin{equation}\label{eq:ex:example-2}
		\|p-\F\|_\infty=\left\{\begin{array}{ll}3-a_0&\text{if }a_0\leq1 \\ a_0+1&\text{if }a_0\geq1 \end{array}\right.=|a_0-1|+2.
	\end{equation}
	One can also see on the figure that $\ext(e_a)=\{1\}$ independently of $a_0$, the extremal point being attained for $f_2=f^+$ if $a_0\leq1$ and for $f_1=f^-$ if $a_0\geq1$. So the signature is
	\begin{equation}
		\sig(p-\F)=\left\{\begin{array}{cl}
		\{(1,-1)\}&\text{if }a_0<1
		\\\{(1,1)\}&\text{if }a_0>1
		\\\{(1,-1),(1,-1)\}&\text{if }a_0=1
		\end{array}\right..
	\end{equation}
\end{example}

In this example, it is worth noting that when $a_0=1$ the signature contains two elements with opposite signs for the same extremal point. This situation does not happens in the case of Chebyshev approximation of one target function, where a given extremal point is associate to the sign of the error of a real valued function, and is either a minimizer or a maximizer but not both. This is related to the error in the original statement of Tanimoto's condition for optimality~\cite{Tanimoto1998} and the original statement of Levis et al. condition for strong uniqueness~\cite{Levis2023}, which were both corrected by Levis et al. in~\cite{Levis2024}.

\subsection{Strong uniqueness of optimal generalized polynomials}\label{ss:strong-uniqueness}

\begin{figure}
	\centering
	\includegraphics[width=0.42\linewidth]{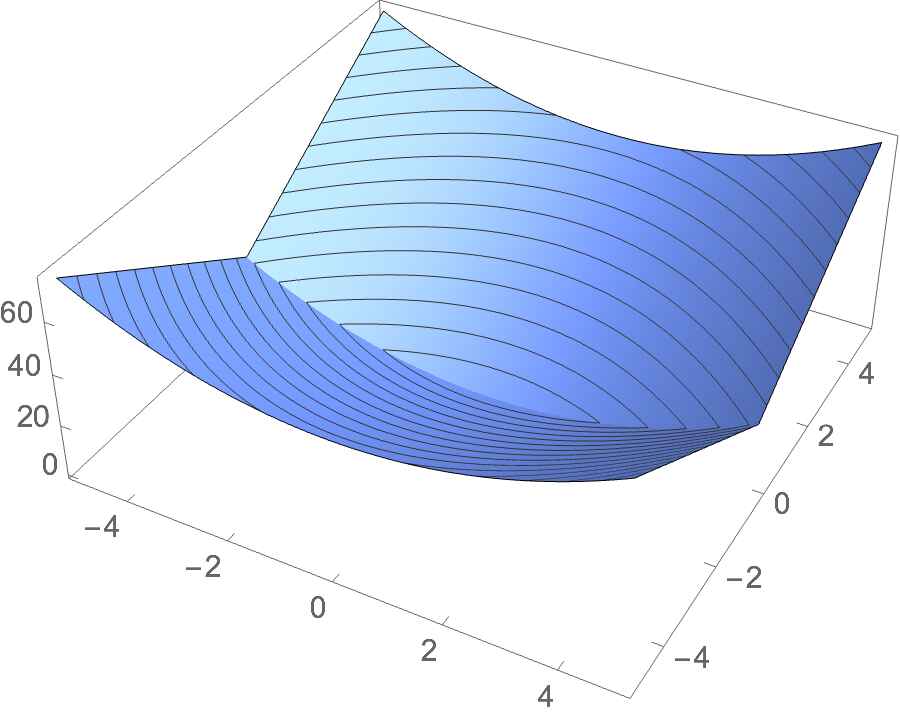}\hspace{1cm}\includegraphics[width=0.42\linewidth]{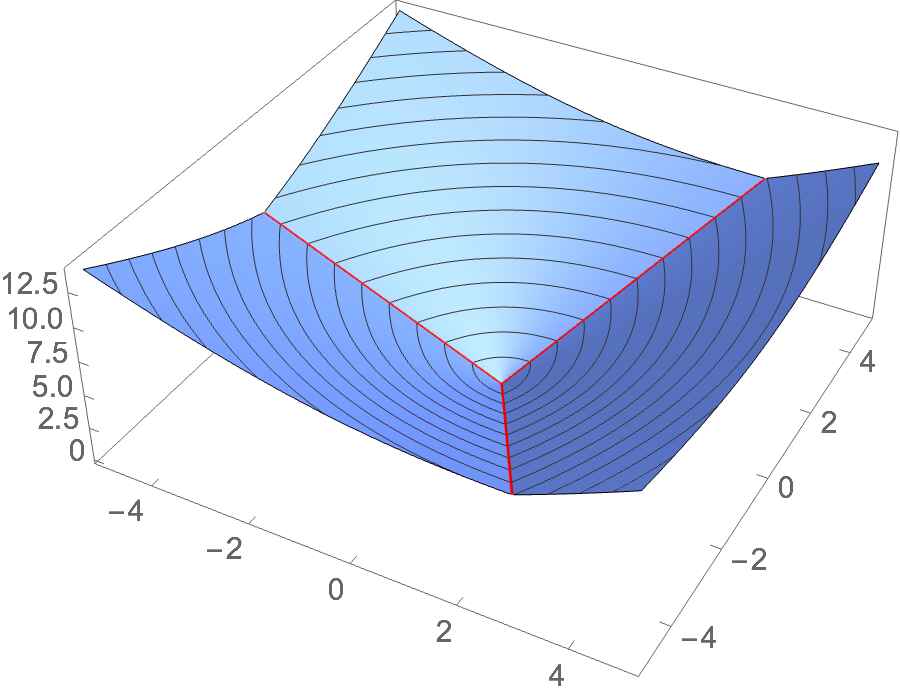}
	\caption{Convex functions with typical unique but non strongly unique minimizer on the left ($f(a)=a_1^2+10|a_2|$), and strongly unique minimizer on the right ($f(a)=a_1^2+a_2^2+\max\{g_0^Ta,g_1^Ta,g_2^Ta\}$ with $g_k=(\cos\tfrac{2k\pi}{3},\sin\tfrac{2k\pi}{3})^T$ for $k\in\{0,1,2\}$).\label{fig:non-strong-uniqueness}}
\end{figure}
A generalized polynomial $\bp$ is a strongly unique best uniform approximation if and only if there exists $r>0$ such that
\begin{equation}\label{eq:strong-definition}
	\forall p\in\V \ , \ \|p-f\|_\infty\geq\|\bp-f\|_\infty+r\|\bp-p\|_\infty.
\end{equation}
With $r=0$ this conditions reduces to simple optimality. Strong uniqueness obviously implies uniqueness. Figure~\ref{fig:non-strong-uniqueness} shows two typical objective functions, representing in an abstract way non strongly unique and strongly unique minimizers\footnote{In fact, minimizers represented in Figure~\ref{fig:non-strong-uniqueness} are respectively non sharp and sharp minimizers of the corresponding convex objective functions, as defined in Section~\ref{sss:sharp}, which exactly corresponds to strongness of minimizers.}.
Strong uniqueness is difficult to check directly, even in the simple case of Example~\ref{ex:example-1}, whose best approximation is actually strongly unique, as proved using Smarzewski's condition in Example~\ref{ex:example-1-Smarzewski}. Strong uniqueness is defined similarly for relative Chebyshev centers. The framework of set-valued functions presented in the previous section offers a definition that mirrors~\eqref{eq:strong-definition}: the generalized polynomial $\bp$ is a strongly unique relative Chebyshev center if and only if there exists $r>0$ such that
\begin{equation}
	\forall p\in\V \ , \ \|p-\F\|_\infty\geq\|\bp-\F\|_\infty+r\|\bp-p\|_\infty.
\end{equation}
As previously, it matches the definition for uniform approximation with $\F=\{f\}$. The best approximation of Example~\ref{ex:example-2} is strongly unique, as shown by~\eqref{eq:ex:example-2}.

%\blue{A good minimax example, though not best uniform problem, illustrating the difference between unique and strongly unique is the drop example.}

As explained in the introduction, strong uniqueness has a practical impact on discretization based methods. E.g., numerical experiments carried out in~\cite{Reemtsen1990} tend to show that non-strong uniqueness is common in multivariate approximation, and requires algorithmic treatment. On the other hand, it is worth mentioning that functions having a strongly unique best approximation are dense in the set of functions having a unique minimizer~\cite[Theorem~3.5]{Nurnberger1982}.
%Furthermore, under mild additional hypotheses on the basis functions and when $X$ is a compact metric space, functions having a unique minimizer are dense in the set continuous functions~\cite[Corollary~3.10]{Nurnberger1982}.

%%%%%%%%%%%%%%%%%%%%%%%%%%%%%%%%%%%%%%%%%%%%%%%%%%%%%%%%%%%%%%%%%%%%%%
%%%%%%%%%%%%%%%%%%%%%%%%%%%%%%%%%%%%%%%%%%%%%%%%%%%%%%%%%%%%%%%%%%%%%%

\section{Optimality conditions for multivariate Chebyshev approximation}\label{s:classical-optimality-conditions}

Similarly to the equioscillation theorem, optimality conditions of multivariate Chebyshev approximation problems rely on the extremal points of the error $p-f$, the set of these extremal points being $\ext(p-f)=\{x\in X:|p(x)-f(x)|=\|p-f\|_\infty\}$ as defined in the previous section. Optimality conditions belong to two classes: the first class, presented in Section~\ref{ss:signature-based}, relies on signatures. These optimality conditions use only the information of the signature, i.e., the error extrema and their sign, to characterize the best approximation. Chebyshev's equioscillation theorem is a typical signature based condition in the case of univariate approximation. For multivariate approximations, the main signature based conditions are  Kolmogorov criterion and Rivlin and Shapiro's annihilating measure condition. The second class relies on basis function vectors evaluated at extremal points and is presented in Section~\ref{ss:basis-function-conditions}, including the kernel condition of Kirchberger, the zero in the convex hull condition of Cheney, and Smarzewski's condition for strong minimality.

\begin{remark}
Some conditions apply more generally to approximation problems in complex variables, namely Kolmogorov criterion, Bartelt's condition, Rivlin and Shapiro's annihilating measure condition and Cheney's convex hull condition, but we present here their real counterparts only.
\end{remark}

\subsection{Conditions based on signatures}\label{ss:signature-based}

The signature of the error has a critical importance because it contains enough information to characterize exactly the optimality of Chebyshev approximation problems. The criteria that allows deciding if a signature corresponds to an optimal solution of the associated Chebyshev problem are of great theoretical and practical importance. For example, in the context of univariate Chebyshev optimization, the equioscillation theorem says that signatures containing $n+1$ extremal points associated to oscillating signs correspond to optimal solutions.
The two main criteria for such signatures in the context of multivariate Chebyshev approximation are Kolmogorov criterion and Rivlin and Shapiro's annihilating measures, presented in the next two subsections. Rivlin and Shapiro have proved that both criteria are equivalent~\cite[Remark 1 and Remark 2 page 677]{Rivlin1961}, hence we can use Rivlin and Shapiro's extremal signature denomination for a signature that satisfies one of them, and therefore correspond to the error of a best uniform approximation. The case of approximation by affine functions enjoys a simple signature based optimality condition, which is presented in the third subsection. Finally, Tanimoto's corrected condition is presented in the last subsection, where the framework of set-valued error functions allows emphasizing its similarity with Rivlin and Shapiro's annihilating measure condition. Bartelt's condition and Levis et al.'s condition characterize strongly unique best approximations, extending Kolmogorov's criterion and Tanimoto's condition respectively, and are presented in the corresponding subsections.

\subsubsection{Kolmogorov criterion and Bartelt's condition}
The most fundamental optimality condition is the so-called Kolmogorov criterion~\cite{Kolmogorov1948}. Informally, a generalized polynomial $\bp$ is the best uniform approximation $f$ if and only there exists no generalized interval whose sign matchs the sign of the error at the extremal points.
\begin{theorem}[Kolmogorov criterion for real-valued uniform approximation]\label{thm:Kolmogorov}
The generalized polynomial $\bp$ is the best uniform approximation of $f$ if and only if% there exists no generalized polynomial $p\in\V$ satisfying
\begin{equation}\label{eq:synchronized}
\neg\Bigl( \ \exists p\in\V \ , \ \forall x\in\ext(\bp-f) \ , \ (\bp(x)-f(x))\, p(x)>0 \ \Bigr).
\end{equation}
\end{theorem}
Many authors use the statement~\eqref{eq:synchronized} of Kolmogorov criteria, e.g.: Rice~\cite[page 446]{Rice1963} says that extremal points satisfying this statement are not isolable; Shapiro~\cite[Lemma 2.2.1 page 7]{Shapiro1971} proves Kolmogorov's criterion; Powel proves the same criterion in~\cite[Equation (7.6) page 74]{Powel1981} and uses it to prove Chebyshev's alternation theorem. An equivalent statement of Kolmogorov criteria is that
\begin{equation}\label{eq:Kolmogorov-min}
\forall p\in\V \ , \ \Bigl(\min_{x\in\ext(\bp-f)}\bigl(\,\bp(x)-f(x)\bigr)\,p(x)\Bigr) \ \leq \ 0,
\end{equation}
used e.g. in~\cite{Brosowski1983}. To see the equivalence, just expand the nonpositive in~\eqref{eq:synchronized} and note that the minimum of some values is nonpositive if and only if there exists a negative such value. It is also equivalent to
\begin{equation}\label{eq:Kolmogorov-max}
\forall p\in\V \ , \ \Bigl(\max_{x\in\ext(\bp-f)}\bigl(\,\bp(x)-f(x)\bigr)\,p(x)\Bigr) \ \geq \ 0,
\end{equation}
used e.g. in~\cite{Bartelt1974}.
\begin{remark}\label{rem:signture-vs-error}
In~\eqref{eq:synchronized},~\eqref{eq:Kolmogorov-min} and~\eqref{eq:Kolmogorov-max} we can change $f(x)-\bp(x)$ to $\bp(x)-f(x)$ because $p$ and $-p$ belong to $\V$. Also, in our context where we approximate real-valued functions, we can replace $f(x)-\bp(x)$ by $\sign(f(x)-\bp(x))$, e.g.,~\eqref{eq:synchronized} then becomes the signature condition
\begin{equation}\label{eq:synchronized-sign}
\neg\Bigl( \ \exists p\in\V,\forall (x,s)\in\sig(\bp-f) \ , \ s\, p(x)>0 \ \Bigr).
\end{equation}
and the condition~\eqref{eq:Kolmogorov-min} becomes
\begin{equation}
\forall p\in\V \ , \ \Bigl(\min_{(x,s)\in\sig(\bp-f)}s\,p(x)\Bigr) \ \leq \ 0.
\end{equation}
In the context of complex-valued approximation, the characterization~\eqref{eq:synchronized} is not applicable but characterizations~\eqref{eq:Kolmogorov-min} and~\eqref{eq:Kolmogorov-max} can be adapted.
\end{remark}

\begin{example}[Application to Example~\ref{ex:example-1}]
	For the uniform approximation of $f(x)=x^Tx$ inside $X=\B^m$, the optimal polynomial is $p(x)=\tfrac{1}{2}$ and its error signature is $\Sigma(\bp-f)=(\Sp^{{m-1}}\!\times\!\{-1\})\cup\{(0,1)\}$. Here, generalized polynomials $p(x)$ are affine functions, for which the image of a convex combination of some vectors is the combination of the images of these vectors, i.e., with provided that $\sum\lambda_i=1$,
	\begin{equation}\label{eq:Kolmogorov-convex-combination}
	\forall p\in\V \ , \ p\bigl(\sum\lambda_ix_i\bigr)=\sum \lambda_ip(x_i).
	\end{equation}
	As a consequence, if a generalized polynomial is negative on $\Sp^{{m-1}}$ then it is also negative in the convex hull of $\Sp^{{m-1}}$ and therefore cannot be positive at $0$. Eventually, no generalized polynomial (here affine function) can match the sign of the signature, and Kolmogorov criterion proves $\bp$ is optimal.
	For any other affine function, in view of the expression~\eqref{eq:example-1-extreme} of the extremal points, there are either a single extremal point (an error minimizer or an error maximizer), or two extremal points (one being an error minimizer the other one an error maximizer), or a sphere of minimizers. In all cases, we can find a affine function with the same sign as the error on one or both extremal points, hence Kolmogorov criterion proves the non optimality.
\end{example}

Identifying signatures patterns preventing any generalized polynomial to match their sign allows deriving practical optimality conditions. For example, in the special case of univariate polynomial approximation, signature that prevent sign matching with any polynomial of degree $n-1$ are exactly those containing $n+1$ points with alternating sign (a degree $n-1$ polynomial with $n+1$ alternating signs has $n$ zeros and is therefore null), leading to Chebyshev equioscillation theorem. This argument extends readily to univariate generalized polynomials satisfying the Haar condition. In the case of multivariate approximation, identifying signature patterns preventing sign matching with generalized polynomials is far more difficult. A noticeable exception is the approximation by affine functions, which enjoy~\eqref{eq:Kolmogorov-convex-combination}, the corresponding optimality conditions being presented in Section~\ref{sss:affine}. More generally, Shapiro~\cite{Shapiro1967} calls patterns preventing sign match with any polynomials Chebyshev patterns. He uses the far reaching Euler-Jacobi formula to deduce some signature patterns preventing sign matching. E.g., with $X\subseteq\R^2$ a signature having $2n$ points on an ellipse with alternating signs when traveling around the ellipse does prevent any polynomial degree $n-1$ to match the sign pattern. Gearhart~\cite{Gearhart1973} deduces from this condition the optimal approximation over the unit disk of two variables monomials $x_1^nx_2^m$ by polynomials of degree $n+m-1$, extending to multivariate the ways and means of Chebyshev polynomials.

\paragraph{Strong uniqueness}

Kolmogorov criterion has been extended to offer a necessary and sufficient condition for strong uniqueness by Bartelt. His condition~\cite[Theorem 5]{Bartelt1973} is striking by the simplicity of its expression and its precision:
\begin{theorem}[Bartelt's condition]\label{thm:Bartelt}
The generalized polynomial $\bp$ satisfies the strong uniqueness condition~\eqref{eq:strong-definition}
\begin{equation}
	\forall p\in\V \ , \ \|p-f\|_\infty\geq\|\bp-f\|_\infty+r\|\bp-p\|_\infty.
\end{equation}
if and only if
\begin{equation}\label{eq:Bartelt}
\forall p\in\V \ , \ \Bigl(\max_{x\in\ext(\bp-f)}\bigl(\,\bp(x)-f(x)\bigr)\,p(x)\Bigr) \ \geq \ r\,\|\bp-f\|_\infty\,\|p\|_\infty,
\end{equation}
\end{theorem}
The case $r=0$ is exactly Kolmogorov criterion. As for Kolmogorov criterion, Bartelt's condition~\eqref{eq:Bartelt} can be expressed equivalently using signatures in our case of approximation of real functions:
\begin{equation}\label{eq:Bartelt-sig}
\forall p\in\V \ , \ \max\{s\,p(x):(x,s)\in\sig(\bp-f)\} \ \geq \ r\,\|p\|_\infty,
\end{equation}
Although it extends Kolmogorov's criterion~\eqref{eq:Kolmogorov-max} in a very natural way, Bartelt's condition is difficult to use in practice, even in the simple case of Example~\ref{ex:example-1}.

\subsubsection{Rivlin and Shapiro's annihilating measure condition}

Rivlin and Shapiro~\cite[Theorem 2 page 678]{Rivlin1961} (see also the lecture notes~\cite[Main Theorem page 14]{Shapiro1971} and~\cite[Theorem 2.6]{Rivlin2020}) have proposed another characterization of optimality based on signatures:
\begin{theorem}[Rivlin and Shapiro annihilating measure condition]
The generalized polynomial $\bp$ is a best uniform approximation of $f$ inside $X$ if and only if there exists a finite signature $\{(x_1,s_1),\ldots,(x_k,s_k)\}\subseteq\sig(\bp-f)$ and constants $c_i>0$ such that
\begin{equation}\label{eq:annihilate}
	\forall p\in\V \ , \ \sum_{i=1}^k \ c_i \, s_i \, p(x_i) = 0.
\end{equation}
The integer $k$ can be chosen less or equal to $n+1$.
\end{theorem}
It is convenient to interpret the $c_i$ as a signed measure whose support is a finite subset of extremal points:
\begin{equation}
\mu(f)=\sum c_i\,s_i\,f(x_i),
\end{equation}
whose support and sign correspond to a subsignature of the error. A signed measure satisfying~\eqref{eq:annihilate} for all generalized polynomials is said to annihilate the subspace of generalized polynomials\footnote{Rivlin and Shapiro present the annihilating measure condition as a "measure-theoretic language convenient shorthand" to their zero in the convex hull condition given in~\cite[Theorem 1 page 672]{Rivlin1961}. This measure-theoretic language also makes appear the connection to standard results in functional analysis, where annihilating linear functionals of Banach space of continuous functions defined in a compact Hausdorff set (which are signed measures by Riesz representation theorem) are used in conjunction with Hahn-Banach extension theorem to characterize points closest to possibly infinite dimensional subspaces, see, e.g.,~\cite[Corollary~3.3 page~520]{Deutsch1967},~\cite[Remark 2.3.5 page 17 and Section 5]{Shapiro1971}, and~\cite[Section 2.1]{Dressler2024}.}. Although equivalent to Kolmogorov criterion, the annihilating measure condition is less easy to use. In particular, its relationship to Chebyshev equioscillation theorem is technical, see~\cite[pages 681 and 682]{Rivlin1961}, and not presented here. A relationship between the annihilating measures condition and the subdifferential condition is described in Remark~\ref{rem:annihilating-subdifferential} in Section~\ref{sss:subdifferential-kernel}.
\begin{example}[Application to Example~\ref{ex:example-1}]
	For the uniform approximation of $f(x)=x^Tx$ inside $X=\B^m$, the optimal polynomial is $p(x)=\tfrac{1}{2}$ and its error signature is $\Sigma(\bp-f)=(\Sp^{{m-1}}\!\times\!\{-1\})\cup\{(0,1)\}$. We pickup from this signature the finite subset $\{(x_1,s_1),(x_2,s_2),(x_3,s_3)\}=\{(-\one,-1),(0,1),(\one,-1)\}$, where $\one\in\R^n$ is the unit norm vector $\tfrac{1}{\sqrt{n}}(1,\ldots,1)^T$. By the compatibility~\eqref{eq:Kolmogorov-convex-combination} of affine functions and convex combinations, we have for an arbitrary $p\in \V$ that $p(x_2)=\tfrac{1}{2}p(x_1)+\tfrac{1}{2}p(x_3)$, from which we deduce the following signed measure with the same finite support and sign as the selected finite signature:
    \begin{subequations}
    \begin{align}
		\mu(p)&=\tfrac{1}{2}s_1p(x_1)+s_2p(x_2)+\tfrac{1}{2}s_3p(x_3)
    	\\ &=-\tfrac{1}{2}p(-\one)+p(0)-\tfrac{1}{2}p(\one).
    \end{align}
    \end{subequations}
%	i.e., in~\eqref{eq:annihilate} we have $1\times1\times p(\one)+2\times(-1)\times p(0)+1\times 1\times p(\one)$.
	One can see that it actually annihilates all affine functions: with an arbitrary $p(x)=\ab^T\,x+a_0$ we have $\mu(p)=-\tfrac{1}{2}p(-\one)+p(0)-\tfrac{1}{2}p(\one)=-\tfrac{1}{2}a_0+a_0-\tfrac{1}{2}a_0=0$. As a consequence of the annihilating measure condition, $\bp=\tfrac{1}{2}$ is a best uniform approximation of $x^Tx$ inside $\B^m$.
%	\paragraph{Non optimal affine approximation}
%	As shows by Equation~\eqref{eq:example-1-extreme}, the error function of non optimal affine functions have one or two extremal points.
%	
%	Now, for $\tp(x)=\sqrt{n}\,\one^T\!x+1$ the error signature is $\Sigma(\tp-f)=\{(\tfrac{-1}{\sqrt{2}}\one,-1),(\tfrac{1}{\sqrt{2}}\one,1)\}$ and annihilating measures should have the form
%	\begin{equation}
%		\mu(f)=c_1f(-\one)-2f(0)+f(\one).
%	\end{equation}
\end{example}

No extension of the annihilating measure condition with strong uniqueness has been proposed in the context of Chebyshev approximation. However, Levis et al.'s strong uniqueness condition for relative Chebyshev centers, given in Section~\ref{sss:Tanimoto}, is an annihilating measure condition for strong uniqueness, which applies in particular to Chebyshev approximation.

%**********************************************************
% Was wrong
%It is less intuitive but brings more information: for example, as shown in~\cite[proof of Theorem~1]{Rivlin1960}, a direct consequence of the existence of this annihilating measure is that all optimal generalized polynomial have the same extremal points\footnote{Consider two optimal generalized polynomials $p_0(x)=a_0^T\phi(x)$ and $p_1(x)=a_1^T\phi(x)$ with maximal error $m$, and the annihilating measure $c$ for $p_0$, which annihilates in particular $p_0-p_1$: $\sum_{i=1}^mc_i \, s_i \, (p_0(x_i)-p_1(x_i))=0$. So by adding $0=c_is_if(x_i)-c_is_if(x_i)$ we have $\sum_{i=1}^mc_i \, \bigl(s_i(p_0(x_i)-f(x_i))-s_i(p_1(x_i)-f(x_i))\bigr)=0$. Note that $s_i(p_0(x_i)-f(x_i))=m$ because $x_i$ is an extremal point for $p_0(x)$, and $|s_i(p_1(x_i)-f(x_i))|\leq m$ because $p_1(x)$ is optimal, so their difference is non-negative, which together with $c_i>0$ prove that the difference is actually null. As a consequence $p_0(x_i)=p_1(x_i)=m$ so $x_i$ are extremal points for $p_1(x)$ too.}, while this fact seems hard to prove directly from Kolmogorov criterion.
%**********************************************************

\subsubsection{Conditions for approximation by affine functions}\label{sss:affine}

Approximation by affine functions enjoys an elegant optimality condition based on signatures. This line was started by Collatz~\cite{Collatz1956}, recognized by many authors in spite of the difficulty of obtaining the original text nowadays. The characterization was further refined by~\cite{Brosowski1965}, and the final formulation is given by Rivlin and Shapiro~\cite[Remark page 697]{Rivlin1961} as a special case of their intersecting convex hull condition~\cite[Theorem 4]{Rivlin1961} (which is presented in Section~\ref{ss:basis-function-conditions} below): with $\phi(x)=(1,x_1,\ldots,x_n)$, the affine function $\bp(x)=\phi(x)^T\ba$ is an optimal Chebyshev approximation of $f(x)$ if and only if the convex hull of $\ext^+(\bp-f)$ intersects the convex hull of $\ext^-(\bp-f)$, where $\ext^\pm(\bp-f)$ contain extremal points with positive and negative error respectively. This easily follows from the annihilating measure condition: suppose that the extremal point $x^-$ has a negative error and is in the convex hull of the extremal points $x^+_1,\ldots,x^+_k$ that have a positive error. Then for an arbitrary affine function $p(x)$ we have $p(x^-)=p(\sum \lambda_ix_i)$, where $\sum\lambda_i=1$ and we can suppose $\lambda_i>0$ (otherwise $x_i$ can be removed from the sum). Hence $p(x^-)=\sum\lambda_ip( x_i)$. So we can define the annihilating measure by $\mu(f)=\sum\lambda_if( x_i)-f(x^-)$.
\begin{example}[Application to Example~\ref{ex:example-1}]
	In the signature $\Sigma(\bp-f)=(\Sp^{{m-1}}\!\times\!\{1\})\cup\{(0,-1)\}$ we see immediately that the extremal point $0$ with negative error is inside the convex hull $\Sp^{{m-1}}$, the convex hull of the extremal points with positive error, hence validating the condition.
\end{example}

Note that for univariate approximation by affine functions, the above condition $\ext^-(\bp-f)\cap\ext^+(\bp-f)\neq \emptyset$ is exactly the alternation condition of the equioscillation theorem. Rivlin and Shapiro~\cite[Problem 4 page 697]{Rivlin1961} deduce from this condition an elegant formal solution to the Chebyshev approximation problem that consists in finding the best affine approximation of $f(x)=\sum_{i=1}^m x_i^2$ on an arbitrary compact convex set $X\subseteq\R^m$. To this end, one needs only the minimal radius sphere circumscribing $X$, say with center $c\in\R^m$ and radius $r>0$. Then the best affine Chebyshev approximation is
\begin{equation}
\bp(x)=\sum_{i=1}^m x_i^2-\sum_{i=1}^m (x_i-c_i)^2+\tfrac{1}{2}r^2,
\end{equation}
which is affine because quadratic terms cancel each other exactly.

\paragraph{Uniqueness} Approximation by affine function enjoys several sufficient conditions for uniqueness, which apply to continuously differentiable functions $f$ and require some assumptions on the domain. For example, Collatz's conditions~\cite{Collatz1956} states that if $X\subseteq\R^2$ is strictly convex then the approximation by affine functions is unique. Other conditions applying to domains in $\R^n$ are presented in~\cite{Rivlin1960} but requires more technical presentations and are not detailed here.

\subsubsection{Tanimoto's condition and Levis et al.'s condition for relative Chebyshev centers}\label{sss:Tanimoto}

An extension of the annihilating measure condition to relative Chebyshev centers has been proposed by Tanimoto~\cite{Tanimoto1998}. His statement was inaccurate and was recently corrected by Levis et al.~\cite{Levis2024}. The presentation given here uses signatures of set-valued error, as defined in Section~\ref{ss:def-center}, and offers a statement homogeneous with the previous annihilating measure condition.

\begin{theorem}[Tanimoto's annihilating measure condition, corrected by Levis et al.]\label{thm:annihilate-center}
The generalized polynomial $\bp$ is a Chebyshev center of the totally complete $\F\subseteq\C(X)$ relative to $\V$ if and only if there exists a finite signature $\{(x_1,s_1),\ldots,(x_k,s_k)\}\subseteq\sig(\bp-\F)$ and constants $c_i>0$ such that
\begin{equation}\label{eq:annihilate-center}
	\forall p\in\V \ , \ \sum_{i=1}^k \ c_i \, s_i \, p(x_i) = 0.
\end{equation}
The integer $k$ can be chosen less or equal to $n+1$. Furthermore, $\tfrac{1}{2}\|f^--f^+\|_\infty<\|\bp-\F\|_\infty$ if and only if all extremal points in the signature $\sig(\bp-\F)$ are different.
\end{theorem}
\begin{remark}
In the original statement~\cite[Theorem~2.4]{Levis2024}, which does not use signatures explicitely, the condition involves $\lambda_i \, (\bp(x_i)-f(x_i)) \, p(x_i)$, for some $f\in\F$, instead of $c_i \, s_i \, p(x_i)$ for some $(x,s)\in\sig(\bp-\F)$ in the statement of Theorem~\ref{thm:annihilate-center}. They are equivalent with $\lambda_i=c_i\,|\bp(x_i)-f(x_i)|$. Using signatures here, in addition to clarifying the connection with Rivlin and Shapiro's classical signature condition, allows a slightly more accurate statement: in the original statement~\cite[Theorem~2.4]{Levis2024} the condition $\tfrac{1}{2}\|f^--f^+\|_\infty<\|\bp-\F\|_\infty$ implies the possibility of choosing different extremal points while here the condition $\tfrac{1}{2}\|f^--f^+\|_\infty<\|\bp-\F\|_\infty$ is equivalent to the signature containing only different extremal points. This slight improvement is because two different functions $f_1,f_2\in\F$ can have the same error on an extremal point, while they will contribute to only one single element in the signature.
\end{remark}
\begin{example}[Application to Example~\ref{ex:example-2-bis}]\label{ex:example-2-bis-Tanimoto}
	We consider again the uniform approximation of $\F(x)=\{x^2+\delta x:\delta \in[0,1]\}$ by affine functions, and the affine function $\bp(x)=\tfrac{1}{2}x+\tfrac{23}{32}$. The signature of the corresponding error is $\sig(\bp-\F)=\{(x_1,s_1),(x_2,s_2),(x_3,s_3),(x_4,s_4)\}=\{(-1,-1),(-\tfrac{1}{4},1),(\tfrac{1}{4},1),(1,-1)\}$. The error maximizer $x_3$ is a convex combination of the error minimizers $x_1$ and $x_4$, i.e. $x_3=\lambda x_1+(1-\lambda)x_4$ with $\lambda=\tfrac{3}{8}$. Therefore for any affine function $p\in\V$ we also have $p(x_3)=\lambda p(x_1)+(1-\lambda)p(x_4)$. We deduce that the following measure
	\begin{subequations}
    \begin{align}
		\mu(f)&=c_1\,s_1\,f(x_1)+c_3\,s_3\,f(x_3)+c_4\,s_4\,f(x_4)
    	\\ &=-\tfrac{3}{8}f(-1)+f(\tfrac{1}{4})-\tfrac{5}{8}f(1)
    \end{align}
    \end{subequations}
	annihilates every affine function, proving that $\bp$ is optimal by Tanimoto's condition.
\end{example}
\begin{example}[Application to Example~\ref{ex:example-2}]\label{ex:example-2-Tanimoto}
	We consider the simultaneous uniform approximation of $f_1(x)=-x$ and $f_2(x)=x+2$ by constant functions $p(x)=a_0$ uniformly inside $X=[-1,1]$. The error signature of $\bp(x)=1$ is $\{(1,-1),(1,1)\}$, with annihilating measure $\mu(f)=-f(1)+f(1)=0$, i.e. in~\eqref{eq:annihilate-center}, $1\times(-1)\times p(1)+1\times1\times p(1)$. In this case, the measure does not only annihilate $\V$ but also $\C(X)$ entirely. This is a rather delicate situation, where the statement~\cite[Theorem 2.4]{Levis2024} proves optimality.
\end{example}
The signature formulation~\eqref{eq:annihilate-center} shows that an extremal point appearing twice in the signature, hence with different signs, is a sufficient condition for optimality. This is explained in Remark~\ref{rem:annihilating-subdifferential} in Section~\ref{sss:subdifferential-kernel} (page~\pageref{rem:annihilating-subdifferential}) using subdifferentials calculus. Here the signature formulation of Tanimoto's condition reveals a rather odd situation: if $\tfrac{1}{2}\|f^--f^+\|_\infty=\|\bp-\F\|_\infty$ then $\bp$ is a Chebyshev center of $\F$ relative to $\V$, but this remains valid for all finite dimensional subspaces that contain $\V$. In particular, if $\|f^--f^+\|_\infty=\max_{x\in X}f^+(x)-\min_{x\in X}f^-(x)$, i.e., $f^-$ and $f^+$ attain their respective minimum and maximum at a common point, then the constant function $\bp(x)=\tfrac{1}{2}\bigl(\max_{x\in X}f^++\min_{x\in X}f^-\bigr)$ is optimal for every finite dimensional subspace that contains a constant function. Example~\ref{ex:example-2-Tanimoto} is one such situation: the constant function $\bp(x)=1$ is also the optimal affine uniform approximation, because no affine function can have a lower error at $x$=1. In such a situation, there is typically infinitely many optimal affine uniform approximation, e.g., $\bp(x)=a(x-1)+1$ for $a\in[-1,1]$ are all optimal affine uniform approximation of $\F$ in Example~\ref{ex:example-2-Tanimoto}.

\paragraph{Strong uniqueness}

Levis et al.~\cite{Levis2024} give the following characterization of optimality with strong uniqueness. The restriction of a function $f$ to a set $E\subseteq X$ is denoted by $f|_E$, and the $\V|_E$ denotes the space of generalized polynomials restricted to $E$. The statement is similar to Tanimoto's condition, expressed here again using signatures, with an additional requirement on the dimension of $\V|_{\{x_1,\ldots,x_k\}}$. 
\begin{theorem}[Levis's condition for strong uniqueness]
The generalized polynomial $\bp$ is a strongly unique Chebyshev center of the totally complete $\F\subseteq\C(X)$ relative to $\V$ if and only if there exists a finite signature $\{(x_1,s_1),\ldots,(x_k,s_k)\}\subseteq\sig(\bp-\F)$ and constants $c_i>0$ such that
\begin{equation}\label{eq:annihilate-center-strong}
	\forall p\in\V \ , \ \sum_{i=1}^k \ c_i \, s_i \, p(x_i) = 0
\end{equation}
and $\dim\V|_{\{x_1,\ldots,x_k\}}=n$. The integer $k$ is greater than $n$ and can be chosen less or equal to $2n$. Furthermore, $\tfrac{1}{2}\|f^--f^+\|_\infty<\|\bp-\F\|_\infty$ if and only if all extremal points in the signature $\sig(\bp-\F)$ are different.
\end{theorem}
The abstract condition $\dim\V|_{\{x_1,\ldots,x_k\}}=n$, which does not require knowing explicitly a basis of $\V$, is equivalent to requiring that the matrix
\begin{equation}
	H(x_1,\ldots,x_k)=\begin{pmatrix}\phi(x_1)&\phi(x_2)&\cdots&\phi(x_k)\end{pmatrix}\in\R^{n\times k},
\end{equation}
is full rank. This matrix appears in Equation~\eqref{eq:def-HaarMatrix} and is used in the context of optimality conditions based on basis function vectors presented in the next section, where it is called a Haar matrix. This latter condition is more practical using explicitly a basis of the subspace $\V$.
\begin{example}[Application to Example~\ref{ex:example-2-bis}]
In example~\ref{ex:example-2-bis-Tanimoto} we used the subsignature $\sig(\bp-\F)=\{(x_1,s_1),(x_3,s_3),(x_4,s_4)\}=\{(-1,-1),(\tfrac{1}{4},1),(1,-1)\}$, which has an annihilating measure. Now, we need to check whether $\dim\V|_{\{x_1,x_3,x_4\}}=2$ or not. For two affine functions $p_i(x)=\phi(x)^Ta_i=a_{i0}+a_{i1}x$, $i\in\{1,2\}$, we have $p_1|_{\{x_1,x_3,x_4\}}=p_2|_{\{x_1,x_3,x_4\}}$ if and only if both restricted functions agree on their restricted domain, i.e., $p_1(x_i)=p_2(x_i)$ for $i\in\{1,3,4\}$. Those three equations correspond to $H(x_1,x_3,x_4)^Ta_1=H(x_1,x_3,x_4)^Ta_2$, which implies $a_1=a_2$ because $H(x_1,x_3,x_4)\in\R^{2\times 3}$ is a Vandermonde matrix and hence is full rank. Hence the dimension of $\dim\V|_{\{x_1,x_3,x_4\}}$ is the same as the dimension of the coefficients of the affine functions, which is $2$, and Levis et al.'s condition proves that $\bp(x)=\tfrac{1}{2}x+\tfrac{23}{32}$ is strongly unique.
\end{example}
%\begin{example}[\blue{Application to Example~\ref{ex:example-2}}]
%\end{example}

\subsection{Conditions based on basis function vectors}\label{ss:basis-function-conditions}

The following conditions need not only the error extremal points $x\in \ext(a)$ and the corresponding signs, but also the basis functions evaluation $\phi$ at extremal points. The advantage of using this additional information is the easiness of checking them with respect to previous signature based conditions. The first three conditions, Kirchberger's kernel condition, Cheney's convex hull condition and Smarzewski's kernel condition, somehow express that zero is in the convex hull of a specified set. The set specified by Cheney's condition may be infinite, in which case Carathéodory theorem proves that zero is in the convex hull of a finite subset of the specified set. Such finite convex hull conditions are expressed using matrix kernel equivalent conditions, the later being easier to check in practice. Finally, we present Rivlin and Shapiro's intersecting convex hull condition, which is in fact a reformulation of the previous conditions.

\subsubsection{Kirchberger's kernel condition}

The following condition is granted to the 1903 work of Kirchberger~\cite{Kirchberger1903} by Watson's historical paper~\cite{Watson2000} and by the historical book~\cite{Steffens2006}. The same statement appears in~\cite{Osborne1969}, presented as a straightforward reformulation of Cheney's Characterization Theorem (see Section~\ref{sss:Cheney} below).

The modern statement~\cite{Watson2000,Osborne1969} uses a Haar matrix evaluated at some extremal points in $X$:
\begin{equation}\label{eq:def-HaarMatrix}
	H(x_1,\ldots,x_k)=\begin{pmatrix}\phi(x_1)&\phi(x_2)&\cdots&\phi(x_k)\end{pmatrix}\in\R^{n\times k}.
\end{equation}
In the special case of univariate polynomial approximation, the Haar matrix becomes a Vandermonde matrix
\begin{equation}
	H(x_1,\ldots,x_k)=\begin{pmatrix}1&1&\cdots&1\\x_1&x_2&\cdots&x_k\\x_1^2&x_2^2&\cdots&x_k^2\\\vdots&\vdots&\ddots&\vdots\\x_1^{n-1}&x_2^{n-1}&\cdots&x_k^{n-1}\end{pmatrix}\in\R^{n\times k}.
\end{equation}
Then, Kirchberger's condition states that a generalized polynomial $\bp(x)$ is optimal if and only if the Haar matrix $H(x_1,\ldots,x_k)$ evaluated at some error extremal points $x_i\in \ext(\bp-f)$ has a non-trivial kernel vector $u\in\R^k$ whose component signs match the error signs:
\begin{theorem}[Kirchberger's condition]\label{thm:Kirchberger}
The generalized polynomial $\bp(x)$ is a best uniform approximation of $f$ inside $X$ if and only if there exist a finite signature $\{(x_1,s_1),\ldots,(x_k,s_k)\}\subseteq\sig(\bp-f)$ such that
\begin{equation}\label{eq:Kirchberger}
\exists u(\neq0)\in\R^k, \ H(x_1,\ldots,x_k)\,u=0, \ s_i\,u_i \geq 0 \text{ for all } i\in\{1,\ldots,k\}.
\end{equation}
The integer $k$ can be chosen less or equal to $n+1$.
\end{theorem}
Obviously, extremal points for which $u_i=0$ are useless in the characterization. More precisely, as noted in~\cite{Osborne1969,Brosowski1983}, the condition can be sharpened: the set of extremal points used in~\eqref{eq:Kirchberger} can always be selected so that it is minimal, in the sense that no strict subset satisfies the same condition. In this case, we have that $H(x_1,\ldots,x_k)$ is rank $k-1$, i.e., has a kernel of dimension $1$, and therefore $k\leq n+1$, $u_i>0$ and $e(a,x_i)\,u_i > 0$ hold for all $i\in\{1,\ldots,k\}$.

\begin{example}[Application to Example~\ref{ex:example-1}]\label{ex:example-1:Kirchberger}
	We consider again the uniform approximation of $f(x)=x^Tx$ inside $\B^m$ and the optimal polynomial $\bp(x)=\tfrac{1}{2}$. We need to select a finite number of extremal points $\{x_1,\ldots,x_k\}$ and form Kirchberger's matrix
	\begin{equation}
	H=\begin{pmatrix}
		1&1&\cdots & 1
		\\x_1&x_2&\cdots&x_k
	\end{pmatrix}\in\R^{(m+1)\times k}
	.
	\end{equation}
	Since the first row contains only positive entries, any kernel vector must have components with different signs. Hence, we must select extremal points with positive and negative errors: $x_1=0$ with positive error, and $x_2$ anywhere in $S^{m-1}$ and the opposite extremal point $-x_2$, both with negative error. This leads to
	\begin{equation}
	H=\begin{pmatrix}
		1&1& 1
		\\0&x_2&-x_2
	\end{pmatrix}\in\R^{(m+1)\times 3}
	,
	\end{equation}
	whose kernel is $u=(-2,1,1)$. This kernel vector satisfies $s_i\,u_i\geq0$ hence Kirchberger condition proves that $\bp$ is optimal.
%	\paragraph{Non optimal affine approximation}
%	Now for the non-optimal polynomial $\tp(x)=n\,\one^T\!x+1$ the error signature is $\Sigma(\tp-f)=\{(-\one,-1),(\one,1)\}$ and Kirschberger's matrix is
%	\begin{equation}
%	H=\begin{pmatrix}
%		1&1
%		\\\one&-\one
%	\end{pmatrix}\in\R^{(m+1)\times 2}
%	,
%	\end{equation}
%	which is full rank, hence has no kernel, proving that $\tp$ is not optimal.
\end{example}

The optimal error of this example has infinitely many extremal points, which is untypical. In typical situations, the kernel of Kirchberger's matrix has dimension $1$, unless, e.g., the problem has symmetries. From a practical numerical point of view, only approximations of extremal points will be used. If there are more extremal points than basis functions, i.e., $k\geq n+1$, then the approximate Haar matrix has one more columns than rows, and the kernel of the approximate matrix can be computed and is an approximate kernel of the exact matrix. However, if $k\leq n$ then the exact Haar matrix has a kernel but the approximate Haar matrix has generically no kernel. In this situation, a SVD should be used to compute an approximate kernel, by chopping near zero singular values.

This kernel formulation of the optimality condition allows casting it into a system of equations, which is usually only a necessary condition (like solving $\nabla m(a)=0$ if $m(a)$ was differentiable). A standard approach in optimization is the to use a Newton operator to solve the system, see, e.g.,~\cite{Hettich1976}, where extremal points, polynomial coefficients and kernel vector components are searched together. A computer assisted proof of the existence of the solution to the resulting system of optimality conditions, together with an error bound, can be performed using Kantorovich theorem~\cite{Ciarlet2012} or interval Newton operators~\cite{Goldsztejn-C2007}.

\paragraph{Relationship to the equioscillation theorem}

The equioscillation theorem is deduced of Kirchberger's kernel condition using the following lemma, which extends to Haar matrices the well known property from $n\times(n+1)$ Vandermonde matrices that their kernel vector components have oscillating signs.
\begin{lemma}[Lemma page 74 of~\cite{Cheney1982}]\label{lem:Cheney}
	With $X=[\lx,\ux]\subseteq \R$ and the Haar condition holds for $\phi$, the Haar matrix $H(x_1,\ldots,x_{n+1})\in\R^{n\times(n+1)}$ with $x_i<x_{i+1}$ has a one dimensional kernel with a kernel vector $0\neq u\in\R^{n+1}$ satisfying $u_i\,u_{i+1}<0$.
\end{lemma}
The equioscillation theorem then follows directly from Kirchberger's condition: with $X=[\lx,\ux]\subseteq \R$ and the Haar condition holds for $\phi$, the generalized polynomial is optimal if and only if there exists $\{(x_1,s_1),\ldots,(x_k,s_k)\}\subseteq\sig(\bp-f)$, with $x_i<x_{i+1}$, and $\exists u(\neq0)\in\R^k$ such that $H(x_1,\ldots,x_k)\,u=0$ and $s_i\,u_i \geq 0$. Finally by Lemma~\ref{lem:Cheney} $u_i$ have alternating sign, hence so do the signs $s_i$ of the error at extremal points. For the converse, if the equioscillation theorem holds, then $\{(x_1,s_1),\ldots,(x_k,s_k)\}\subseteq\sig(\bp-f)$, with $x_i<x_{i+1}$, has oscillating signs. So does the kernel vector Haar matrix $H(x_1,\ldots,x_k)\,u=0$ by Lemma~\ref{lem:Cheney}, therefore either $u$ or $-u$ satisfies the requirement of Kirchberger's condition.

\subsubsection{Cheney's zero in the convex hull condition and Bartelt's condition}\label{sss:Cheney}

Cheney's Characterization Theorem~\cite[page 73]{Cheney1982} consists in testing whether $0$ is inside the convex hull of a set of vectors related to basis functions vectors evaluated at some extremal points:
\begin{theorem}[Cheney's convex hull condition]
The generalized polynomial $\bp$ is a best uniform approximation of $f$ inside $X$ if and only if
\begin{equation}\label{eq:Cheney-convex-hull}
0\in\conv\bigl\{\bigl(\bp(x)-f(x)\bigr)\,\phi(x): x\in\ext(\bp-f)\bigr\}.
\end{equation}
\end{theorem}
Note that in the convex hull condition, $f-\bp$ could be used instead of $\bp-f$, since changing the sign of all vectors does not impact the convex hull condition.
\begin{example}[Application to Example~\ref{ex:example-1}]\label{ex:example-1:Cheney}
For the uniform approximation of $f(x)=x^Tx$ inside $X=\B^m$, the optimal polynomial is $\bp(x)=\tfrac{1}{2}$ with uniform error $\tfrac{1}{2}$, and its error signature is $\Sigma(\bp-f)=(\Sp^{m-1}\!\times\!\{-1\})\cup\{(0,1)\}$. The basis vector function being $\phi(x)=(1,x)$, we see that the set $\bigl\{\bigl(\bp(x)-f(x)\bigr)\,\phi(x): x\in\ext(\bp-f)\bigr\}$ is $\bigl\{\tfrac{s}{2}\,\phi(x): (s,x)\in\sig(\bp-f)\bigr\}$, that is
\begin{equation}\label{eq:example-1:Cheney}
\tfrac{1}{2} \ \Bigl( \ (\{-1\}\!\times\!\Sp^{{m-1}})\cup\{(1,0)\} \ \Bigr).
\end{equation}
It is represented in the left plot of Figure~\ref{fig:Cheney-example-1} for $m=2$, where the blue circle represents vectors $-\tfrac{1}{2}\,\phi(x)$ for extremal points $x\in\Sp^{m-1}$ whose error is $-\tfrac{1}{2}$, and the red point represents $\tfrac{1}{2}\,\phi(x)$ for the extremal point $x=0$ whose error is $\tfrac{1}{2}$. Their convex hull is the gray cone, which contains 0. Hence Cheney's conditions proves that $\bp$ is optimal.
\end{example}
\begin{figure}
\centering
\includegraphics[width=0.32\linewidth]{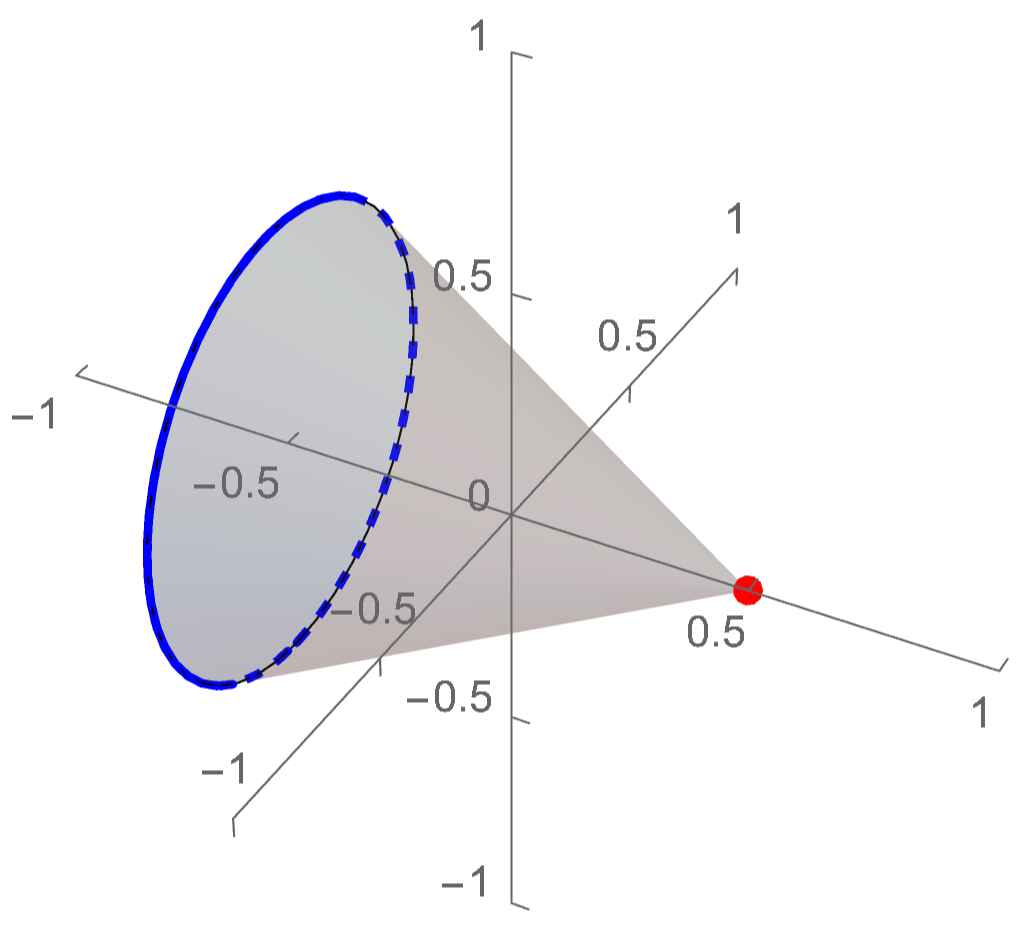}\hfill\includegraphics[width=0.32\linewidth]{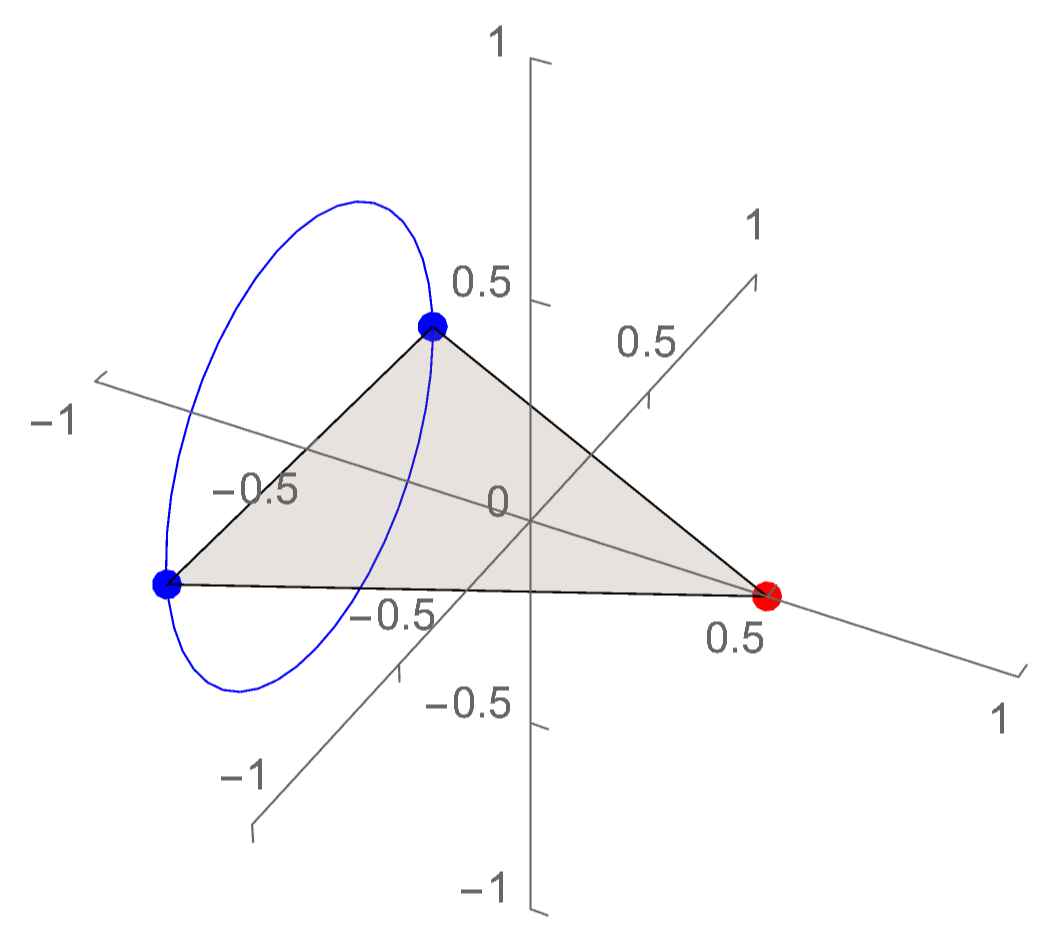}\hfill\includegraphics[width=0.32\linewidth]{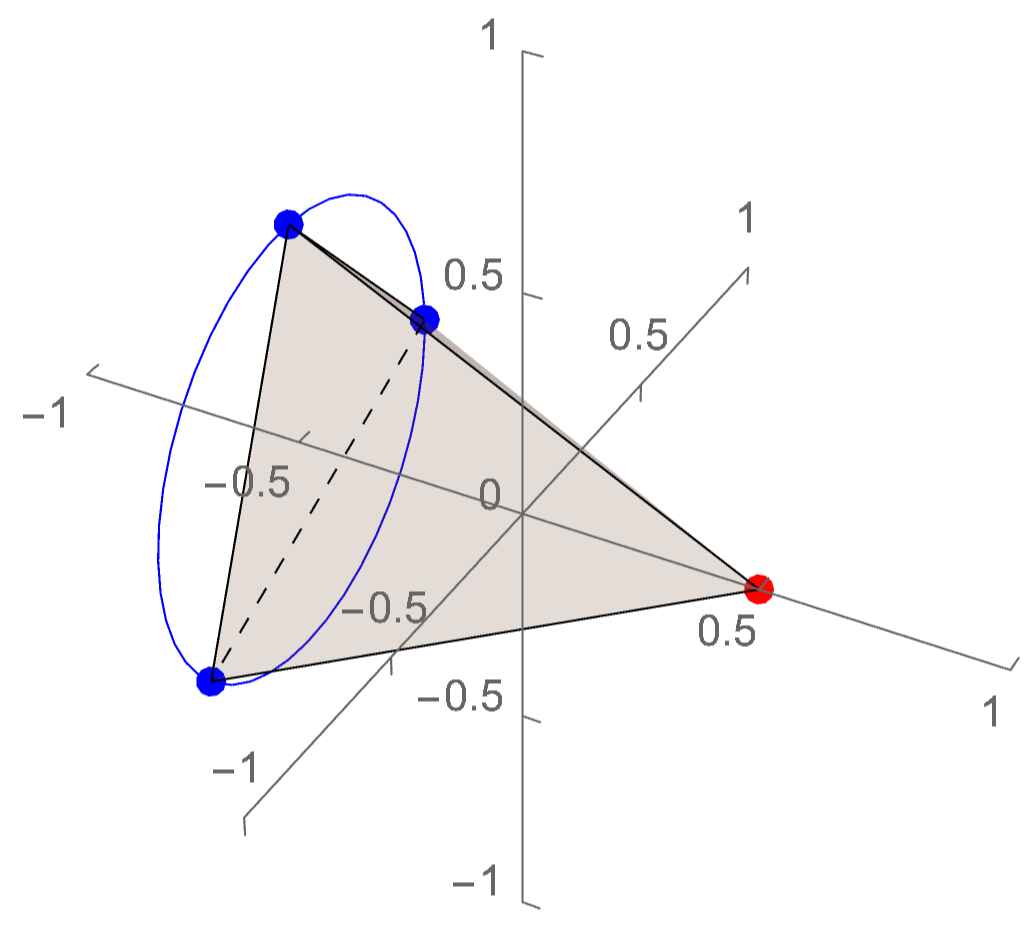}
\caption{In blue and red, the vector $\bigl(\bp(x)-f(x)\bigr)\,\phi(x)$ involved in~\eqref{eq:Cheney-convex-hull} evaluated at error extremal points that are minimizers and maximizers respectively. In There convex hull is represented in gray. From left to right: the full set of extremal points, three extremal points (two minimizers and the maximizer) and four extremal points (three minimizers and the maximizer).\label{fig:Cheney-example-1}}
\end{figure}
The set of extremal points $\ext(\bp-f)$ can be infinite, like in the previous example, but by Carathéodory theorem only a finite set of extremal points is actually necessary. We emphasize here the equivalent kernel condition for further comparison to other conditions. It involves the matrix
\begin{equation}\label{eq:Cheney-matrix}
	C(x_1,\ldots,x_k)=\begin{pmatrix} \ e_1\,\phi(x_1)&e_2\,\phi(x_2)&\cdots&e_k\,\phi(x_k) \ \end{pmatrix}\in\R^{n\times k},
\end{equation}
where $e_i=\bp(x_i)-f(x_i)$ is the error evaluated at $x_i$. The kernel version of Cheney's condition is that the matrix $C(x_1,\ldots,x_k)$ evaluated at some extremal points $x_i\in \ext(\bp-f)$ has a non-trivial kernel vector $u\in\R^k$ with non-negative components: 
\begin{corollary}[Cheney's kernel condition]
The generalized polynomial $\bp$ is a best uniform approximation of $f$ inside $X$ if and only if there exist $x_1,\ldots,x_k\in\ext(\bp-f)$ such that
\begin{equation}\label{eq:Cheney-finite}
\exists u(\neq0)\in\R^k, \ C(x_1,\ldots,x_k)\,u=0, \ u_i \geq 0 \text{ for all } i\in\{1,\ldots,k\}.
\end{equation}
The integer $k$ can be chosen less or equal to $n+1$.
\end{corollary}
The connection of Cheney's condition to Kirchberger's connection is straightforward: the columns of $C(x_1,\ldots,x_k)$ are the same as the columns of $H(x_1,\ldots,x_k)$ but multiplied by $e_i$, hence the components of the kernel vectors are also multiplied by $e_i$. As a consequence, non-negativeness of the components of a kernel vector of $C(x_1,\ldots,x_k)$ is equivalent to a kernel vector of $H(x_1,\ldots,x_k)$ having the same sign of the error for the last. Rivlin and Shapiro proved that this condition is necessary~\cite[Theorem 1 page 672]{Rivlin1961}. As mentioned previously, Osborn and Watson identifies Cheney's condition with Kirchberger's condition in~\cite{Osborne1969}. Cheney's condition is also proved in~\cite[Theorem 2.4 page 72]{Rivlin2020}.
\begin{example}[Application to Example~\ref{ex:example-1}]\label{ex:example-1:Cheney-finite}
As in Example~\ref{ex:example-1:Kirchberger}, we select the error maximum $x_1=0$ and to opposite minimizers $x_2=v_1$, where $v_1=(1,0,\ldots,0)^T$ is the first vector basis, and $x_3=-v_1$. The center plot of Figure~\ref{fig:Cheney-example-1} shows that zero is still in the convex hull for this subset of extremal points. Equivalently, Cheney's matrix
	\begin{equation}\label{eq:ex:example-1:Cheney-finite}
	C=\frac{1}{2}\begin{pmatrix}
		1&-1& -1
		\\0&-v_1&v_1
	\end{pmatrix}\in\R^{(m+1)\times 3},
	\end{equation}
	whose columns are $e_i\,\phi(x_i)$ with $\phi(x)=(1,x)$, has a nontrivial non-negative kernel vector, here $u=(2,1,1)^T$.
\end{example}

\paragraph{Strong uniqueness}

Bartelt~\cite[Theorem 6]{Bartelt1973} proves the following necessary and sufficient condition for strong uniqueness.
\begin{theorem}[Bartelt's convex hull condition]\label{thm:Bartelt-convexhull}
The generalized polynomial $\bp$ is a strongly unique best uniform approximation of $f$ inside $X$ if and only if
\begin{equation}\label{eq:thm:Bartelt-convexhull}
0\in\interior\conv\bigl\{\bigl(\bp-f\bigr)\,\phi(x): x\in\ext(\bp-f)\bigr\}.
\end{equation}
\end{theorem}
\begin{example}[Application to Example~\ref{ex:example-1}]
As in example~\ref{ex:example-1:Cheney}, we see on the left plot of Figure~\ref{fig:Cheney-example-1} that not only zero is inside the convex hull of Cheney's $\conv\bigl\{\bigl(\bp-f\bigr)\,\phi(x): x\in\ext(\bp-f)\bigr\}$ but that it furthermore inside its interior. This holds obviously in arbitrary dimension for the convex hull of~\eqref{eq:example-1:Cheney}. Hence, $\bp(x)=\tfrac{1}{2}$ is strongly unique.
\end{example}
Theorem~\ref{thm:Bartelt-convexhull} is a qualitative condition in the sense that it does not provide any information on the value of $r$ in the definition of strong uniqueness. We deduce the following interesting property of Bartelt's condition: the convex hull in~\eqref{eq:thm:Bartelt-convexhull} must have a nonempty interior, therefore there must exist at least $n+1$ extremal points so that the optimal approximation is strongly unique.

\subsubsection{Smarzewski's condition for strong uniqueness}

Smarzewski~\cite[Corollary 2.2 page 218]{Smarzewski1990} proved a kernel optimality condition for strong uniqueness. He uses a matrix similar to Kirchberger's and Cheney's matrices:
\begin{equation}\label{eq:Smarzewski-matrix}
	S(x_1,\ldots,x_k)=\begin{pmatrix}s_1\phi(x_1)&s_2\phi(x_2)&\cdots&s_k\phi(x_k)\end{pmatrix}\in\R^{n\times k},
\end{equation}
where $s_i$ is the sign of the error at the extremal point $x_i$. Its columns are again proportional to the columns of the previous two matrices. In particular, the vectors $\phi(x_i)$ are multiplied by the sign of the error in Smarzewski's matrix instead of the error itself in Cheney's matrix, so that $C(x_1,\ldots,x_k)=\|p-f\|_\infty \,S(x_1,\ldots,x_k)$.
\begin{theorem}[Smarzewski's kernel condition]\label{thm:Smarzewski}
The generalized polynomial $\bp$ is the strongly unique best uniform approximation of $f$ inside $X$ if and only if there exists a finite signature $\{(x_1,s_1),\ldots,(x_k,s_k)\}\subseteq\Sigma(\bp-\F)$ such that
\begin{equation}\label{eq:Smarzewski}
\exists u\in\R_{>0}^k, \ S\,u=0,
\end{equation}
where $S=\begin{pmatrix}s_1\phi(x_1)&s_2\phi(x_2)&\cdots&s_k\phi(x_k)\end{pmatrix}\in\R^{n\times k}$ is furthermore full rank. The integer $k$ is greater than $n$ and can be chosen less or equal to $2n$.
\end{theorem}
\begin{example}[Application to Example~\ref{ex:example-1}]\label{ex:example-1-Smarzewski}
	Like in Example~\ref{ex:example-1:Cheney-finite}, we can use the signature $\{(0,1),(v_1,-1),(-v_1,-1)\}$, which leads to  Smarzewski's matrix $S=2C$ with $C$ Cheney's matrix~\eqref{eq:ex:example-1:Cheney-finite}. For $m=1$ the matrix $S$ is full rank in addition to have a kernel vector with positive components, hence Smarzewski's condition proves the optimal polynomial $\bp(x)=\tfrac{1}{2}$ is strongly unique. However, for $m\geq 2$ the matrix is not full rank anymore (it has a nontrivial kernel and the number of rows is greater or equal to the number of columns). In this case, we need to select another finite signature whose convex hull has an interior that contains zero\footnote{The connection between Smarzewski's kernel condition for strong minimality and the convex hull conditions will be clarified in Section~\ref{sss:subdifferential-kernel} using the subgradient interpretation of Smarzewski's condition.}, like in the third plot of Figure~\ref{fig:Cheney-example-1}:
	\begin{equation}\label{eq:ex:example-1-Smarzewski-signature}
	\{(0,1),(e_1,-1),\ldots,(e_m,-1),(-\one,-1)\},
	\end{equation}
	where $e_i\in\R^m$ is the $i^\mathrm{th}$ basis vector and $\one\in\R^m$ is again the unit norm vector $\tfrac{1}{\sqrt{m}}(1,\ldots,1)^T$, leading to the following Smarzewski's matrix:
	\begin{equation}\label{eq:ex:example-1-Smarzewski-S}
	S=\frac{1}{2}\begin{pmatrix}
		1&-\one^T& -1
		\\0&-I&\one
	\end{pmatrix}\in\R^{(m+1)\times (m+2)}
	\end{equation}
	whose columns are $s_i\,\phi(x_i)$ with $\phi(x)=(1,x)$. This matrix is full rank and has a positive kernel vector $(m+\sqrt{m},1,\ldots,1,\sqrt{m})\in\R^{m+2}$, therefore Smarzewski's condition proves the optimal polynomial $\bp(x)=\tfrac{1}{2}$ is strongly unique.
\end{example}

Brosowski~\cite[Theorem 3]{Brosowski1983} gives an equivalent condition expressed in a technical way involving the set of all minimal extremal signatures, which is not detailed here.

\subsubsection{Rivlin and Shapiro's intersecting convex hull condition}

In the typical case where the set of basis functions includes a constant function, Rivlin and Shapiro~\cite[Theorem 4 page 696]{Rivlin1961} used their annihilating measure condition to prove the following more practical condition: the generalized polynomial $p(x)=\phi(x)^Ta$ is optimal if and only if
\begin{equation}
\conv\{\phi(x):x\in\ext^+(a)\}\cap\conv\{\phi(x):x\in\ext^-(a)\}\neq\emptyset,
\end{equation}
where as previously $\ext^\pm(\ba )$ contain extremal points with positive and negative error respectively. This condition was rediscovered by Sukhorukova et~al.~\cite[Theorem 2]{Sukhorukova2018}, where the authors made appear its connection to zero in the convex hull conditions, and further improved it to obtain a practical test~\cite{Sukhorukova2021}.
%\footnote{Without loss of generality, we suppose that $\phi_1(x)=1$. The subgradient convex hull condition $\sum_{i=1}^m\lambda_i\,\epsilon(\ba,x_i)\,\phi(x_i)=0$, for some nonnegative $\lambda_i$ not all zero, can be rewritten so that contributions of positive and negative errors are grouped together: $\sum_{i\in I_+}\lambda_i\,\phi(x_i)=\sum_{i\in I_-}\lambda_i\,\phi(x_i)$, with $I_\pm=\{i:\epsilon(\ba,x_i)=\pm1\}$. Since $\phi_1(x)=1$ we must have $\sum_{i\in I_+}\lambda_i=\sum_{i\in I_-}\lambda_i=\tfrac{1}{2}\|\lambda\|_1$, and dividing both sides by $\tfrac{1}{2}\|\lambda\|_1$ shows it expresses exactly that $\conv\{\phi(x_i)\in\R^n:i\in I_+\}$ and $\conv\{\phi(x_i)\in\R^n:i\in I_-\}$ have a nonempty intersection.}
\begin{example}[Application to Example~\ref{ex:example-1}]
	For the uniform approximation of $f(x)=x^Tx$ inside $X=\B^m$, in view of the expression~\eqref{eq:example-1-extreme}, the only case where the convex hull of the error minimizers intersects the convex hull of the error maximizers is when minimizers are a sphere, and the maximizer is a point inside the sphere, corresponding to the optimal affine function $\bp(x)=\tfrac{1}{2}$.
\end{example}

%%%%%%%%%%%%%%%%%%%%%%%%%%%%%%%%%%%%%%%%%%%%%%%%%%%%%%%%%%%%%%%%%%%%%%
%%%%%%%%%%%%%%%%%%%%%%%%%%%%%%%%%%%%%%%%%%%%%%%%%%%%%%%%%%%%%%%%%%%%%%
%%%%%%%%%%%%%%%%%%%%%%%%%%%%%%%%%%%%%%%%%%%%%%%%%%%%%%%%%%%%%%%%%%%%%%
\section{Convex optimization optimality conditions}\label{s:convex-optimality-conditions}

The function $m:\R^n\rightarrow \R$ defined by
\begin{equation}\label{eq:m}
m(a)=\max_{x\in X}|\phi(x)^Ta-f(x)|=\|p-f\|_\infty
\end{equation}
is called a pointwise supremum in the context of nonsmooth convex optimization. This section briefly presents the background in convex analysis needed for the optimality conditions of such pointwise supremum functions. The reader is referred to~\cite{jbhu2001,Nesterov2018} for an introduction to convex optimization (\cite[Section 5.3 page 198]{jbhu2001} includes some basic facts about uniform approximation in the framework of convex optimization, including a simplified version of Chebyshev equioscillation theorem). In the context of nonsmooth convex analysis, subgradients and subdifferentials of nonsmooth convex functions generalize gradient of differentiable functions. They are presented in Subsection~\ref{ss:subdifferential}, together with some elementary rules for their computation. Optimality conditions for convex functions rely on subdifferentials and are presented in Subsection~\ref{ss:optimality-convex}, with an emphasis on conditions for sharp minimality, the convex optimization counterpart of strong minimality. This framework is instantiated to the pointwise supremum function~\eqref{eq:m} in Subsection~\ref{ss:Chebyshev-convex}, leading to optimality conditions for Chebyshev approximation, where relative Chebyshev centers are treated in a homogeneous way.

%%%%%%%%%%%%%%%%%%%%%%%%%%%%%%%%%%%%%%%%%%%%%%%%%%%%%%%%%%%%%%%%%%%%%%
%%%%%%%%%%%%%%%%%%%%%%%%%%%%%%%%%%%%%%%%%%%%%%%%%%%%%%%%%%%%%%%%%%%%%%
\subsection{The subdifferential of a convex function}\label{ss:subdifferential}

Convex functions are generally defined inside $\R^n$ with values into $\R\cup\{+\infty\}$, where the the value $+\infty$ allows encoding a domain: vectors where the value is $+\infty$ are not candidate for being minimizers, hence interpreted as outside the domain. This offers an elegant homogeneous treatment of constrained and unconstrained optimization, with unified optimality conditions. Here, we restrict our attention to convex functions $m:\R^n\rightarrow\R$, which enjoy a sensibly simpler theory. Most definitions and properties presented here extend naturally to the general case, the reader is reffered to standard textbooks, e.g.,~\cite{jbhu2001,Nesterov2018}.

\subsubsection{Definition of the subdifferential and basic facts}

Let $m:\R^n\rightarrow \R$ be a convex function. If $m$ is differentiable at $\ba$ then it's gradient $g=\nabla m(\ba)$ is the unique vector such that the affine function $t(a)=m(\ba)+g^T(a-\ba)$ is the tangent, i.e., it satisfies $m(a)=t(a)+o(\|a-\ba\|)$. It turns out that, because of convexity, this tangent is also an affine underestimator, i.e.
\begin{equation}\label{eq:affine-underestimator}
\forall a\in\R^n \ , \ m(a)\geq m(\ba)+g^T(a-\ba).
\end{equation}
The tangent is in addition the only affine underestimator. In fact, the uniqueness of the affine underestimator characterizes the differentiability of convex functions. A convex function that has several affine underestimators at some $\ba$ is not differentiable at this $\ba$, see Figure~\ref{fig:subdifferential-definition}. Subgradients generalize gradients in the sense that they correspond to affine underestimators of convex functions: $g\in\R^n$ is a subgradient of $m$ at $\ba$ if and only if $m(a)\geq m(\ba)+g^T(a-\ba)$ holds for all $a\in\R^n$. The subdifferential of $m$ at $\ba$, denoted by $\partial m(\ba)$, is the set of all subgradients at $\ba$. As said previously, a convex function $m(a)$ is differentiable at $\ba$ if and only if it has a unique affine underestimator at $\ba$, which turns out to be the gradient, i.e., if and only if $\partial m(a)=\{\nabla m(a)\}$. Subdifferentials are of critical importance for characterizing optimal solutions of convex optimization problems, hence the need of understanding their structure. In our finite dimensional setting, subdifferentials are nonempty, compact and convex. Subdifferentials often result of the convex hull of a smaller set of subgradients, called here \emph{generating sets of subgradients}. These generating sets of subgradients are important because they offer a convenient representation of subdifferentials.
%
%We call a set $E$ of subgradients whose convex hull is the subdifferential, i.e., such that $\partial m(a)=\conv E$, a generating set of subgradients. Generating sets of subgradients are important because they offer a convenient representation of a subdifferential. The subdifferential being convex and compact, the finite dimensional version of the Krein–Milman theorem shows that the subdifferential is the convex hull of its set of extreme\footnote{What we call here extreme subgradients are extremal points of the subdifferential, a classical notion in convex analysis. Unfortunately, in spite of similar terms, they are not related to extremal signatures, which characterizes optimal solutions of Chebyshev optimization problems.} subgradients, i.e., subgradients that are not convex combinations of any other subgradients. The set of extreme subgradients of $\partial m(a)$ is denote by $E(a)$. It is the smallest generating set of subgradients. In practice, one may not be able to identify the smallest generating set of subgradient that generates the subdifferential, but will be interested by small generating sets of subgradients.
%
The following two examples show subdifferentials of typical convex functions.
\begin{figure}[t!]
	\includegraphics[width=0.5\linewidth]{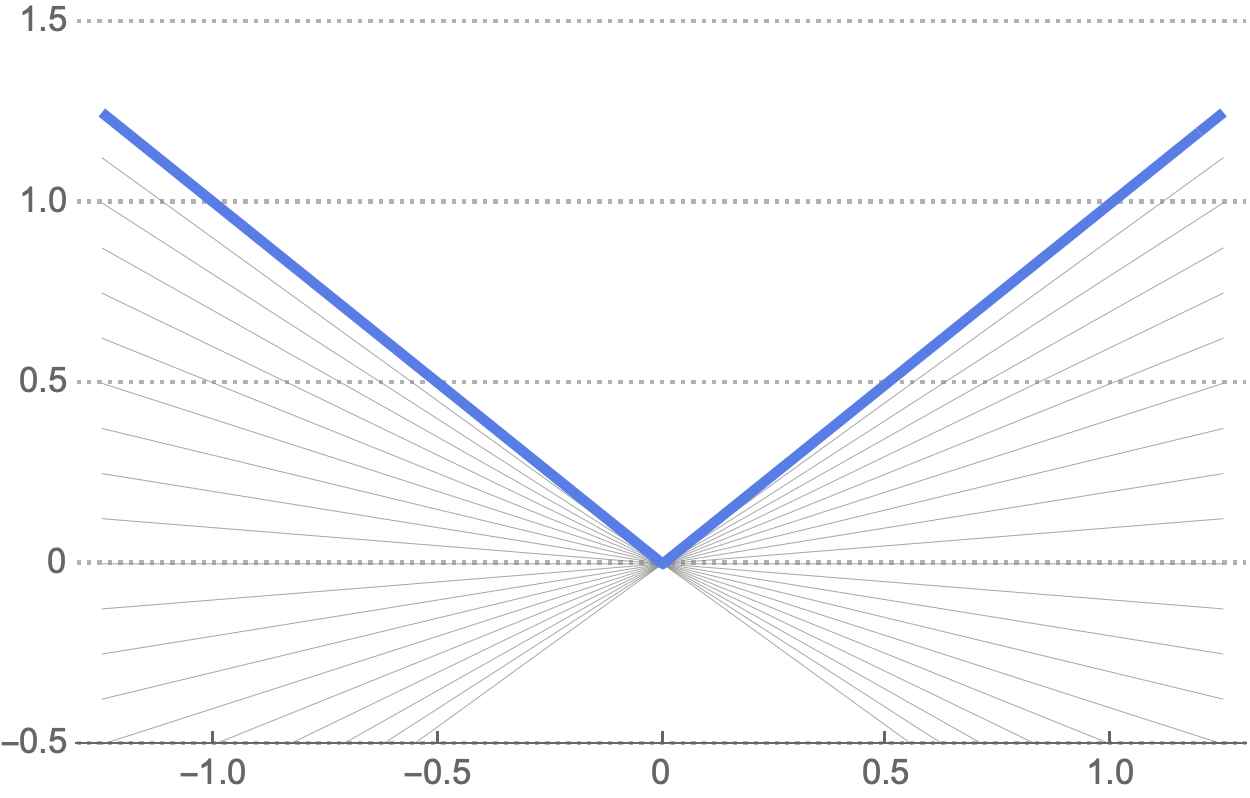}\hfill\includegraphics[width=0.4\linewidth]{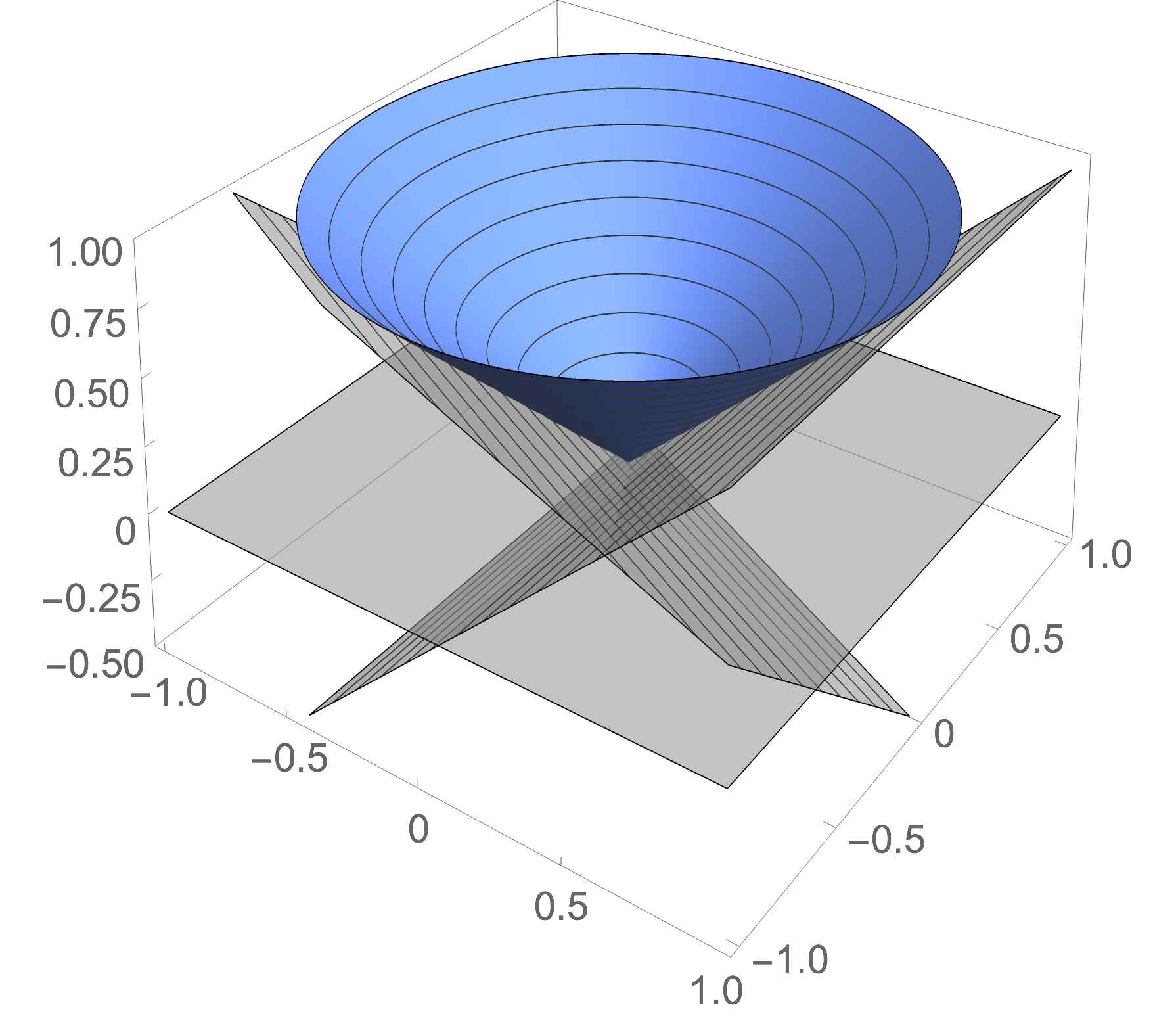}
	\caption{The function $m(a)=|a|$ on the left, the function $m(a)=\|a\|_2=\sqrt{a_1^2+a_2^2}$ on the right. In blue, the graph of the functions; in gray, some affine underestimators at $a=0$. The functions are differentiable for $a\neq0$, therefore the only affine underestimator is the tangent.\label{fig:subdifferential-definition}}
\end{figure}
\begin{example}[Subdifferential of $\abs$ function]\label{ex:abs}
	In the context of univariate functions, subgradients are vectors of dimension one which are identified to scalars and called subderivatives. We consider the function $m(a)=|a|$. It is differentiable for $a\neq0$, in which case $\partial m(a)=\{m'(a)\}$, that is explicitly, $\partial m(a)=\{-1\}$ if $a<0$ and $\partial m(a)=\{1\}$ if $a>0$. Now for $a=0$, the left plot of Figure~\ref{fig:subdifferential-definition} shows all affine underestimators $a\mapsto m(0)+g\,(a-0)=g\,a$. We can see on the figure that such an affine function is an underestimator if and only if $g\in[-1,1]$, hence $\partial m(0)=[-1,1]=\conv\{-1,1\}$. The set $\{-1,1\}$ is therefore a generating set of subderivatives, and furthermore it is the smallest generating set of subderivatives.
\end{example}
In the context of subdifferentials, it is convenient to define the sign function as the set-value function corresponding to the subdifferential of the absolute value function:
\begin{equation}
	\Sign(a)=\left\{\begin{array}{ll}\{-1\}&\text{if }a<0\\{[-1,1]}&\text{if }a=0\\\{1\}&\text{if }a>0,\end{array}\right.
\end{equation}
so that $\partial|a|=\Sign(a)$. We will use the non capitalized name function for the usual real-valued $\sign$ function $\sign(a)=a/|a|$ for $a\neq0$ and $\sign(0)=0$.
\begin{example}[Subdifferential of the Euclidean norm]\label{ex:2norm}
	The function $m(a)=\|a\|_2$ is represented in the right plot of Figure~\ref{fig:subdifferential-definition} for $n=2$. It is differentiable for $a\neq0$ with $\nabla m(a)=a/\|a\|_2$ therefore $\partial m(a)=\left\{a/\|a\|_2\right\}$. In order to compute $\partial m(0)$ we need to find all affine underestimators at the origin, i.e., find all vectors $g$ such that $m(0)+g^T(a-0)\leq \|a\|_2$ holds for all $a\in\R^n$ (three such underestimators are represented in the right plot of Figure~\ref{fig:subdifferential-definition}). By Cauchy-Schwartz inequality, those vector $g$ are exactly those that satisfy $\|g\|_2\leq1$. As a consequence, $\partial m(0)$ is the unit Euclidean ball. The unit Euclidean sphere is the smallest generating set of subgradients.
\end{example}

An crucial property of subdifferentials is that they allow generalizing the directional derivative computation that uses gradient in the case of differentiable functions: the directional derivative of a convex function can be computed using its subdifferential~\cite[Theorem 4.4.2 page 189]{jbhu2001} by
\begin{equation}\label{eq:dirder-general}
m'(a,u)=\max\{u^T\,g:g\in\partial m(a)\}.
\end{equation}
If the function is differentiable at $a$ then $\partial m(a)=\{\nabla m(a)\}$ and we recover the usual directional derivative of differentiable functions: $m'(a,u)=u^T\,\nabla m(a)$.

\subsubsection{Elementary computation rules}

Standard subdifferential computation rules allow finding practical generating sets of subgradients, see e.g.~\cite[Section 4 page 183]{jbhu2001} or~\cite[Section 3.1.6 page 167]{Nesterov2018}. One such simple and useful subdifferential calculation rule is the subdifferential of the sum of two functions, which is the Minkowski sum of their respective subdifferentials: $\partial (m_1+m_2)(a)= \partial m_1(a)+\partial m_1(a)$. The inclusion $\partial (m_1+m_2)(a)\supseteq \partial m_1(a)+\partial m_1(a)$ is an easy consequence of the subgradient definition: if $g_i\in m_i(\ba)$ then we have $m_i(\ba)+g_i^T(a-\ba)\leq m_i(a)$ hold for all $a\in\R^n$. Summing the two inequalities leads to
\begin{equation}
	m_1(\ba)+m_2(\ba)+(g_1+g_2)^T(a-\ba)\leq m_1(a)+m_2(a)
\end{equation}
for all $a\in\R^n$, which means that $g_1+g_2\in \partial (m_1+m_2)(\ba)$. The reverse inclusion is more subtil and enjoys no simple elementary proof,~\cite{jbhu2001} and~\cite{Nesterov2018} give two different proofs based on the epigraph of the function and on the directional derivative of the function respectively. This subdifferential calculation rule is illustrated in the following two examples.
\begin{figure}[t!]
	\includegraphics[width=0.45\linewidth]{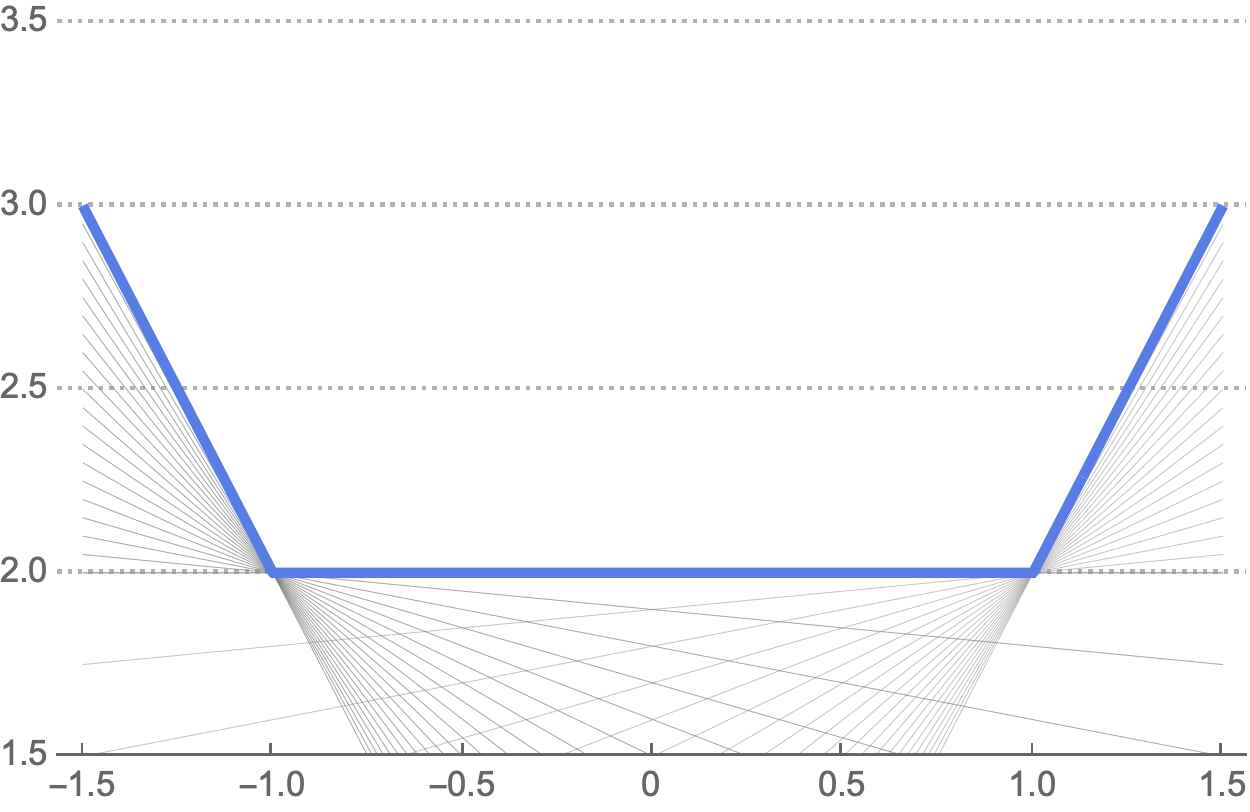}\hfill\includegraphics[width=0.4\linewidth]{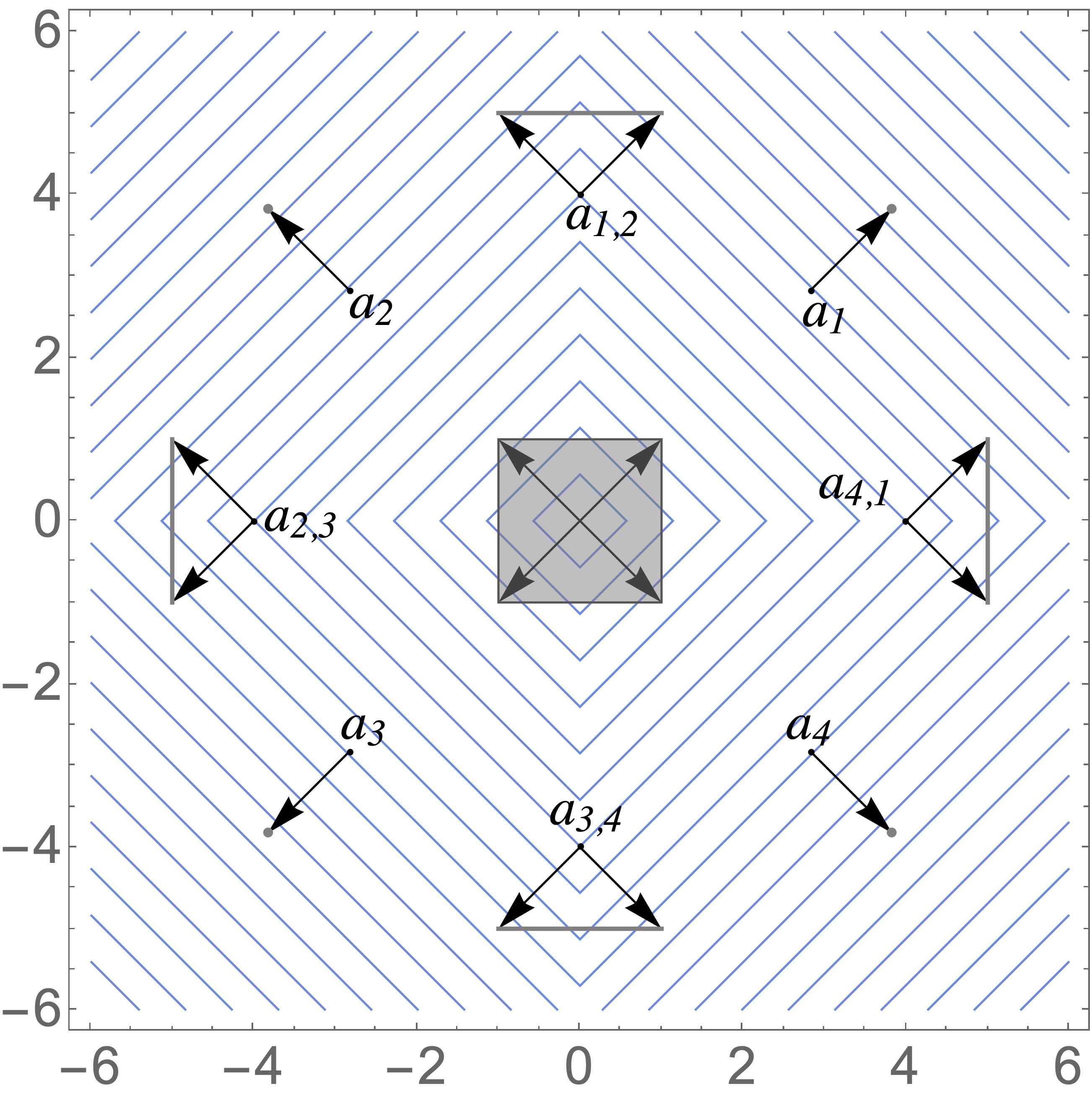}
	\caption{On the left the graph of the function $m(a)=|a-1|+|a+1|$ in blue and representative affine underestimators at $a=-1$ and $a=1$ in gray; we can see that $\partial m(-1)=[-2,0]$ and $\partial m(1)=[0,2]$. On the right, the $1$-norm level-sets with subdifferentials, which are convex hull of subgradients represented by arrows, shown at some representative vectors $a_{ij}$.\label{fig:subdifferential-sums}}
\end{figure}
\begin{example}[Sum of univariate functions]\label{ex:subdifferential-ex-sumunivar}
	Consider the function $m(a)=|a+1|+|a-1|$. Since $\partial|a+1|=\Sign(a+1)$ and $\partial|a-1|=\Sign(a-1)$ we have $\partial m(a)=\Sign(a+1)+\Sign(a-1)$, which is explicitly
    \begin{equation}\label{eq:subdifferential-ex-sumunivar}
    	\partial m(a)=\left\{\begin{array}{rll}\{-1\}+\{-1\}&= \ \{-2\}&\text{if }a<-1\\{[-1,1]+\{-1\}}&= \ [-2,0]&\text{if }a=-1\\\{1\}+\{-1\}&= \ \{0\}&\text{if }-1<a<1\\ \{1\}+[-1,1]&= \ [0,2]&\text{if }a=1\\\{1\}+\{1\}&= \ \{2\}&\text{if }a>1.\end{array}\right.
    \end{equation}
    Therefore, $m(a)$ is differentiable excepted at $a=-1$ and $a=1$. The underestimators corresponding to the subdifferentials at $a=-1$ and $a=1$ are represented on the left plot of Figure~\ref{fig:subdifferential-sums}.
\end{example}
Gradients of differentiable functions can be computed component wise, but this does not hold generally for subderivatives, for which componentwise calculation leads to overestimation. However, it is not hard to see that if $m(a)$ depends on one variable only, i.e., $m(a)=\tilde m(a_i)$, then its subdifferential is
\begin{equation}
\partial m(a)=(\{0\},\ldots,\{0\},\partial \tilde m(a_i),\{0\},\ldots,\{0\}),
\end{equation}
where a vector with set components is understood as the cartesian product of its components. The following example shows how this subdifferential calculation rule can be used together with the summation calculation rule.
\begin{example}[The $\ell^1$ norm as a sum of multivariate functions]\label{ex:norm1}
	The $1$-norm is the sum of univariate absolute value functions: $m(a)=\|a\|_1=|a_1|+\cdots+|a_n|$. Its subdifferential is therefore the sum of the differentials of each summand, which has been previously found to be the sign function:
	\begin{equation}\label{eq:ex:norm1}
	\partial m(a)=\begin{pmatrix}\Sign(a_1)\\\{0\}\\ \vdots\\\{0\}\end{pmatrix}+\begin{pmatrix}\{0\}\\\Sign(a_2)\\ \vdots\\\{0\}\end{pmatrix}+\cdots+\begin{pmatrix}\{0\}\\\{0\}\\ \vdots\\\Sign(a_n)\end{pmatrix}=\begin{pmatrix}\Sign(a_1)\\\Sign(a_2)\\ \vdots\\\Sign(a_n)\end{pmatrix}.
	\end{equation}
	The case $n=2$ is shown on the right plot of Figure~\ref{fig:subdifferential-sums}, where subdifferentials are represented for selected vectors. E.g., the vector $a_{1,2}=(0,1)$ belongs to the first and second quadrants, hence
	\begin{equation}
	\partial m(a_{1,2})=\begin{pmatrix}\Sign(0)\\\Sign(1)\end{pmatrix}=\begin{pmatrix}[-1,1]\\ \{1\}\end{pmatrix}=\conv\left\{\begin{pmatrix}-1\\ 1\end{pmatrix},\begin{pmatrix}1\\ 1\end{pmatrix}\right\}.
	\end{equation}	
	For the origin, we have
	\begin{equation}
	\partial m(0)=\begin{pmatrix}\Sign(0)\\\Sign(0)\end{pmatrix}=\begin{pmatrix}[-1,1]\\ [-1,1]\end{pmatrix}=\conv\left\{\begin{pmatrix}-1\\ 1\end{pmatrix},\begin{pmatrix}1\\ 1\end{pmatrix},\begin{pmatrix}1\\ -1\end{pmatrix},\begin{pmatrix}-1\\ -1\end{pmatrix}\right\}.
	\end{equation}
	Note that in $\R^n$ the convex hull representation requires $2^n$ generating vector while the cartesian product representation requires only $n$ intervals.
\end{example}

Another standard rule of particular interest here is the one for pointwise supremum functions, which are functions
\begin{equation}
m(a)=\max_{x\in X} \, m_x(a),
\end{equation}
with $X$ is compact, $m_x:\R^n\rightarrow\R$ is convex for all $x\in X$, and the function $x\mapsto m_x(a)$ is upper semi-continuous\footnote{A function $m(x)$ is upper semi-continuous at $x$ if $\limsup_{x_k\rightarrow x}m(x_k)\leq m(x)$ for all sequences $(x_k)_{k\in\N}$ that converges to $x$. We will consider only continuous functions in the sequel, which are upper semi-continuous.} in $X$ for all $a\in\R^n$.
%As previously, this latter function will be denoted by $m_a:X\rightarrow \R$.
Then, $m:\R^n\rightarrow\R$ is also convex and hence continuous. The subdifferential of this pointwise supremum $m(a)$ enjoys a simple explicit expression:
\begin{equation}\label{eq:subdifferential-ps}
\partial m(a)=\conv\bigcup\{\partial\, m_x(a):x\in \act(a)\},
\end{equation}
where $\act(a)=\{x\in X:m_x(a)=m(a)\}$ contains the maximizers of $m_a$, often called active indices.
\begin{remark}
Remind that the union operation acts on a set of sets, e.g., the union of the two sets $A$ and $B$ is denoted by $\bigcup\{A,B\}$, with the usual shortcut $A\cup B$. In~\eqref{eq:subdifferential-ps}, the subdifferential is defined to be the convex hull of the union of all subdifferentials $\partial\, m_x(a)$ for each $x\in \act(a)$.
\end{remark}
Formulas for the subdifferential of a pointwise supremum dates from the $70$s, and are nowadays textbook classic. As for the summation rule, the inclusion $\partial m(a)\supseteq\conv\bigcup\{\partial\, m_x(a):x\in \act(a)\}$ is a simple consequence of the subgradient definition and holds with no additional assumption. Again, the reverse inclusion is more subtil, and holds only under additional assumptions. Several versions exist with different assumptions, the one used here, \cite[Theorem 4.4.2 page 189]{jbhu2001}, is the simplest and is restricted to real-valued convex functions defined in $\R^n$, and requires only that $X$ is compact and the functions $x\mapsto m_x(a)$ is upper semi-continuous for all $a\in\R^n$. Several extensions exist, where the convex function has values in $\R\cup\{+\infty\}$~\cite[Proposition A22 page 154]{Bertsekas1971-PhD} encoding a domain for the function, and to infinite dimensional spaces, where Valadier's theorem~\cite{Valadier1969} plays a central role. See~\cite[page 247]{jbhu2001} for some historical details about this formula. Weakening assumptions for such formulas is a current research topic, e.g.,~\cite{Correa2023} and references therein.

In our context, $m_x(a)$ will be differentiable at active indices, i.e., $\partial\, m_x(a)=\{\nabla m_x(a)\}$ whenever $x\in \act(a)$. In this case, the subdifferential~\eqref{eq:subdifferential-ps} becomes
\begin{equation}\label{eq:subdifferential-ps-diff}
	\partial m(a)=\conv\{\nabla m_x(a):x\in\act(a)\}.
\end{equation}
Computing the subdifferential of a pointwise supremum function therefore mainly consists in finding the set of active indices and evaluating the gradient of $m_x(a)$ at these active indices. The following example reinterprets the previous absolute value function as finite pointwise maximum.
\begin{example}[The $\abs$ function as a pointwise maximum]
The absolute value function can be defined as a finite point wise supremum function, also called a pointwise maximum function: $m(a)=|a|=\max_{x\in X}m_x(a)$ with $X=\{-1,1\}$, $m_x(a)=x \,a$. Its subdifferential can be computed using~\eqref{eq:subdifferential-ps}. To this end, for a fixed $a$ we must find $\act(a)$, the set of indices $x$ such that $m(a)=m_x(a)$: here, if $a<0$ then $\act(a)=\{-1\}$; if $a>0$ then $\act(a)=\{1\}$; finally $\act(0)=\{-1,1\}$. The subdifferential follows: if $a<0$ then $\partial m(a)=\conv\bigcup\{\partial m_{-1}(a)\}=\{m_{-1}'(a)\}=\{-1\}$; if $a>0$ then $\partial m(a)=\conv\bigcup\{\partial m_{1}(a)\}=\{m_{1}'(a)\}=\{-1\}$; finally if $\partial m(0)=\conv\bigcup\{\partial m_{-1}(a),\partial m_{1}(a)\}=\conv\{m_{-1}'(a),m_{1}'(a)\}=[-1,1]$.
\end{example}
The next two examples show the explicit subdifferential computation for two examples of simple Chebyshev approximation problems. Formal computations are possible in these two cases of approximation of quadratic functions by linear or affine functions.
\begin{figure}[t!]
	\includegraphics[width=0.32\linewidth]{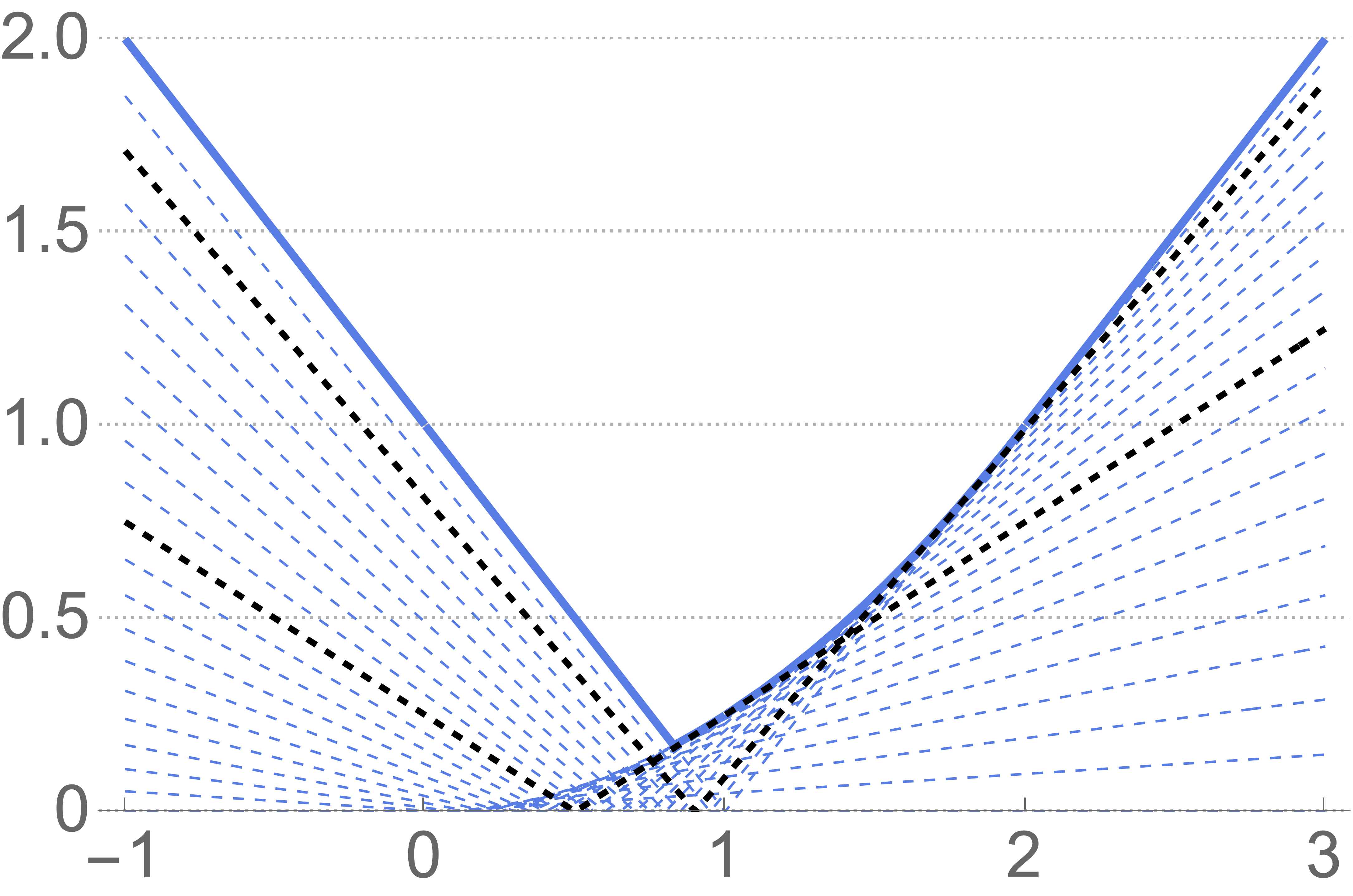}\hfill\includegraphics[width=0.32\linewidth]{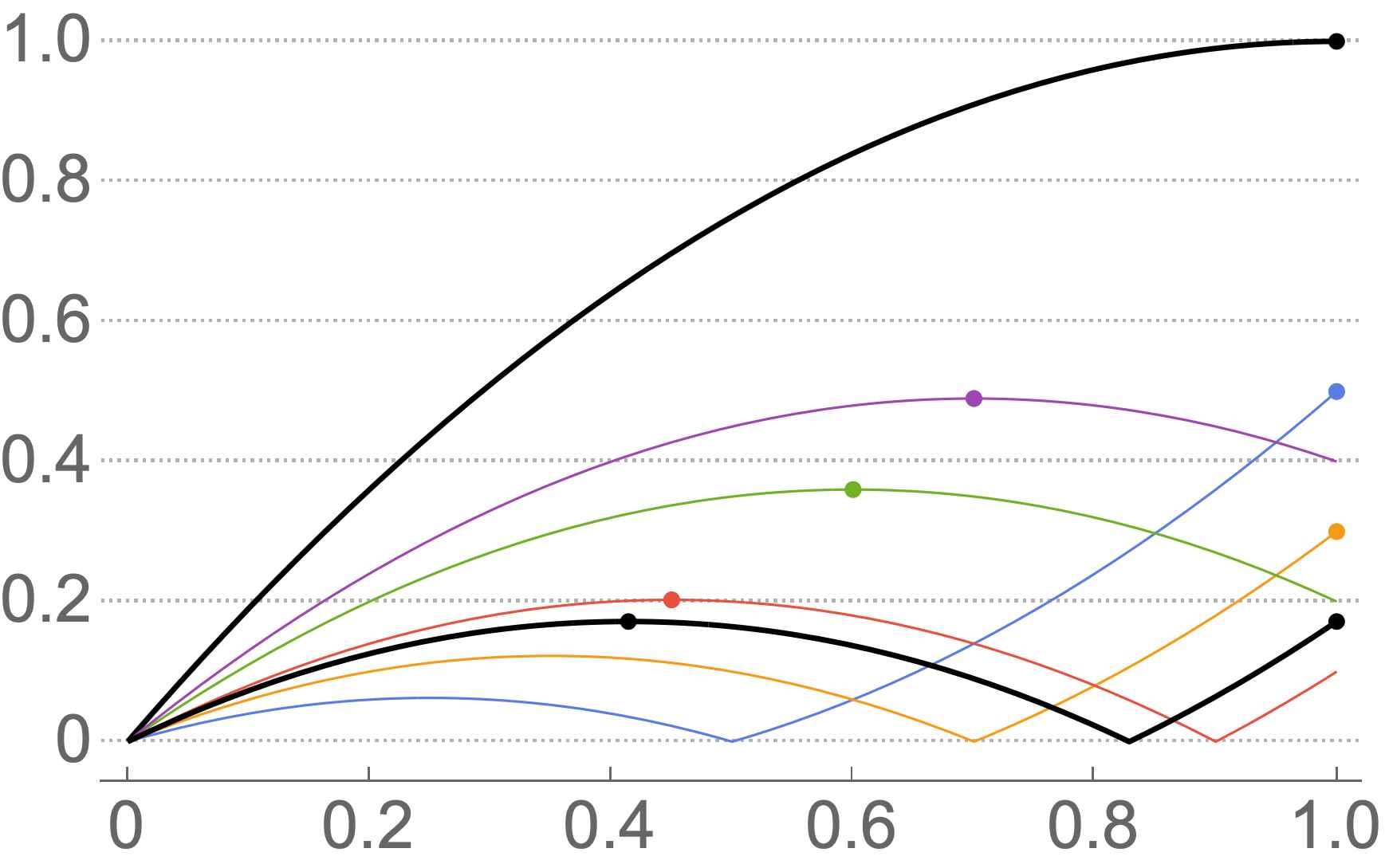}\hfill \includegraphics[width=0.32\linewidth]{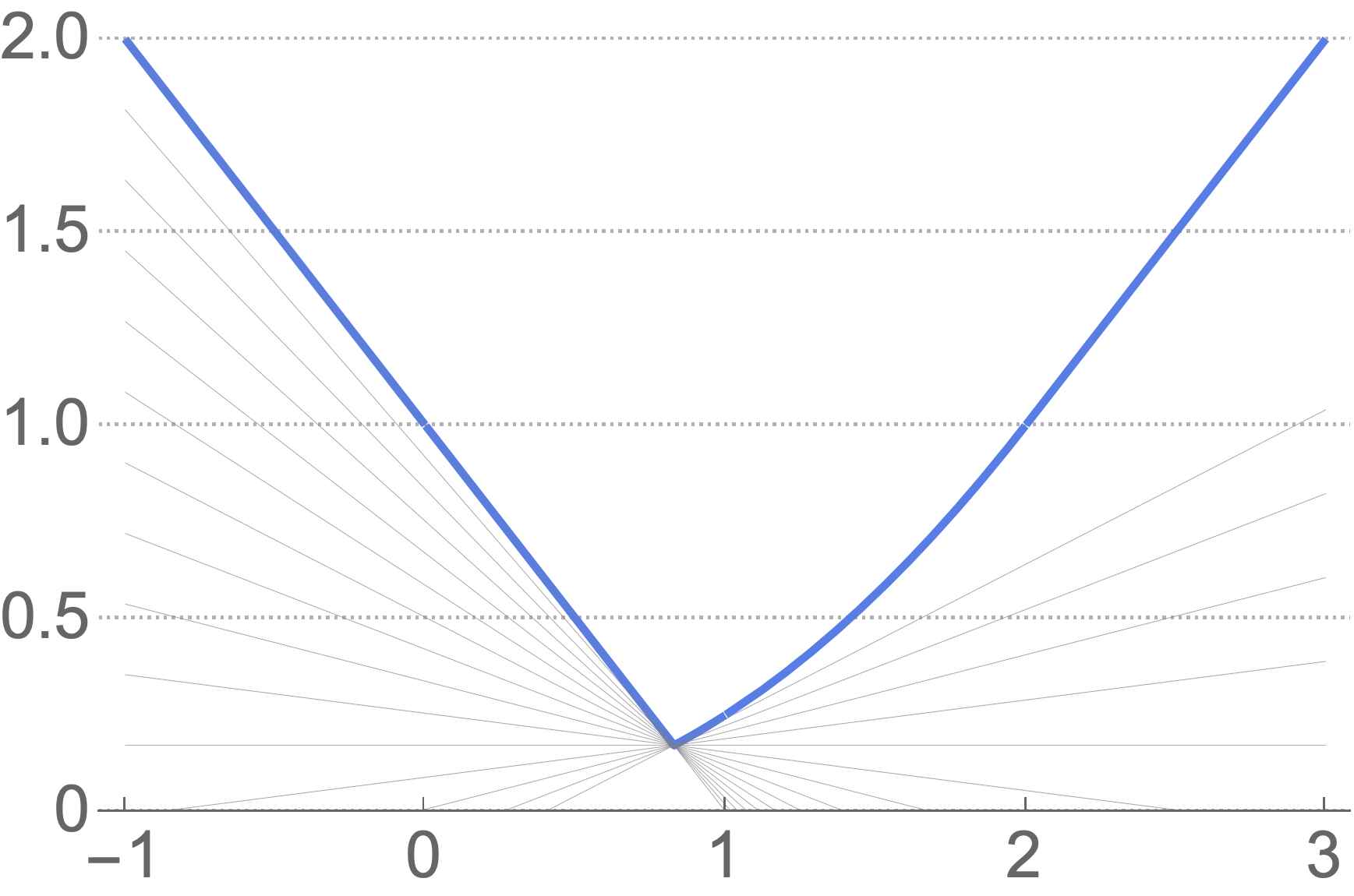}
	\caption{Left: the plot of $m(a)$ in thick blue, together with $m_x(a)$ in dashed line for several values of $x$ inside $[0,1]$ (two of them being highlighted in black, showing they are absolute values of affine functions). Middle: functions $m_x(a)$ for representative values of $a$. Right: the plot of $m(a)$ in blue, with representative affine underestimators at $a=0$.\label{fig:subdifferential-approximation}}
\end{figure}
\begin{example}[Approximation of a univariate quadratic function by a linear function]\label{ex:approximation-1D}
We consider the problem of uniformly approximating $f(x)=x^2$ by a linear function $a x$ inside the interval $X=[0,1]$. This consists in minimizing the pointwise supremum function $m(a)=\max_{x\in X}|ax-x^2|$, whose graph is shown in the left plot of Figure~\ref{fig:subdifferential-approximation}. On the one hand, for a fixed $x\in X$, the function $m_x(a)=|ax-x^2|$ is the absolute value of an affine function, and therefore convex. Such functions are broken lines displayed in dashed on the same plot for several values of $x$, and we see that the plain line graph of the pointwise maximum is indeed the upper enveloppe of the dashed broken line functions. On the other hand, for a fixed $a$ the function $x\mapsto m_x(a)$ is the absolute value of a quadratic function, and so continuous. Therefore the subgradient formula~\eqref{eq:subdifferential-ps} can be used.

Now in order to compute the subdifferential of $m(a)$, for a fixed $a$ one should compute $\max_{x\in X}|ax-x^2|$ and the corresponding maximizers, which we called the active indices. The function $x\mapsto |ax-x^2|$ is displayed on the center plot of Figure~\ref{fig:subdifferential-approximation} for different values of $a$, maxima being show by points. We can observe that functions have one maximizer, except the black function that has two maximizers. Formally, a maximizer of $x\mapsto |ax-x^2|$ subject to $x\in[0,1]$ can lie only at one bound of $[0,1]$ or at the unconstrained extrema $\tfrac{a}{2}$ when this latter is inside $[0,1]$, i.e., when $0\leq a\leq 2$. Therefore two cases arise: first, if $a<0$ or $a>2$ then the minimizer can be either $0$ or $1$, where the function evaluates to $0$ and $|a-1|$ respectively. Hence the maximizer is $x=1$ and $m(a)=|a-1|$. Second, if $0\leq a\leq 2$ then the minimizer can be either $0$ or $1$ or $\tfrac{a}{2}$, where the function evaluates to $0$, $|a-1|$ and $\tfrac{1}{4}\,a^2$ respectively. A basic investigation shows that, on the one hand, $|a-1|>\tfrac{1}{4}\,a^2$ inside $[0,\ba[$, with $\ba=-2+2\sqrt{2}\approx0.83$, hence the maximizer is at $x=1$ and the maximum is $|a-1|$. On the other hand, $|a-1|<\tfrac{1}{4}\,a^2$ inside $]\ba,2[$, hence the maximizer is at $x=\tfrac{a}{2}$ and the maximum is $\tfrac{1}{4}\,a^2$. For $a=\ba$ and $a=2$ we have $|a-1|=\tfrac{1}{4}\,a^2$ and $\act(a)=\{\tfrac{a}{2},1\}$. The function $x\mapsto |ax-x^2|$ is represented in black and gray for $a=\ba$ and $a=2$ respectively in the center plot of Figure~\ref{fig:subdifferential-approximation}. For $a=\ba$ there are two maximizers, $\act(\ba)=\{\tfrac{\ba}{2},1\}$, while for $a=2$ we have $\act(2)=\{\tfrac{2}{2},1\}=\{1\}$ and there is only one maximizer which is both stationary and on the boundary of the interval. Table~\ref{tab:approximation-1D} summarizes these computations. We see that $m(a)$ is differentiable everywhere excepted at $\ba$. Its graph is shown on the right plot of Figure~\ref{fig:subdifferential-approximation}, together with representative affine underestimators at $\ba$.
\begin{table}[b!]
\begin{center}
\begin{tabular}{|c|c|c|c|}
\hline $a<\ba$ & $\act(a)=\{1\}$ & $m(a)=1-a$ & $\partial m(a)=\{-1\}$
\\\hline $a=\ba$ & $\act(\ba)=\{\tfrac{\ba}{2},1\}$ & $m(\ba)=1-\ba=\tfrac{1}{4}\,\ba^2$ & $\partial m(\ba)=\conv\{-1,\tfrac{\ba}{2}\}$
\\\hline $\ba<a<2$ & $\act(a)=\{\tfrac{a}{2}\}$ & $m(a)=\tfrac{1}{4}\,a^2$ & $\partial m(a)=\{\tfrac{a}{2}\}$
\\\hline $a=2$ & $\act(2)=\{\tfrac{a}{2},1\}=\{1\}$ & $m(2)=1$ & $\partial m(2)=\{\tfrac{a}{2},1\}=\{1\}$
\\\hline $2<a$ & $\act(a)=\{1\}$ & $m(a)=a-1$ & $\partial m(a)=\{1\}$
\\\hline
\end{tabular}
\end{center}
\caption{Data for Example~\ref{ex:approximation-1D}: active indices, pointwise supremum value and subdifferential depending on the value of $a$.\label{tab:approximation-1D}}
\end{table}
\end{example}
The second example is similar: we now approximate $x^2$ by an affine function inside $[-1,1]$. We know that the optimal affine function is $p(x)=\tfrac{1}{2}$, for which the error function $x^2-p(x)$ is $\tfrac{1}{2}\,T_2(x)$. This is exactly Example~\ref{ex:example-1} with $n=1$.
\begin{figure}[t!]
\begin{center}
	\includegraphics[width=0.45\linewidth]{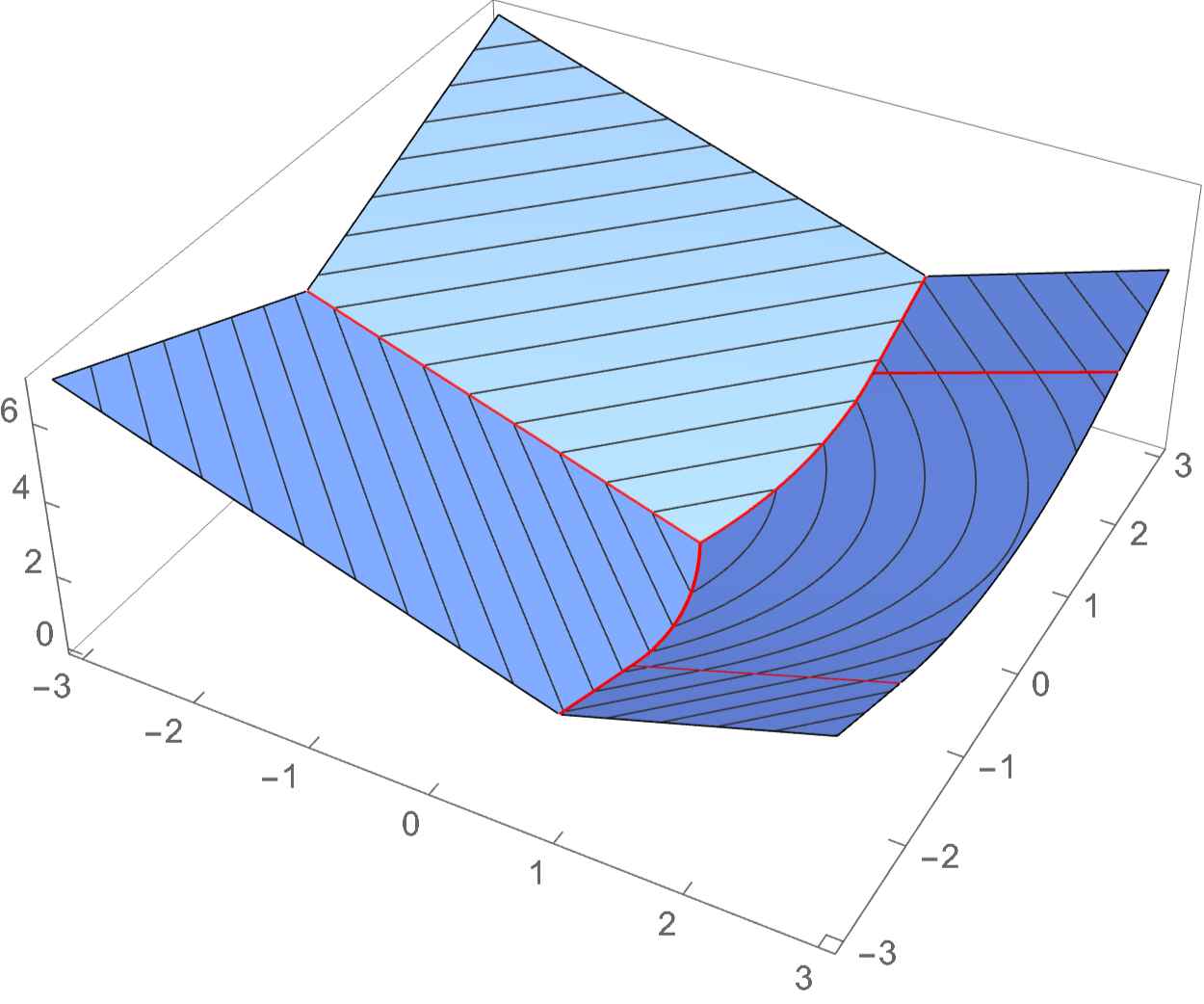}\hspace{1cm}\includegraphics[width=0.37\linewidth]{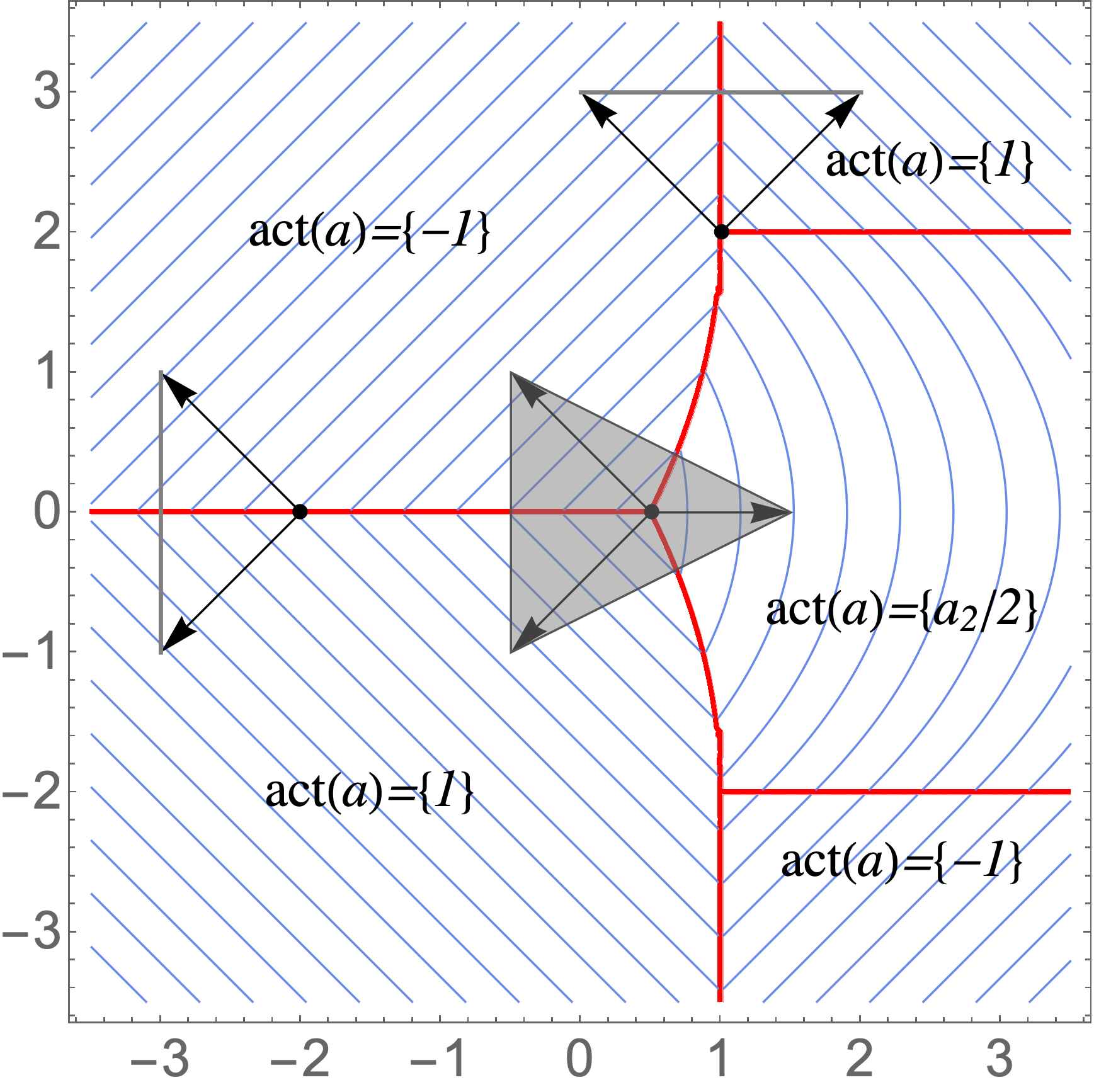}
\end{center}
	\caption{Pointwise supremum function for the second Chebyshev polynomial. Left: the plot of $m(a)$. Right: contour plot of $m(a)$ with subdifferentials at representative vectors. Red curves delimit areas with different active indices. We see well that the minimizer of $m(a)$ is $(\tfrac{1}{2},0)$, corresponding to the polynomial $\tfrac{1}{2}T_2(x)$.\label{fig:subdifferential-approximation-T2}}
\end{figure}
\begin{example}[The second Chebyshev polynomial]\label{ex:T2}
We now consider the problem of uniformly approximating $f(x)=x^2$ by an affine function $a_1+a_2 x$ inside the interval $X=[-1,1]$, and aim at computing the subdifferential of the pointwise supremum function $m(a)=\max_{x\in X}|a_1+a_2x-x^2|$. The graph of this pointwise supremum function is shown on the left plot of Figure~\ref{fig:subdifferential-approximation-T2}. For a fixed $a\in\R^2$, we need to find the maximizers of $x\mapsto |a_1+a_2x-x^2|$. As previously, the maximizers can lie only at three indices, either $x=-1$ or $x=1$ or $x=\tfrac{a_2}{2}$ if this latter lies inside $X$. A case by case study can be carried out similarly to previous example, and gives the same extremal points as in Example~\ref{ex:example-1} with $n=1$. The right plot of Figure~\ref{fig:subdifferential-approximation-T2} shows the level sets of $m(a)$ and different regions separated by red lines with their respective maximizers. Red lines are exactly where there are several maximizers. Three vectors are selected on the red lines, for each of them the subdifferential is represented.
\end{example}
The formal computation of the active indices can be carried out only in such simple situations like the previous examples, while numerical evaluation of subgradients is to be performed in practice (Taylor series and polynomial root finding algorithms~\cite{Chevillard2011} or branch-and-bound algorithms~\cite{Sahinidis1996} can be used for the computation of the uniform norm and maximizers, which ultimately rely on interval computations to produce bounds on the range of nonlinear functions~\cite{Moore2009,Goldsztejn-SIADS2011,Martin-SINUM2013}).

It is noticeable from these two examples that searching for maximizers of the function $m_a(x)=|\phi(x)^Ta-f(x)|$ in the process of computing the subdifferential of $m(a)$ is exactly searching for the extremal points of the error function $p-f$, with $p(x)=\phi(x)^Ta$, defined in the context of Chebyshev approximation:
\begin{equation}
\ext(p-f)=\act(a).
\end{equation}

%%%%%%%%%%%%%%%%%%%%%%%%%%%%%%%%%%%%%%%%%%%%%%%%%%%%%%%%%%%%%%%%%%%%%%
%%%%%%%%%%%%%%%%%%%%%%%%%%%%%%%%%%%%%%%%%%%%%%%%%%%%%%%%%%%%%%%%%%%%%%
\subsection{Optimality conditions for convex optimization}\label{ss:optimality-convex}

Standard optimality conditions and optimality conditions with sharp minimizers are presented in the following two subsections. Sharpness has been introduced by Polyak~\cite{Polyak1979,Polyak1987} in order to investigate the convergence of subgradient methods. Sharp minimizers correspond exactly to minimizers that are strongly unique in the context of Chebyshev approximation, hence the importance here of the related optimality conditions. Kernel optimality conditions, which are more practical, are presented as corollaries to the usual optimality conditions.

\subsubsection{Optimality conditions}

The subdifferential gives rise to two equivalent characterizations of optimality. Firstly, $\ba $ is a minimizer of $m(a)$ if and only if $0\in\partial m(\ba )$. This condition generalizes the necessary condition $\nabla m(a)=0$ for differentiable functions. Secondly, $\ba $ is a minimizer of $m(a)$ if and only if the directional derivative $m'(\ba,u)$ is non-negative in all directions $u$. To summarize we have:
\begin{proposition}\label{prop:optimality}
	Let $m:\R^n\rightarrow\R$ be a convex function. For an arbitrary $\ba\in\R^n$, the following three conditions are equivalent:
%	\begin{align}
%		&\forall a\in\R^n, \ m(a)\geq m(\ba);\hspace{7cm}
%		\\ &0 \in \partial m(\ba);
%		\\ &\forall u\in\R^n, \ m'(\ba,u)\geq0.
%	\end{align}
	\begin{itemize}
		\item[(1)] $\forall a\in\R^n, \ m(a)\geq m(\ba)$;
		\item[(2)] $0 \in \partial m(\ba)$;
		\item[(3)] $\forall u\in\R^n, \ m'(\ba,u)\geq0$;
	\end{itemize}
\end{proposition}
\begin{example}
	Let $m(a)=|a+1|+|a-1|$ like in Example~\ref{ex:subdifferential-ex-sumunivar}. The subdifferential expression~\eqref{eq:subdifferential-ex-sumunivar} shows that $0\in\partial m(a)$ if and only if $a\in[-1,1]$. Therefore, the set $[-1,1]$ correspond to all minimizers, as confirmed by the left graph of Figure~\ref{fig:subdifferential-sums}.
\end{example}
\begin{example}\label{ex:norm1-opt}
	Let $m(a)=|a_1|+|a_2|$ like in Example~\ref{ex:norm1}. The subdifferential expression~\eqref{eq:ex:norm1} shows that $0\in\partial m(a)$ if and only if $a=0$, which is therefore the unique minimizers. This can also be seen on the right graph of Figure~\ref{fig:subdifferential-sums}.
\end{example}
\begin{example}[The second Chebyshev polynomial]\label{ex:T2-opt}
	Continuing Example~\ref{ex:T2}, we observe on the right plot of Figure~\ref{fig:subdifferential-approximation-T2} that the only vector $a$ for which $0\in m(a)$ is $\ba=(\tfrac{1}{2},0)^T$. This is the unique minimizer of $m(a)$, which corresponds to the monic Chebyshev polynomial $\tfrac{1}{2}\,T_2(x)=\phi(x)^T\,\ba=x^2-\tfrac{1}{2}$.
\end{example}

The subdifferential $\partial m(\ba)$ may contain infinitely many subgradients, making the convex hull condition \emph{(2)} of Proposition~\ref{prop:optimality} potentially difficult to use in practice. Carathéodory's convex hull theorem shows that only $k\leq n+1$ vectors from a generating set of subgradients are sufficient, leading to the following equivalent finite convex hull optimality condition:
\begin{equation}\label{eq:finite-convex-hull-condition}
	\exists g_1,\ldots,g_{k}\in E \ , \ 0\in\conv\{g_1,\ldots,g_k\},
\end{equation}
where $E$ is a generating set of subgradients. The following kernel condition is a simple rewriting of this finite convex hull conditions:
\begin{corollary}\label{cor:optimality}
Let $E$ be a generating set of subgradients. The optimality conditions of Proposition~\ref{prop:optimality} are equivalent to
\begin{equation}\label{eq:kernel-condition}
	\exists g_1,\ldots,g_{k}\in E \ , \ \exists \,u(\neq0)\in(\R_{\geq0})^k \ , \ G\,u=0,\hspace{2cm}
\end{equation}
where $G=\big( \, g_1 \ \cdots \ g_k \, \bigr)$. The number $k$ of subgradients can be chosen less or equal to $n+1$.
\end{corollary}
Optimality is therefore expressed as finding a nonnegative kernel vector of a matrix whose columns are subgradients.

\subsubsection{Optimality conditions for sharp minimizers}\label{sss:sharp}

A sharp minimizer~\cite{Polyak1979,Polyak1987} $\ba$ satisfies $m(a)\geq m(\ba)+r\|a-\ba\|$ for all $a\in\R^n$ for an arbitrary but fixed $r>0$, and is therefore the unique minimizer. For $r=0$, the definition reduces to standard optimality. The following proposition is a stronger version of Proposition~\ref{prop:optimality}, which provides optimality conditions for sharp minimizers. Its statement is slightly improved with respect to \cite[Lemma~3 page 137]{Polyak1987} and is folklore for experts of the field. The proof is provided for completeness and to emphasize how simple sharpness is added to optimality in the context of convex optimization. The case $r=0$ is included in the statement, and corresponds to standard optimality.
\begin{proposition}\label{prop:strong-uniqueness}
	Let $m:\R^n\rightarrow\R$ be a convex function. For arbitrary $\ba\in\R^n$ and $r\geq0$, the following three conditions are equivalent
	\begin{itemize}
		\item[(1)] $\forall a\in\R^n,\ m(a)\geq m(\ba)+r\,\|a-\ba\|_2$;
		\item[(2)] $r\,\B^n \subseteq \partial m(\ba)$, where $\B^n$ is the $n$ dimensional Euclidean ball;
		\item[(3)] $\forall u\in\R^n,\ m'(\ba,u)\geq r\|u\|_2$.
	\end{itemize}
\end{proposition}
\begin{proof}
$(2)\Rightarrow(1)$: For all subgradient $g\in B_r$ we have $m(a)\geq m(\ba)+g^T(a-\ba)$. Therefore, $m(a)\geq m(\ba)+\max_{g\in B_r}g^T(a-\ba)=m(\ba)+\|\bg\|_2\|a-\ba\|_2=m(\ba)+r\,\|a-\ba\|_2$ for $\bg$ radius $r$ aligned with $(a-\ba)$, using Cauchy-Schwarz inequality. $(1)\Rightarrow(3)$: we have $m'(\ba,u)=\lim_{t\rightarrow 0^+}\tfrac{1}{t}\bigl(m(\ba+tu)-m(\ba)\bigr)\geq \lim_{t\rightarrow 0^+}\tfrac{1}{t}\bigl(m(\ba)+r\,\|tu\|_2-m(\ba)\bigr)=r\|u\|_2$. $\neg(2)\Rightarrow\neg(3)$: Let $\bg\in B_r$ with $\bg\notin \partial m(\ba)$. Since $\partial m(\ba)$ is compact, the separating hyperplane theorem proves the existence of $u\in\R^n$ such that $\forall g\in\partial m(\ba),u^Tg<u^T\bg$. Finally, $m'(\ba,u)=\max_{g\in\partial m(\ba)}u^Tg<u^T\bg\leq \|u\|_2\,\|\bg\|_2\leq r\,\|u\|_2$, which is $\neg(3)$.
\end{proof}
\begin{figure}[t!]
\begin{center}
	\includegraphics[width=0.2\linewidth]{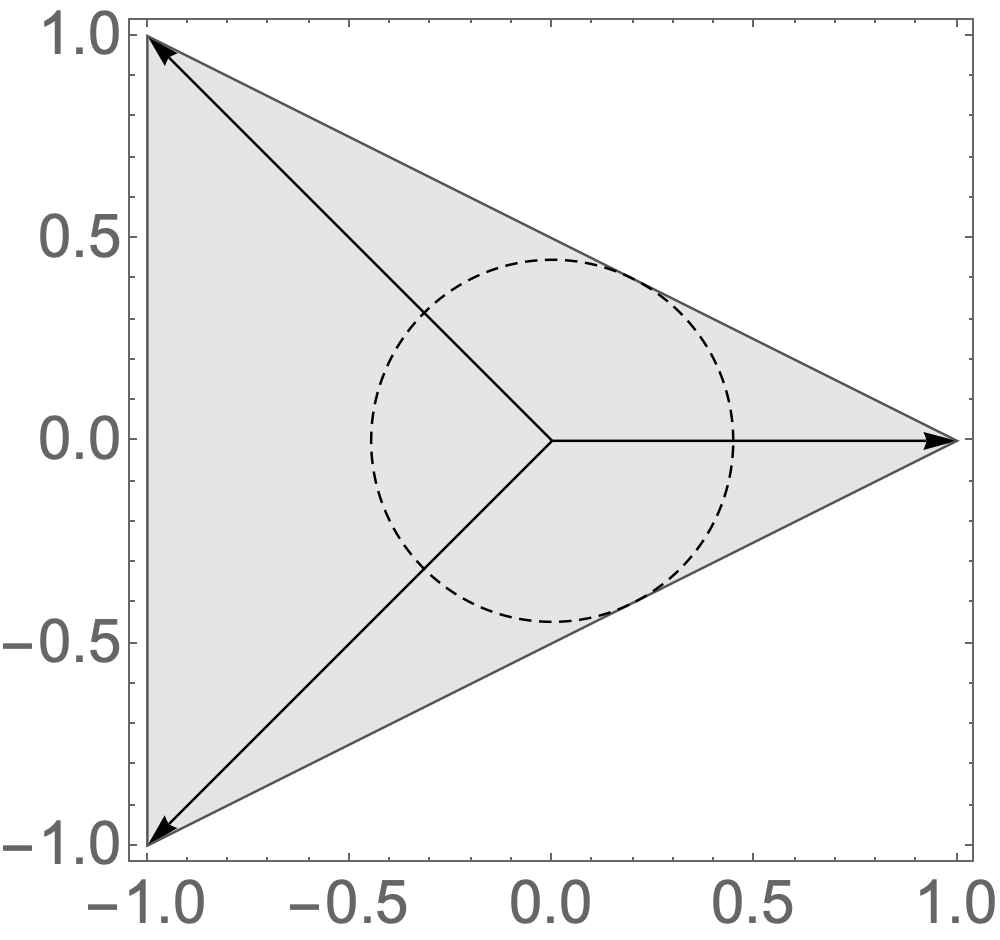}\hfill\includegraphics[width=0.52\linewidth]{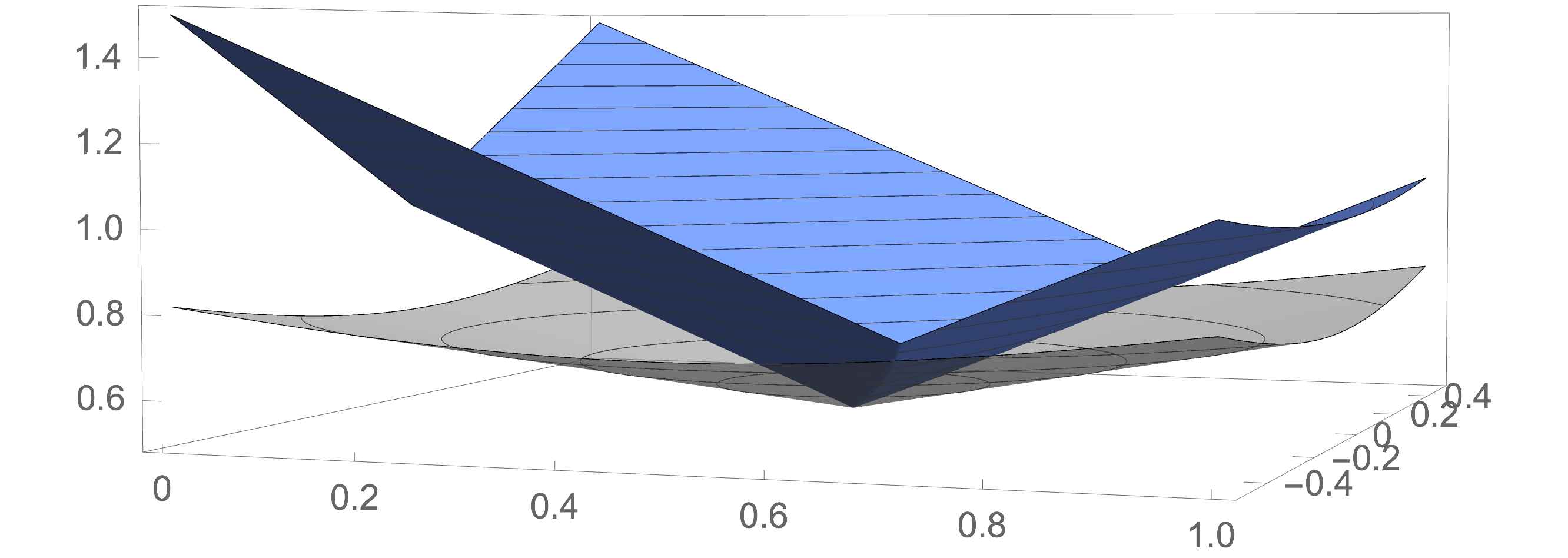}\hfill\includegraphics[width=0.27\linewidth]{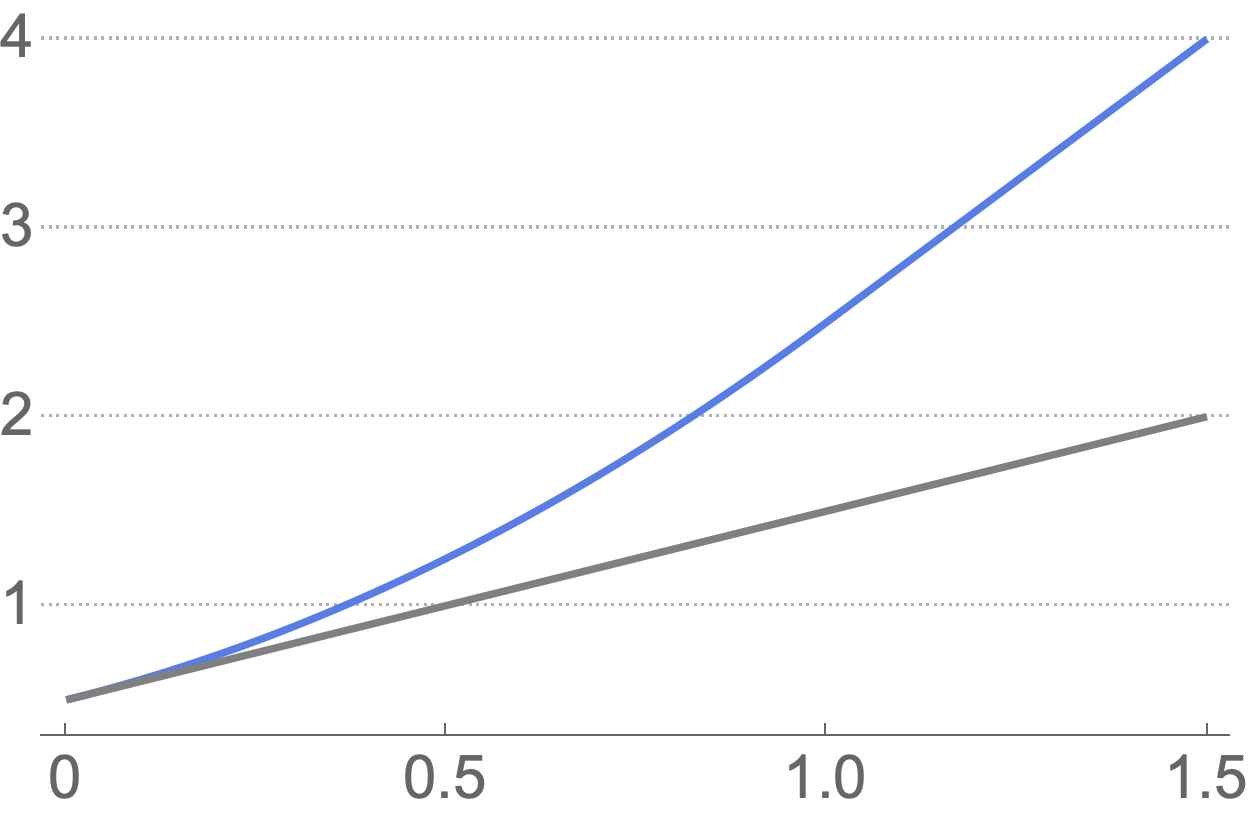}
\end{center}
	\caption{From left to right: the largest ball contained in the subdifferential of Example~\ref{ex:T2-sharp}, the corresponding lower bound displayed in $\R^2$ and in the direction $(1,2)^T$.\label{fig:sharp-T2}}
\end{figure}
\begin{example}[The second Chebyshev polynomial]\label{ex:T2-sharp}
	Continuing Example~\ref{ex:T2-opt}, let $m(a)=\max_{x\in[-1,1]}|a_1+a_2x-x^2|$. The vector $\ba=(\tfrac{1}{2},0)$ was identified to be a minimizer. Its subdifferential $\partial m(\ba)$ is represented on the left plot of Figure~\ref{fig:sharp-T2}. The largest ball included inside the subdifferential is shown in dashed, and has radius $1/\sqrt{5}$. Hence Proposition~\ref{prop:strong-uniqueness} proves that $\tfrac{1}{2}+1/\sqrt{5}\,\|x\|_2$ is less than $m(a)$ for all $a$. Both functions are represented on the central plot of Figure~\ref{fig:sharp-T2}. We can see that $1/\sqrt{5}$ is the largest such constant on the right plot of Figure~\ref{fig:sharp-T2}, where both functions displayed in the direction $(1,2)^T$ are seen to be tangent to each other.
\end{example}

\paragraph{The kernel condition for sharp minimizers}

As previously, the subdifferential $\partial m(\ba)$ may contain infinitely many subgradients, making the strict convex hull condition \emph{(2)} of Proposition~\ref{prop:strong-uniqueness} potentially difficult to use in practice. For such a strict convex hull condition, one uses Steinitz's theorem~\cite[Theorem 3.26 page 136]{Soltan2015}, instead of Carathéodory's theorem, to prove that $k\leq 2n$ subgradients from a generating set of subgradients are sufficient, leading to the following equivalent finite strict convex hull condition:
\begin{equation}\label{eq:finite-convex-hull-strict-uniqueness}
	\exists g_1,\ldots,g_k\in E \ , \ 0\in\interior\conv\{g_1,\ldots,g_k\},
\end{equation}
where $E$ is a generating set of subgradients. Note that now $2n$ vectors may be necessary to have zero in the interior of the convex hull, as compared to $n+1$ for Carathéodory's theorem, see~\cite[figure at top of page 136]{Soltan2015} and Figure~\ref{fig:subdifferential-sums}, where no subset of $3$ vectors have $0$ in the interior of their convex hull, hence $4=2n$ vectors are necessary for having $0$ in the interior of the convex hull. Note also that the radius information $r$ is lost in the condition~\eqref{eq:finite-convex-hull-strict-uniqueness}, although some quantitative Steinitz theorem~\cite{Ivanov2024} can bring related information. Therefore, the condition~\eqref{eq:finite-convex-hull-strict-uniqueness} is equivalent to $\ba$ is a strong minimizer.

The corresponding kernel condition is not anymore an obvious equivalence of the finite strict convex hull condition~\eqref{eq:finite-convex-hull-strict-uniqueness}, like~\eqref{eq:kernel-condition} was obviously equivalent to~\eqref{eq:finite-convex-hull-condition}: here we need to use Stiemke's alternative theorem, which asserts that for an arbitrary matrix $G=\bigl( \, g_1 \ \cdots \ g_k \, \bigr)\in\R^{n\times k}$ we have
\begin{equation}
	\Bigl( \ \forall v\in\R^n \, ,\, \bigl( G^Tv\leq 0\Rightarrow G^Tv=0\bigr) \ \Bigr) \ \iff \ \exists u\in(\R_{>0})^k,G\,u=0.
\end{equation}
The left side of the equivalence says that the directional derivative of the linear pointwise maximum $l(u)=\max\{g_1^Tu,\ldots,g_k^Tu\}$ is nonnegative in all directions $0\neq v\in\R^n$. The right side is a strict kernel condition, where the kernel vector has positive components, instead of nonnegative components in~\eqref{eq:kernel-condition}. Informally, Stiemke's alternative theorem says that subgradients points toward all directions, i.e., $\sum u_ig_i=0$ for some $u_i>0$, if and only if the directional derivative of $l(u)$ is non-negative in all directions $0\neq u\in\R^n$. The fact that $G$ is full rank must be deduced to conclude formally: by Proposition~\ref{prop:strong-uniqueness} the condition~\eqref{eq:finite-convex-hull-strict-uniqueness} is equivalent to $l(u)$ having a positive directional derivative in all directions. With $G=\big( \, g_1 \ \cdots \ g_k \, \bigr)\in\R^{n\times k}$, this condition can be expressed by $\forall v\in\R^n, G^Tv\leq0\Rightarrow v=0$. This implies that $G$ is full rank, so we have the following equivalent condition: G full rank and $\forall v\in\R^n, G^Tv\leq0\Rightarrow G^Tv=0$. Finally, applying Stiemke's alternative theorem give the following equivalent kernel characterization of optimality with strong uniqueness:

\begin{corollary}\label{cor:sharp-optimality}
Let $E$ be a generating set of subgradients. The sharp optimality conditions of Proposition~\ref{prop:strong-uniqueness} hold for some $r>0$ if and only if
\begin{equation}\label{eq:kernel-condition-strong-uniqueness}
	\exists g_1,\ldots,g_{k}\in E \ , \ \exists \,u\in(\R_{>0})^k \ , \ G\text{ full rank and }G\,u=0,
\end{equation}
where $G=\big( \, g_1 \ \cdots \ g_k \, \bigr)\in\R^{n\times k}$. The number $k$ of subgradients is greater than $n$ (because $G$ is full rank with a nontrivial kernel) and can be chosen less or equal to $2n$.
\end{corollary}
The following example illustrates the subtlety of Steinitz's theorem and Stiemke's alternative theorem.
\begin{example}
	Continuing Example~\ref{ex:norm1-opt}, let $m(a)=|a_1|+|a_2|$. We have $\partial m(0)=\conv E$ with $E=\{(-1,-1)^T,(-1,1)^T,(1,-1)^T,(1,1)^T\}$. Subgradient matrices with two columns
	\begin{equation}
		\begin{pmatrix}-1&1\\-1&1\end{pmatrix}\text{ and }\begin{pmatrix}-1&1\\1&-1\end{pmatrix}
	\end{equation}
	have a positive kernel vector, which proves optimality by Corollary~\ref{cor:optimality}, but they are not full rank hence do not allow applying Corollary~\ref{cor:sharp-optimality}. Subgradient matrices with three subgradients have a kernel vector with one zero component, hence again do not allow applying Corollary~\ref{cor:sharp-optimality}. This is finally the subgradient matrix with all four subgradients that is full rank and has a kernel vector $(1,1,1,1)^T$ with positive components.	This is a typical situation where $2\leq n+1$ vectors are enough to prove that $0$ is in the convex hull, in coherence with Carathéodory theorem, while $4=2n$ vectors are necessary to prove that $0$ is in the interior of the convex hull, in coherence with Steinitz's theorem.
\end{example}

%%%%%%%%%%%%%%%%%%%%%%%%%%%%%%%%%%%%%%%%%%%%%%%%%%%%%%%%%%%%%%%%%%%%%%
%%%%%%%%%%%%%%%%%%%%%%%%%%%%%%%%%%%%%%%%%%%%%%%%%%%%%%%%%%%%%%%%%%%%%%
\subsection{Subdifferential computation for uniform approximation problems}\label{ss:Chebyshev-convex}

We now compute the subdifferential of Chebyshev approximation problems and relative Chebyshev centers in a homogeneous way. The corresponding conditions for optimality will be given in Section~\ref{ss:convex-Chebyshev-approximation} below.

%%%%%%%%%%%%%%%%%%%%%%%%%%%%%%%%%%%%%%%%%%%%%%%%%%%%%%%%%%%%%%%%%%%%%%
\subsubsection{Subdifferential of Chebyshev approximation problems}\label{s:subdifferential-Chebyshev}

We now consider the problem consisting in minimizing $m(a)=\|\phi^Ta-f\|_\infty$, where $\phi^Ta-f$ maps $x\in X$ to $\phi(x)^Ta-f(x)=p(x)$. This is a pointwise supremum with $m_x(a)=|\phi(x)^Ta-f(x)|$. The function $m_x(a)$ is convex for all $x\in X$, and the function $x\mapsto m_x(a)$ is continuous for all $a\in\R^n$, so we can apply the pointwise supremum formula~\eqref{eq:subdifferential-ps}. To this end, we need to evaluate $\partial m_x(a)$ for $x\in\act(a)=\ext(p-f)$. Since we assumed that $m(a)$ is positive for all $a\in\R^n$, so is $m_x(a)$ at active indices, and we have $m_x(a)=|\phi(x)^Ta-f(x)|=\sigma(a,x)(\phi(x)^Ta-f(x))$, where $\sigma(a,x)=\sign(\phi(x)^Ta-f(x))\in\{-1,1\}$ is constant in a small enough neighborhood of $a$. Therefore $m_x(a)$ is differentiable, with $\nabla m_x(a)=\sigma(a,x)\phi(x)$. Finally, the subdifferential calculus rule for pointwise supremum~\eqref{eq:subdifferential-ps-diff} leads to
\begin{subequations}\label{eq:subdifferential-Chebyshev-approximation}
\begin{align}
\partial \|\phi^Ta-f\|_\infty &= \conv \{\nabla m_x(a):x\in\act(a)\}
\\ &=\conv \{\sigma(a,x)\,\phi(x):x\in\act(a)\}
\\ &=\conv \{s\,\phi(x):(x,s)\in\sig(\phi^Ta-f)\}.
\end{align}
\end{subequations}
and we have finally
\begin{equation}
	\partial \|\phi^Ta-f\|_\infty =\conv \{s\,\phi(x):(x,s)\in\sig(\phi^Ta-f)\}\label{eq:subdifferential-Chebyshev-approximation-3}.
\end{equation}
The vectors $s\,\phi(x)$, for $(x,s)$ in the signature $\sig(\phi^Ta-f)$, form a generating set of subgradients for the uniform norm. Note that the subdifferential is computed with respect to the basis coordinates $a\in\R^n$, hence depend on the basis functions $\phi_i$.

%%%%%%%%%%%%%%%%%%%%%%%%%%%%%%%%%%%%%%%%%%%%%%%%%%%%%%%%%%%%%%%%%%%%%%
\subsubsection{Subdifferential of relative Chebyshev center problems}\label{sss:subdifferentitial-center}

We now consider the problem consisting in minimizing $m(a)=\|\phi^Ta-\F\|_\infty$. As explained in Section~\ref{ss:def-center}, under the assumption that $\F$ is totally complete, we have $\|\phi^Ta-\F\|_\infty=\max\{\|\phi^Ta-f^-\|_\infty,\|\phi^Ta-f^+\|_\infty\}$, both functions $f^-,f^+:X\rightarrow\R$ being continuous. The subdifferential of $\|\phi^Ta-\F\|_\infty$ can therefore be computed from the rule for pointwise maximum the two functions $\|\phi^Ta-f^-\|_\infty$ and $\|\phi^Ta-f^+\|_\infty$: with $p=\phi^Ta$ for compactness,
\begin{equation}
	\partial\|p-\F\|_\infty=\left\{\begin{array}{cl}
		\partial \|p-f^-\|_\infty & \text{if } \|p-f^-\|_\infty>\|p-f^+\|_\infty
		\\\partial \|p-f^+\|_\infty & \text{if } \|p-f^-\|_\infty<\|p-f^+\|_\infty
		\\\conv(\partial\|p-f^-\|_\infty\cup\partial\|p-f^+\|_\infty)& \text{otherwise.}
	\end{array}\right.
\end{equation}
In the first two cases, we have $\partial\|p-f^\pm\|_\infty=\conv\{s\,\phi(x):(x,s)\in\sig(p-f^\pm)\}$ using~\eqref{eq:subdifferential-Chebyshev-approximation-3}, which are both equal to $\conv\{s\,\phi(x):(x,s)\in\sig(p-\F)\}$ by~\eqref{eq:sig-center-explicit}. For the third case, using the fact that $\conv((\conv A)\cup(\conv B))=\conv(A\cup B)$ we have
\begin{subequations}
\begin{align}
	\conv(\partial\|p-f^-\|_\infty&\cup\partial\|p-f^+\|_\infty)
	\\=&\conv\{s\,\phi(x):(x,s)\in\sig(p-f^\pm)\cup\sig(p-f^\pm)\}
	\\=&\conv\{s\,\phi(x):(x,s)\in\sig(p-\F)\},
\end{align}
\end{subequations}
using~\eqref{eq:subdifferential-Chebyshev-approximation-3} and~\eqref{eq:sig-center-explicit} like in the first two cases. Finally, all cases agree to the following formula:
\begin{equation}\label{eq:subdifferential-rcc}
	\partial\|\phi^Ta-\F\|_\infty=\conv\{s\,\phi(x):(x,s)\in\sig(\phi^Ta-\F)\},
\end{equation}
to which~\eqref{eq:subdifferential-Chebyshev-approximation-3} is a special case with $\F=\{f\}$.

%%%%%%%%%%%%%%%%%%%%%%%%%%%%%%%%%%%%%%%%%%%%%%%%%%%%%%%%%%%%%%%%%%%%%%
%%%%%%%%%%%%%%%%%%%%%%%%%%%%%%%%%%%%%%%%%%%%%%%%%%%%%%%%%%%%%%%%%%%%%%
\subsection{Optimality conditions for uniform approximation}\label{ss:convex-Chebyshev-approximation}

Using the subdifferential computed in the previous section, each optimality condition for convex programing from Section~\ref{ss:optimality-convex} gives rise to an optimality condition for uniform approximation. They are presented here for set-valued functions $\F$ that are totally complete, i.e., for relative Chebyshev center under the assumption of total completeness, and optimality conditions for standard uniform approximation of a real function $f$ correspond exactly to $\F(x)=\{f(x)\}$.

%%%%%%%%%%%%%%%%%%%%%%%%%%%%%%%%%%%%%%%%%%%%%%%%%%%%%%%%%%%%%%%%%%%%%%
\subsubsection{Directional derivative condition}
Using the subdifferential expression of the directional derivative~\eqref{eq:dirder-general} and the subdifferential~\eqref{eq:subdifferential-rcc} we have
\begin{subequations}
\begin{align}
	m'(\ba,u)&=\max\{(s\,\phi(x))^Tu:(x,s)\in\sig(\bp-\F)\}
	\\&=\max\{s\,(\phi(x)^Tu):(x,s)\in\sig(\bp-\F)\}
\end{align}
\end{subequations}
with $\bp=\phi^T\ba$. So we can interpret the search direction $u$ as the coordinates of a generalized polynomial $p=\phi^Tu$ in the basis $\phi$ and write
\begin{equation}\label{eq:dirder-uniform}
	m'(\bp,p):=\max\{s\,p(x):(x,s)\in\sig(\bp-\F)\}.
\end{equation}
Then the equivalence between the Proposition~\ref{prop:strong-uniqueness} first condition (strong optimality) and the third condition (non-negative directional derivative) gives rise to the optimality condition in the following theorem, where $\|p\|_{\phi,2}$ is the $2$-norm of the vector of the coefficients of $p$ in the basis $\phi$.
\begin{theorem}\label{thm:convex-Chebyshev-approximation-dirder}
Let $\F:X\rightarrow2^\R$ be a totally continuous set-valued function and $r\geq0$. The generalized polynomial $\bp$ satisfies
\begin{equation}
	\forall p\in\V \ , \ \|p-\F\|_\infty\geq\|\bp-\F\|_\infty+r\|\bp-p\|_{\phi,2}
\end{equation}
if and only if
\begin{equation}
\forall p\in\V \ , \ \max\{s\,p(x):(x,s)\in\sig(\bp-\F)\} \ \geq \ r\,\|p\|_{\phi,2},
\end{equation}
\end{theorem}
For $r=0$ this is exactly Kolmogorov criterion (Theorem~\ref{thm:Kolmogorov}), while for $r>0$ this is Bartelt's condition (Theorem~\ref{thm:Bartelt} and Equation~\eqref{eq:Bartelt-sig}) but with the $2$-norm of the coefficients instead of the uniform norm, here both extended to relative Chebyshev centers.

Equation~\eqref{eq:dirder-uniform} reveals a connection between Kolmogorov criterion and the directional derivative of the uniform norm: the polynomial $p$ in Kolmogorov criterion is interpreted as a search direction, and Kolmogorov criterion says exactly that the directional derivative is non-negative in all search directions.
\begin{remark}\label{rem:Kolmogorov}
In fact, the quantity
\begin{equation}
	\max_{x\in\ext(\bp-f)}\bigl(\,\bp(x)-f(x)\bigr)\,p(x)
\end{equation}
involved in Kolmogorov's criterion~\eqref{eq:Kolmogorov-max} is exactly the directional derivative of the pointwise supremum function $m(\bp)=\tfrac{1}{2}\,\max_{x\in X}(\bp(x)-f(x))^2$, which has the same optimal generalized polynomials, in the direction $p$.
\end{remark}

%%%%%%%%%%%%%%%%%%%%%%%%%%%%%%%%%%%%%%%%%%%%%%%%%%%%%%%%%%%%%%%%%%%%%%
\subsubsection{Zero in the convex hull condition}

Using the subdifferential expression~\eqref{eq:subdifferential-rcc}, the equivalence between the Proposition~\ref{prop:strong-uniqueness} first condition (strong optimality) and the second condition (zero in the convex hull of the subdifferential) gives rise to the optimality condition in the following theorem:
\begin{theorem}\label{thm:convex-Chebyshev-approximation-subdifferential}
Let $\F:X\rightarrow2^\R$ be a totally continuous set-valued function and $r\geq0$. The generalized polynomial $\bp$ satisfies
\begin{equation}
	\forall p\in\V \ , \ \|p-\F\|_\infty\geq\|\bp-\F\|_\infty+r\|\bp-p\|_{\phi,2}.
\end{equation}
if and only if
\begin{equation}
	r\,\B^n \ \subseteq \ \conv\bigl\{s\,\phi(x): (x,s)\in\sig(\bp-\F)\bigr\}.
\end{equation}
\end{theorem}
For $r=0$ this is exactly Cheney's condition replacing the sign of the error $s$ by the error itself $\bp(x)-f(x)$, leading to homothetic sets $\conv\bigl\{s\,\phi(x): (x,s)\in\sig(\bp-f)\bigr\}$ and $\conv\bigl\{(\bp(x)-f(x))\,\phi(x):x\in\ext(\bp-f)\bigr\}$ and to equivalent zero in the convex hull conditions. For $r>0$, this is a quantitive improvement on Bartelt's quantitative condition given in Theorem~\ref{thm:Bartelt-convexhull}, here extended to relative Chebyshev centers.

%%%%%%%%%%%%%%%%%%%%%%%%%%%%%%%%%%%%%%%%%%%%%%%%%%%%%%%%%%%%%%%%%%%%%%
\subsubsection{Kernel condition}\label{sss:subdifferential-kernel}

The following theorem instantiates Corollary~\ref{cor:sharp-optimality} with the subdifferential of the uniform norm~\eqref{eq:subdifferential-rcc}.
\begin{theorem}\label{thm:convex-Chebyshev-approximation}
The generalized polynomial $\bp$ is a best uniform approximation of the totally continuous set-valued function $\F:X\rightarrow2^\R$ if and only if there exists a finite signature $\{(x_1,s_1),\ldots,(x_k,s_k)\}\subseteq\Sigma(\bp-\F)$ such that
\begin{equation}\label{eq:subgradient-condition}
\exists u(\neq0)\in\R_{\geq0}^k, \ S\,u=0,
\end{equation}
where $S=\begin{pmatrix}s_1\phi(x_1)&s_2\phi(x_2)&\cdots&s_k\phi(x_k)\end{pmatrix}\in\R^{n\times k}$ is a matrix whose columns are subgradients. The integer $k$ can be chosen less or equal to $n+1$. Furthermore, $\bp$ is strongly unique if and only if $\{(x_1,s_1),\ldots,(x_k,s_k)\}$ and $u$ can be chosen so that $S$ is full rank and $u\in\R_{>0}^k$. In this case, the integer $k$ is greater than $n$ and can be chosen less or equal to $2n$.
\end{theorem}
Note that the subgradient matrix is the same as the matrix in Smarzewski's condition, both being denoted by $S$. Hence Theorem~\ref{thm:convex-Chebyshev-approximation} is exactly Smarzewski's condition extended to relative Chebyshev centers. The following example shows how subgradients provide an insightful interpretation of Smarzewski's condition, in addition to generalizing the condition to relative Chebyshev centers.

\begin{example}[Continuing Example~\ref{ex:example-1}]
	Considering again the uniform approximation of $f(x)=x^Tx$ inside $\B^m$ by affine functions, in Example~\ref{ex:example-1-Smarzewski} the signature $\{(0,1),(e_1,-1),(-e_1,-1)\}$ leaded to the matrix
	\begin{equation}
	S=\begin{pmatrix}
		1&-1& -1
		\\0&-e_1&e_1
	\end{pmatrix}\in\R^{(m+1)\times 3},
	\end{equation}
	which proved optimality but not strong optimality. Theorem~\ref{thm:convex-Chebyshev-approximation} leads to the same conclusion, but the interpretation of the matrix columns as subgradients puts some light on the situation: the three subgradients give rise to a piecewise affine underestimator
	\begin{equation}
		l(a)=\max\{\tfrac{1}{2}+S_{:1}^T(a-\ba),\tfrac{1}{2}+S_{:2}^T(a-\ba),\tfrac{1}{2}+S_{:3}^T(a-\ba)\},
	\end{equation}
	where $S_{:i}$ is the $i^\mathrm{th}$ column of $S$. Now since $S$ has a non negative kernel vector, Gordan's alternative theorem shows that the directional derivative of $l$ is non-negative in all directions\footnote{Here we can see that directly: by way of contradiction, assume $S_{:i}^Tu<0$ for the three columns of $S$. The scalar product with the first column entails $u_1<0$. The scalar product with the second and third columns that $-u_1\pm\sum_{2\leq i\leq m}u_i<0$, which cannot be both true.}, hence $\ba$ is optimal. However, kernel vectors of $S^T$ have a zero directional derivative, so this affine underestimator does not allow proving strong uniqueness.
	
	The second signature of Example~\ref{ex:example-1-Smarzewski} is~\eqref{eq:ex:example-1-Smarzewski-signature} leading to $m+2$ subgradients that are the columns of the matrix $S$ in~\eqref{eq:ex:example-1-Smarzewski-S}. The corresponding affine underestimator has a positive directional derivative in all directions by Stiemke’s alternative, hence proving strong minimality.
	\end{example}

\begin{remark}\label{rem:annihilating-subdifferential}
Simple computations show that annihilating measure conditions (Rivlin and Shapiro's condition, Tanimoto's condition and Levis et al.'s condition) are closely related to the subgradient condition: the annihilating measure condition is that the generalized polynomial $\bp(x)=\phi(x)^Ta$ has a signature $\{(x_1,s_1),\ldots,(x_m,s_m)\}\subseteq \sig(\bp-\F)$ for which there exist $c_1,\ldots,c_m>0$ such that
\begin{equation}
\forall a\in\R^n \ , \ \sum_{i=1}^m \ c_i \, s_i \, \bigl(\phi(x_i)^Ta\bigr) = 0.
\end{equation}
Now, by associativity and distributivity one has
\begin{equation}
	\sum_{i=1}^m \ c_i \, s_i \, \bigl(\phi(x_i)^Ta\bigr)= \Bigl(\ \sum_{i=1}^m \ c_i \, s_i \, \phi(x_i)\ \Bigr)^Ta = \bigl( \ S\,c \ \bigr)^Ta,
\end{equation}
which is null for all $a\in\R^n$ if and only if $S\,c=0$. Hence, the measure weights $c_i$ in the annihilating measure condition are exactly the components of the kernel vector in the subgradient kernel condition. This is related to the dual nature of the annihilating measure condition, the components of the kernel vector of the subgradient matrix being exactly the dual variables of the linear problem consisting of minimizing the linear underestimator.
\end{remark}

The formalism of set-valued error functions and their signatures provides an homogeneous presentation of the optimality condition for uniform approximation of real-valued functions and relative Chebyshev centers. It also put lights on the fact that the signature of a set-valued error function may contain twice the same extremal point with opposite signs. As explained in Section~\ref{sss:Tanimoto}, this is a sufficient condition for optimality. This situation seemed rather delicate in the context of annihilating measures of Tanimoto's condition, but is clearly explained by subgradients: if $(x,-1)$ and $(x,1)$ belong to the signature $\sig(\phi^T\ba-f)$ then $\phi(x)$ and $-\phi(x)$ are two subgradients and the piecewise affine underestimator $\max\{m(\ba)-\phi(x)(a-\ba),m(\ba)+\phi(x)(a-\ba)\}$ alone entails optimality of $\ba$.

%%%%%%%%%%%%%%%%%%%%%%%%%%%%%%%%%%%%%%%%%%%%%%%%%%%%%%%%%%%%%%%%%%%%%%
%%%%%%%%%%%%%%%%%%%%%%%%%%%%%%%%%%%%%%%%%%%%%%%%%%%%%%%%%%%%%%%%%%%%%%
%%%%%%%%%%%%%%%%%%%%%%%%%%%%%%%%%%%%%%%%%%%%%%%%%%%%%%%%%%%%%%%%%%%%%%
\section{Numerical applications}\label{s:numerical-applications}

Three cases of uniform approximation problems are solved numerically by the classical two-step method: first the classical multivariate Runge function is approximated for different dimensions and degrees in Subsection~\ref{ss:Runge}. Second, the two dimensional inverse geometric model of the DexTAR parallel robot is approximated uniformly in a singularity-free workspace by polynomials of different degrees in Subsection~\ref{ss:DexTAR}. Finally, the problem consisting in minimizing the approximation error together with the polynomial evaluation error, which was recently investigated~\cite{Arzelier2025}, is solved relying on an optimality condition derived from a simple subdifferential computation.

The classical two-step approach~\cite{Hettich1976} is used the three cases. It consists first in solving a finite discretization of the uniform problem and second in applying a local Newton method to a local version of the optimality conditions. This simple approach is sensitive to the initial discretization: it has to be thin enough so that all extremal points are correctly identified and Newton method converges to the actual optimal uniform approximation. A guaranteed a posteriori numerical optimality verification is non-trivial in general: extremal points are identified using local optimality conditions in the two-step approach, confirming their global optimality requires using global techniques\footnote{Nonlinear branch-and-bounds algorithms~\cite{Sahinidis1996,ibex} can provide a distance to optimality, but cannot actually guarantee optimality. This latter can be proved using exclusion regions~\cite{Schichl2004} and branch-and-prune algorithms~\cite{ibex}.}, which are not in the scope of this survey. In the sequel, global optimality of extremal points where always successfully verified a posteriori by hand on a case by case basis.

%%%%%%%%%%%%%%%%%%%%%%%%%%%%%%%%%%%%%%%%%%%%%%%%%%%%%%%%%%%%%%%%%%%%%%
%%%%%%%%%%%%%%%%%%%%%%%%%%%%%%%%%%%%%%%%%%%%%%%%%%%%%%%%%%%%%%%%%%%%%%
%\subsection{Chebyshev polynomials on unions of intervals}\label{ss:union-intervall}
%
%\cite{Foucart2019}. Just explain that the subgradient condition holds, and that if the Haar condition holds in the convex hull of all intervals then the subgradient condition reads as the usual equioscillation theorem. Since this holds for polynomials, we can apply it to the cases of \cite{Foucart2019} and compute the polynomials by a two phase method, e.g.,~\cite{Hettich1976}. Give an example using a simple call to simplex algorithm, and one Newton iteration, ask Valentin?

%%%%%%%%%%%%%%%%%%%%%%%%%%%%%%%%%%%%%%%%%%%%%%%%%%%%%%%%%%%%%%%%%%%%%%
%%%%%%%%%%%%%%%%%%%%%%%%%%%%%%%%%%%%%%%%%%%%%%%%%%%%%%%%%%%%%%%%%%%%%%
\subsection{Uniform approximation of the $n$ dimensional Runge function}\label{ss:Runge}

\begin{table}[tbhp]
\footnotesize
\captionsetup{position=top} %<- Needed for using subtables created with the subfig package
\caption{Data for the approximation of the Runge function: $m$ is the dimension of $X$; $\mathrm{deg}$ is the degree of the polynomial and $n$ the corresponding number of monomials, i.e., the dimension of the subspace; $\#\mathrm{act}$ is the number of active constraints at the solution of the discretized linear problem, $\#\mathrm{ext}$ is the number of nearby local error maximizers; $\#\mathrm{zero}$ is the number of zeros in the linear combination of subgradients for the optimal polynomial. The last three columns correspond to values of errors, rounded to six decimals: the maximal error of the discretized linear problem, the error at extremal points at the limit polynomial of Newton iteration, and the global error of the latter polynomial.}\label{tab:Runge}
\begin{center}
\subfloat[$m=2$, number of samples $=36^2=1296$]
{
\begin{tabular}{|c|c|c|c|c|c|c|c|} \hline
\hspace{0.1cm}deg\hspace{0.1cm} & \hspace{0.2cm}$n$\hspace{0.2cm} & \#act & \#ext & \#zero & discrete error & Newton error & global error
\\\hline
$1$ & $3$ & $4$ & $4$ & $0$ & $0.308467$ & $0.310345$ & $0.310345$
\\\hline
$2$ & $6$ & $7$ & $7$ & $0$ & $0.165171$ & $0.165451$ & $0.165451$
\\\hline
$3$ & $10$ & $11$ & $9$ & $0$ & $0.091215$ & $0.091658$ & $0.091658$
\\\hline
$4$ & $15$ & $15$ & $14$ & $0$ & $0.062767$ & $0.062844$ & $0.062844$
\\\hline
$5$ & $21$ & $21$ & $16$ & $0$ & $0.039094$ & $0.039866$ & $0.039866$
\\\hline
\end{tabular}
}
\\\subfloat[$m=3$, number of samples $=10^3=1000$]
{
\begin{tabular}{|c|c|c|c|c|c|c|c|} \hline
\hspace{0.1cm}deg\hspace{0.1cm} & \hspace{0.2cm}$n$\hspace{0.2cm} & \#act & \#ext & \#zero & discrete error & Newton error & global error
\\\hline
$1$ & $4$ & $5$ & $5$ & $0$ & $0.351315$ & $0.352793$ & $0.352793$
\\\hline
$2$ & $10$ & $11$ & $11$ & $2$ & $0.208558$ & $0.221605$ & $0.221605$
\\\hline
\end{tabular}
}
\\\subfloat[$m=4$, number of samples $=6^4=1296$]
{
\begin{tabular}{|c|c|c|c|c|c|c|c|} \hline
\hspace{0.1cm}deg\hspace{0.1cm} & \hspace{0.2cm}$n$\hspace{0.2cm} & \#act & \#ext & \#zero & discrete error & Newton error & global error
\\\hline
$1$ & $5$ & $6$ & $6$ & $0$ & $0.375742$ & $0.377351$ & $0.377351$
\\\hline
$2$ & $15$ & $16$ & $16$ & $1$ & $0.247837$ & $0.258191$ & $0.258191$
\\\hline
\end{tabular}
}
\\\subfloat[$m=5$, number of samples $=4^5=1024$]
{
\begin{tabular}{|c|c|c|c|c|c|c|c|} \hline
\hspace{0.1cm}deg\hspace{0.1cm} & \hspace{0.2cm}$n$\hspace{0.2cm} & \#act & \#ext & \#zero & discrete error & Newton error & global error
\\\hline
$1$ & $6$ & $7$ & $7$ & $0$ & $0.392578$ & $0.393671$ & $0.393671$
\\\hline
\end{tabular}
}
\\\subfloat[$m=6$, number of samples $=3^6=729$]
{
\begin{tabular}{|c|c|c|c|c|c|c|c|} \hline
\hspace{0.1cm}deg\hspace{0.1cm} & \hspace{0.2cm}$n$\hspace{0.2cm} & \#act & \#ext & \#zero & discrete error & Newton error & global error
\\\hline
$1$ & $7$ & $8$ & $8$ & $0$ & $0.397987$ & $0.405442$ & $0.405442$
\\\hline
\end{tabular}
}
\\\subfloat[$m=7$, number of samples $=2^7=128$]
{
\begin{tabular}{|c|c|c|c|c|c|c|c|} \hline
\hspace{0.1cm}deg\hspace{0.1cm} & \hspace{0.2cm}$n$\hspace{0.2cm} & \#act & \#ext & \#zero & discrete error & Newton error & global error
\\\hline
$1$ & $8$ & $9$ & $9$ & $0$ & $0.40974$ & $0.414405$ & $0.414405$
\\\hline
\end{tabular}
}
\\\subfloat[$m=8$, number of samples $=2^{8}=256$]
{
\begin{tabular}{|c|c|c|c|c|c|c|c|} \hline
\hspace{0.1cm}deg\hspace{0.1cm} & \hspace{0.2cm}$n$\hspace{0.2cm} & \#act & \#ext & \#zero & discrete error & Newton error & global error
\\\hline
$1$ & $9$ & $10$ & $10$ & $0$ & $0.418580$ & $0.421501$ & $0.421501$
\\\hline
\end{tabular}
}
\\\subfloat[$m=9$, number of samples $=2^{9}=512$]
{
\begin{tabular}{|c|c|c|c|c|c|c|c|} \hline
\hspace{0.1cm}deg\hspace{0.1cm} & \hspace{0.2cm}$n$\hspace{0.2cm} & \#act & \#ext & \#zero & discrete error & Newton error & global error
\\\hline
$1$ & $10$ & $11$ & $11$ & $0$ & $0.42545$ & $0.427283$ & $0.427283$
\\\hline
\end{tabular}
}
\\\subfloat[$m=10$, number of samples $=2^{10}=1024$]
{
\begin{tabular}{|c|c|c|c|c|c|c|c|} \hline
\hspace{0.1cm}deg\hspace{0.1cm} & \hspace{0.2cm}$n$\hspace{0.2cm} & \#act & \#ext & \#zero & discrete error & Newton error & global error
\\\hline
$1$ & $11$ & $12$ & $12$ & $0$ & $0.430968$ & $0.432102$ & $0.432102$
\\\hline
\end{tabular}
}
\end{center}
\end{table}

\begin{figure}
\centering
\includegraphics[width=0.29\linewidth]{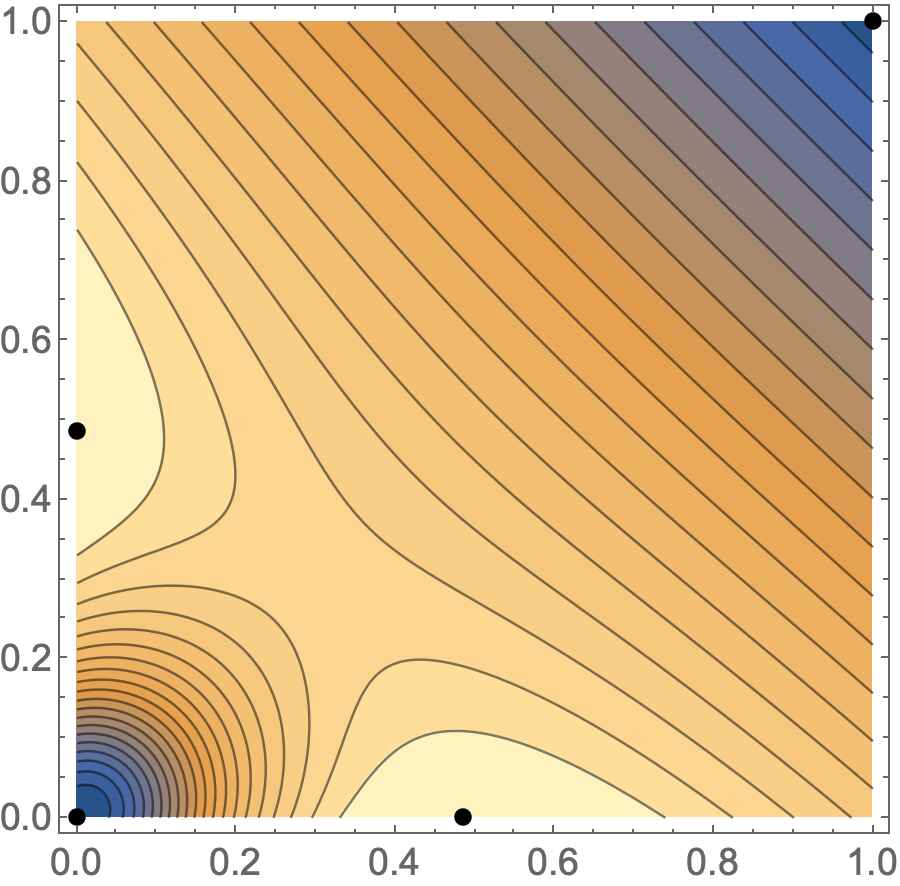}\hfill\includegraphics[width=0.29\linewidth]{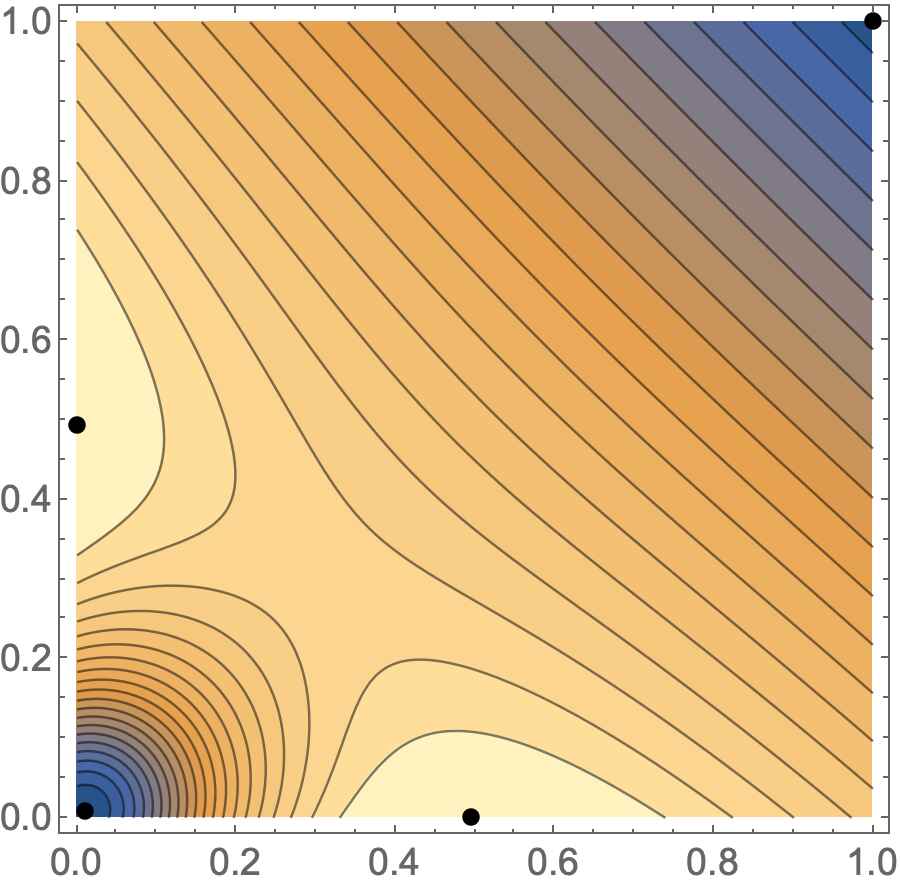}\hfill\includegraphics[width=0.4\linewidth]{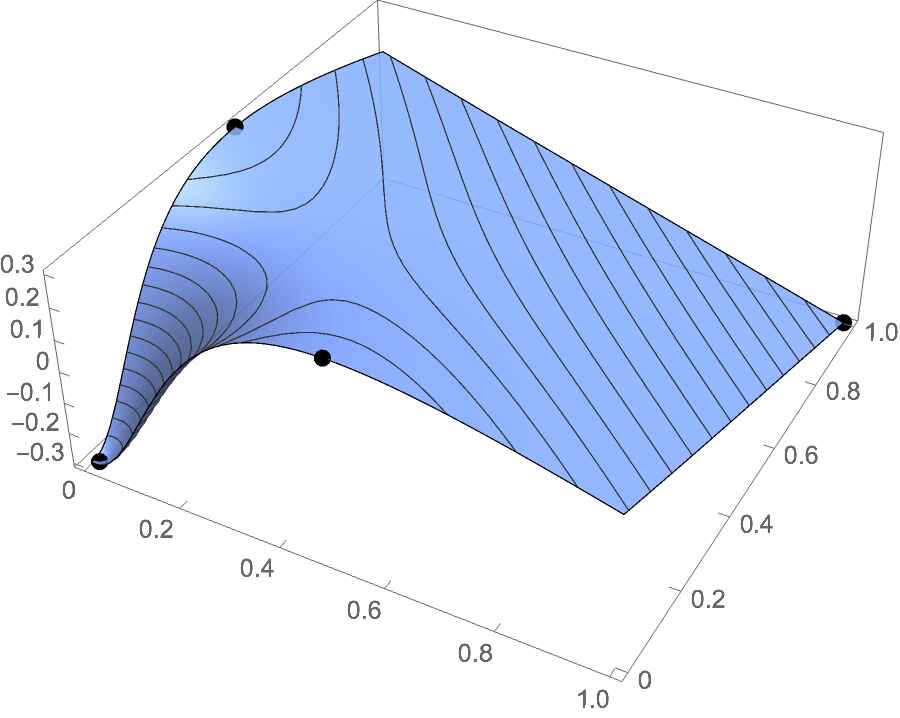}
\\\includegraphics[width=0.29\linewidth]{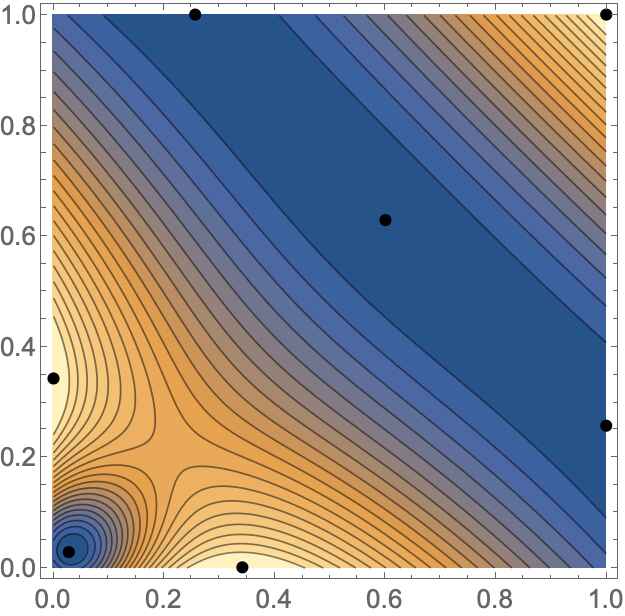}\hfill\includegraphics[width=0.29\linewidth]{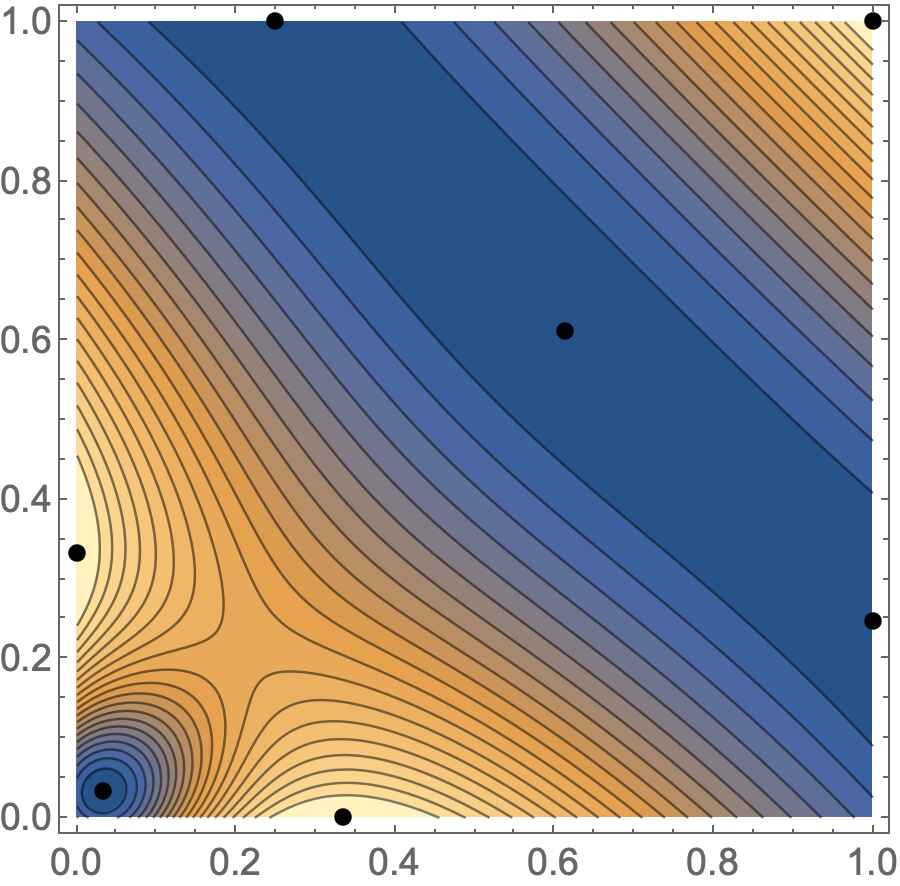}\hfill\includegraphics[width=0.4\linewidth]{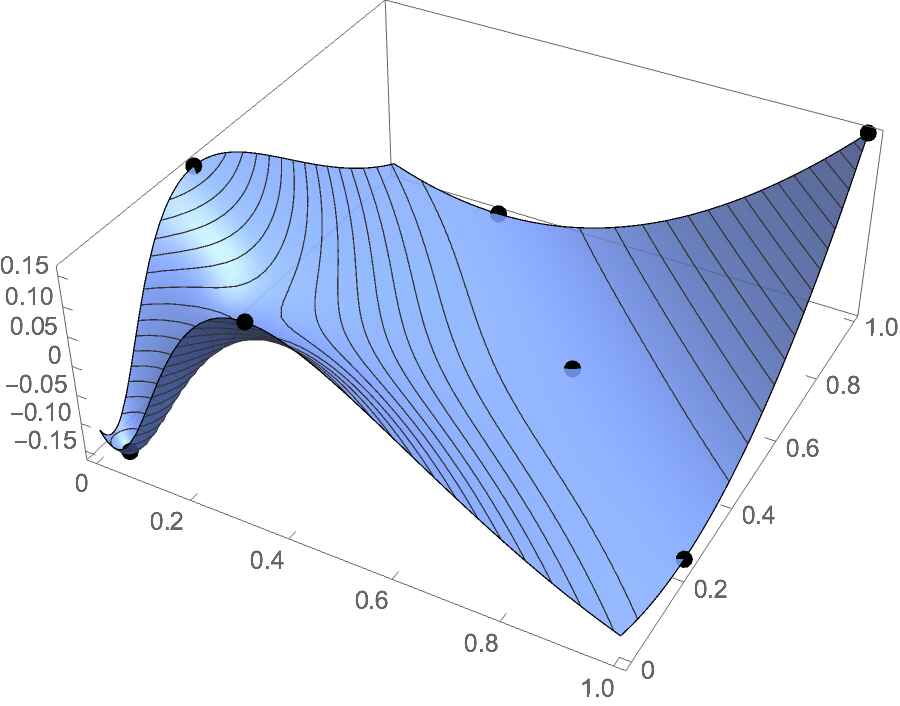}
\\\includegraphics[width=0.29\linewidth]{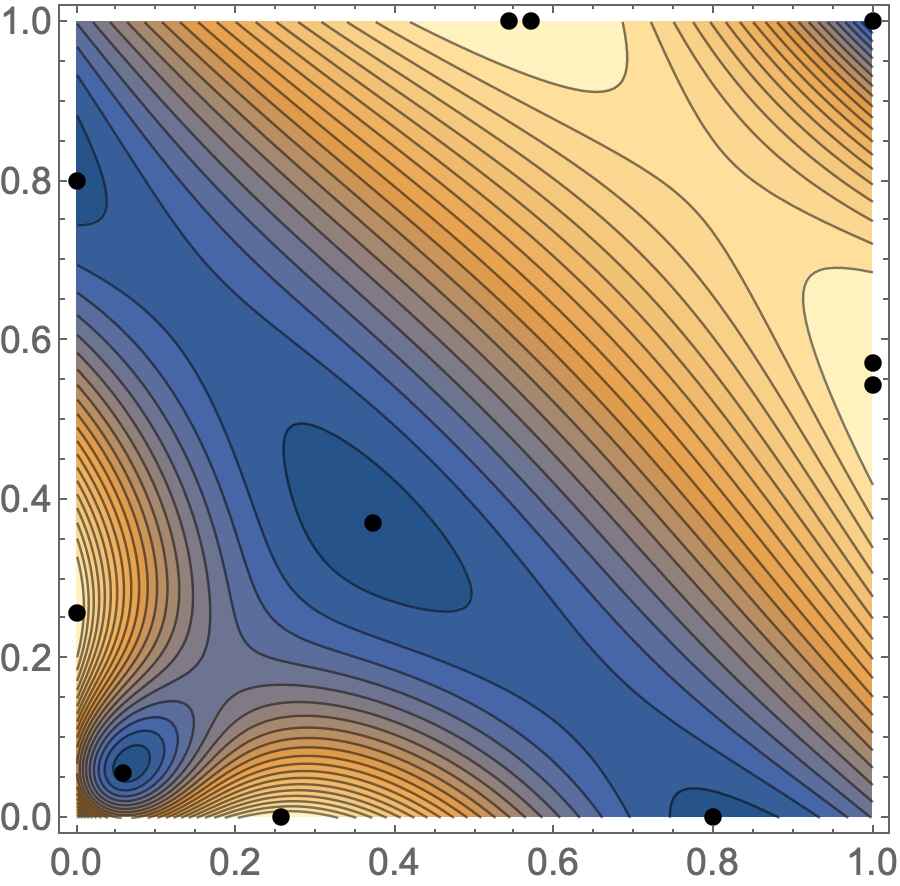}\hfill\includegraphics[width=0.29\linewidth]{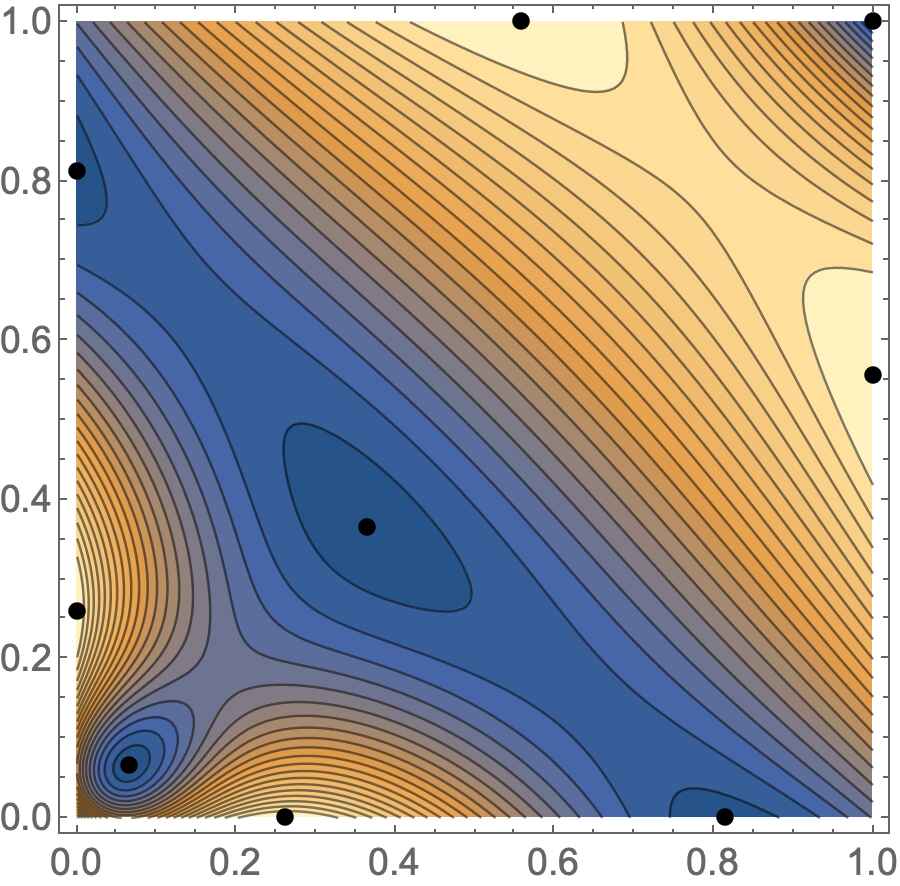}\hfill\includegraphics[width=0.4\linewidth]{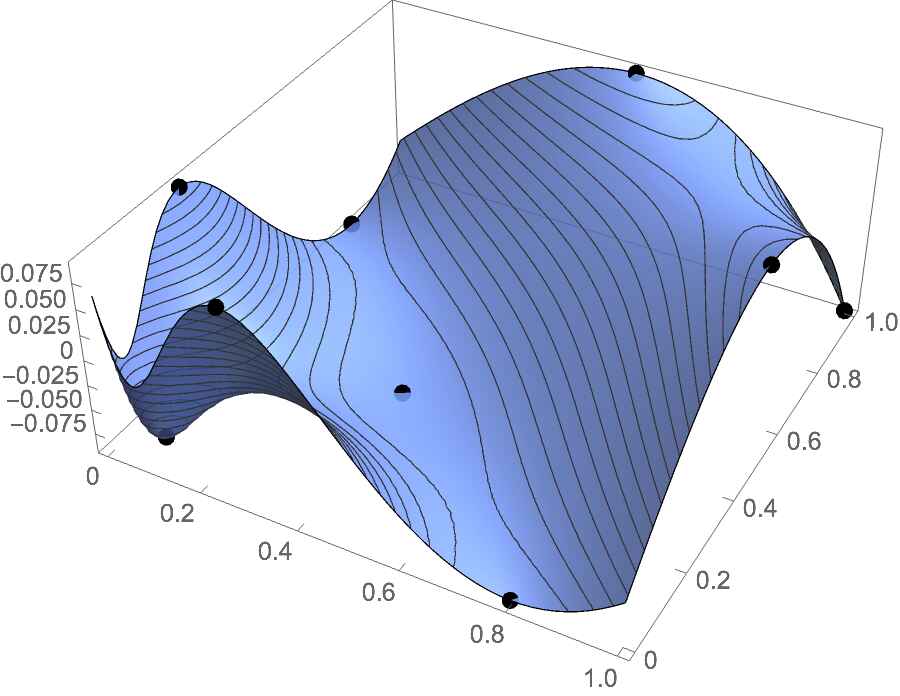}
\\\includegraphics[width=0.29\linewidth]{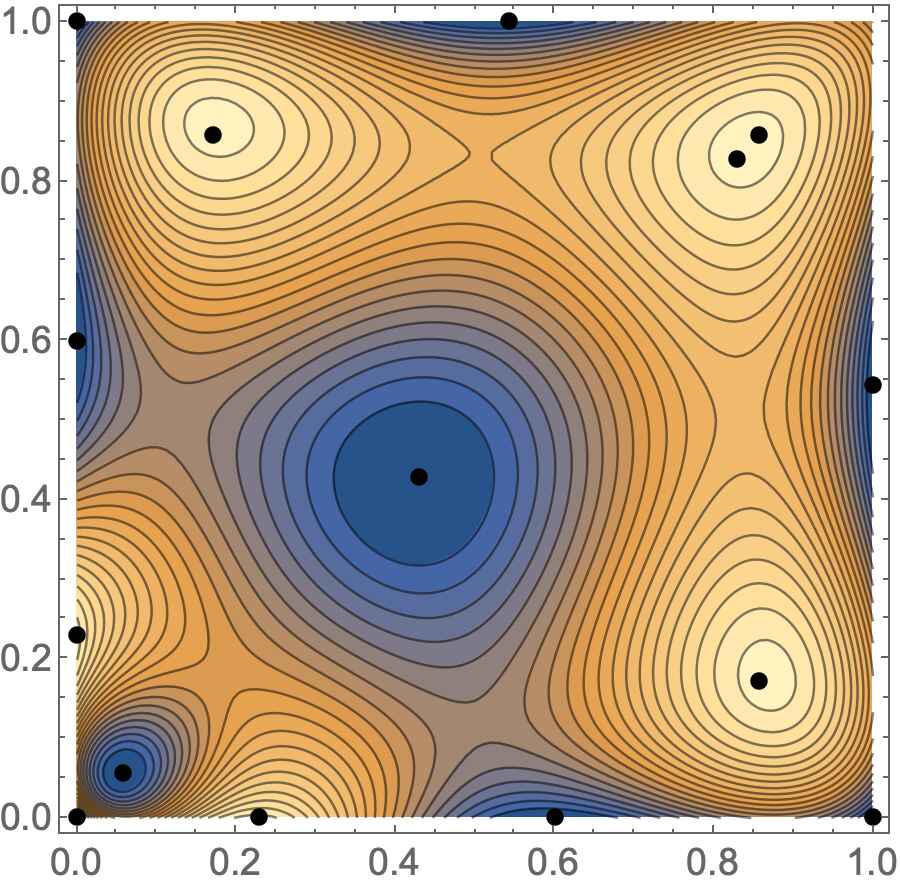}\hfill\includegraphics[width=0.29\linewidth]{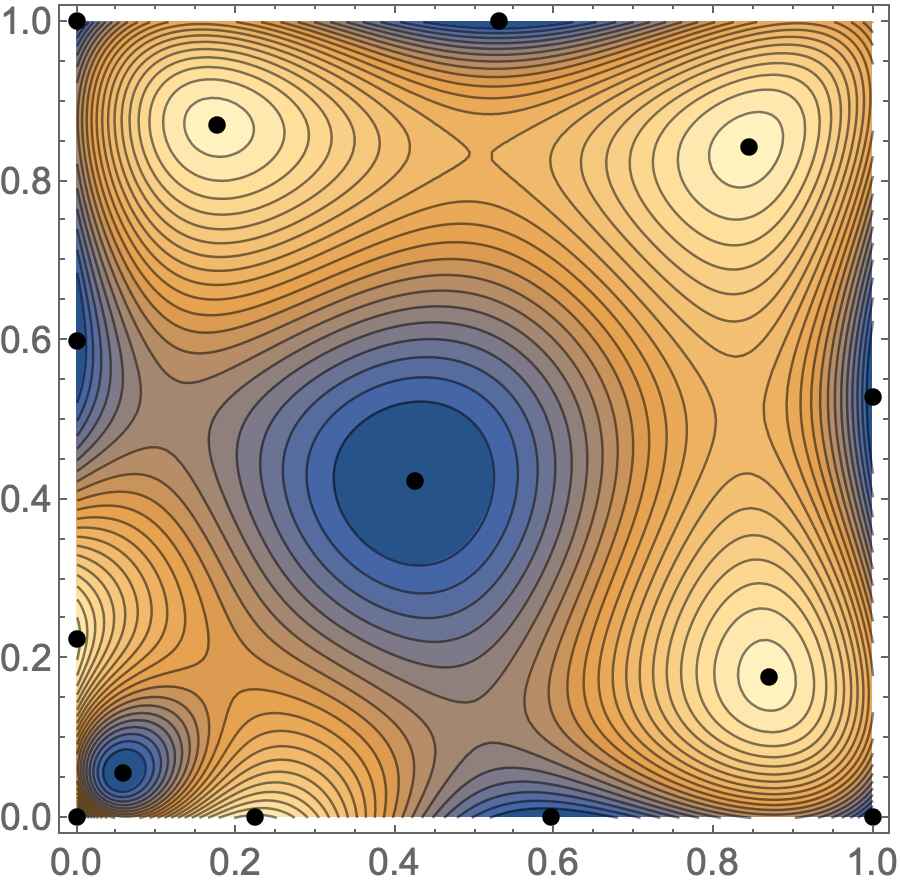}\hfill\includegraphics[width=0.4\linewidth]{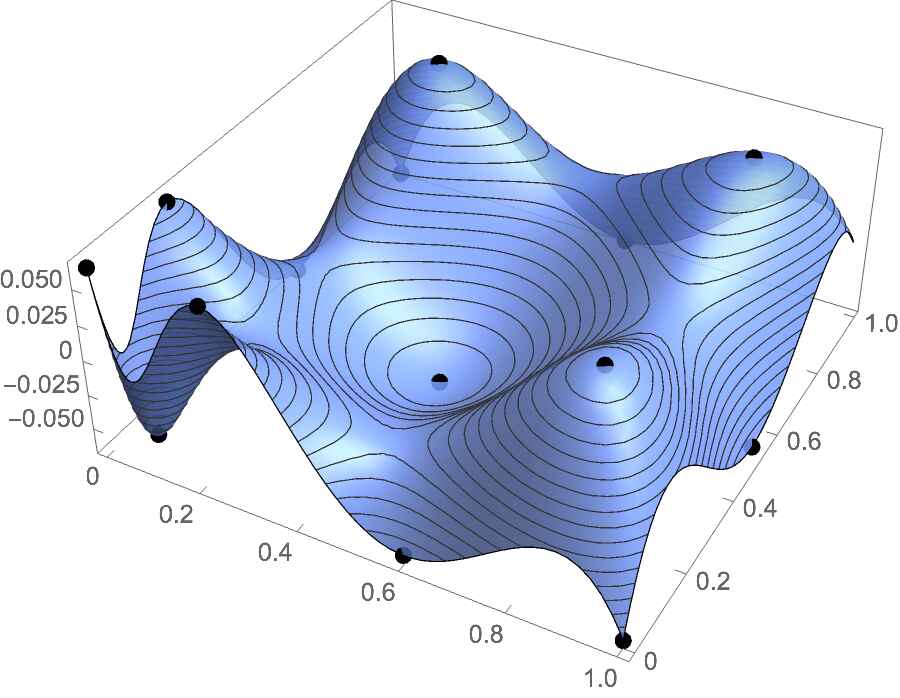}
\\\vspace{0.35cm}\includegraphics[width=0.29\linewidth]{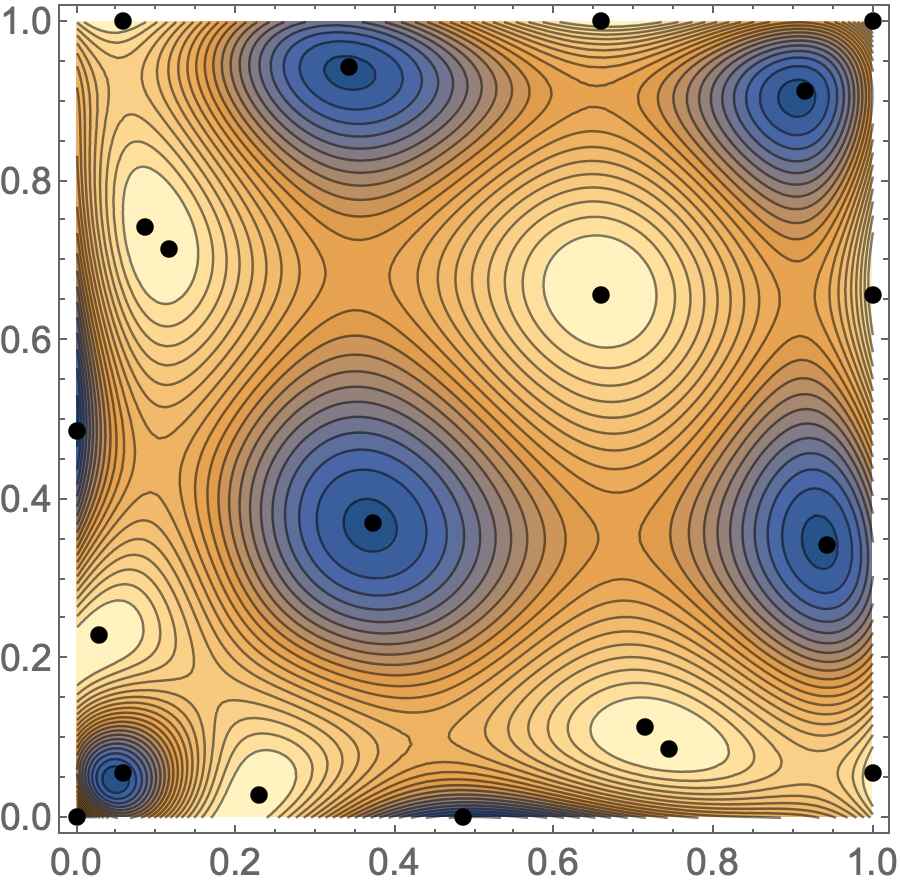}\hfill\includegraphics[width=0.29\linewidth]{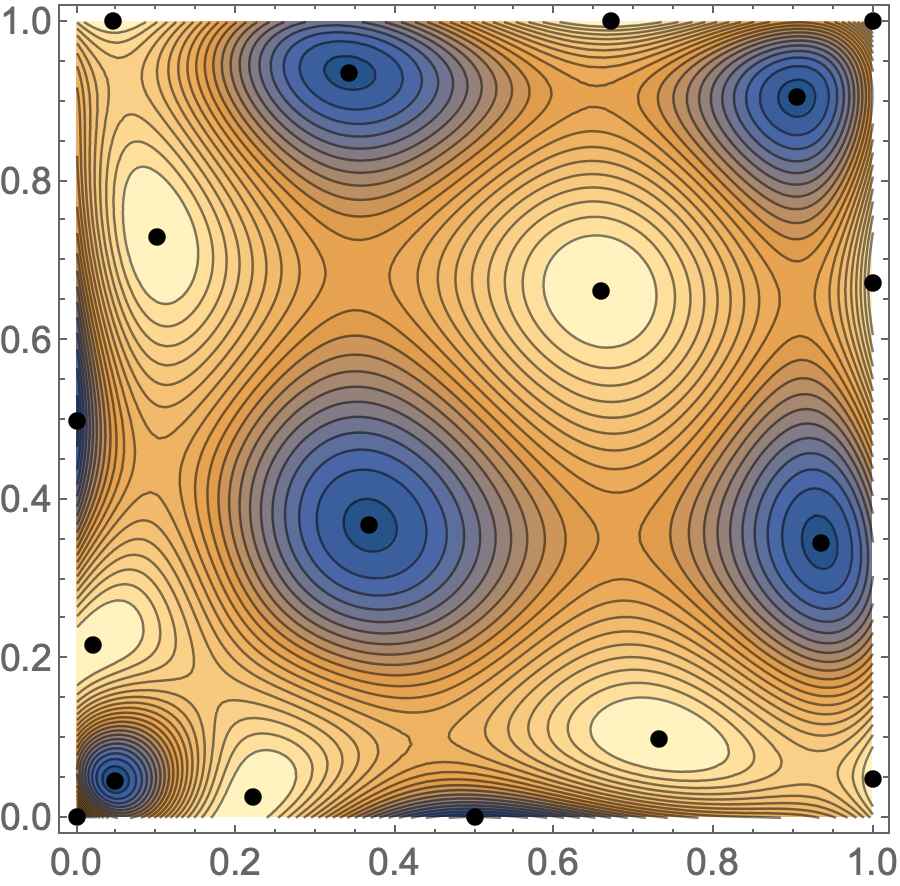}\hspace{0.8cm}\includegraphics[width=0.29\linewidth]{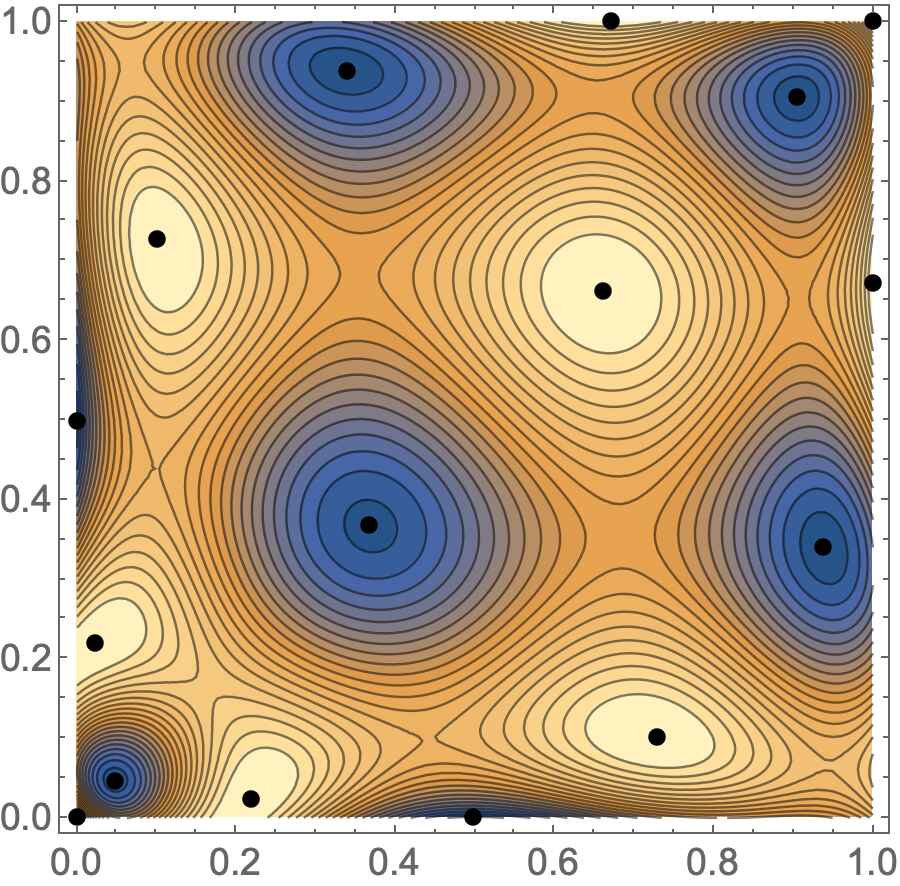}\hspace{0.65cm}
%\\\includegraphics[width=0.29\linewidth]{fig-Runge-5ext1-wrong}\hfill\includegraphics[width=0.29\linewidth]{fig-Runge-5ext2-wrong}\includegraphics[width=0.4\linewidth]{fig-Runge-5error}
\caption{Approximation of the two dimensional Runge function: The $i^\mathrm{th}$ row corresponds to approximation by degree $i$ polynomials. The first two columns show the error level-set of the polynomial computed by discretization (with active points detected watching the dual solution for the first column), and nearby local error maximizers. The third column shows the error of the optimal polynomial and extremal points of the corresponding signature.\label{fig:Runge}}
\end{figure}

We consider the multivariate Runge function~\cite{hecht2024} defined by
\begin{equation}
	f(x)=\frac{1}{1+25\, x^Tx}
\end{equation}
inside $X=[0,1]^m\subseteq\R^m$. We approximate it using the canonical polynomial basis of degree $d$, i.e., $\phi(x)=(1,x_1,x_2,\ldots,x_1x_2\cdots x_d,\ldots,x_n^d)\in\R^n$, where the order of the $n$ monomials in the basis is fixed but has no impact.

The first step of the two-step approach consists in solving a discretized problem. To this end, we consider a discretized domain $\tX\subseteq X$. Here, for $m=2$ we chose a regular grid with $36$ samples in each dimension, for a total of $36^2=1296$ samples. For higher dimensions, we chose the thinnest regular grid so that the total number of samples does not exceed $1296$. The resulting number of samples for each dimension $m\in\{1,2,\ldots,10\}$ is detailed in Table~\ref{tab:Runge}. The discretized approximation problem is
\begin{equation}
	\min_{a\in\R^n} \ \max_{x\in \tX} \ \bigl| \ \phi(x)^Ta-f(x) \ \bigr|.
\end{equation}
It is rewritten to the following linear problem (LP)
\begin{equation}
	\min_{\begin{subarray}{c}t\in\R,a\in\R^n\\\forall x\in\tX,\phantom{-}\phi(x)^Ta-f(x)\leq t\\\forall x\in\tX,-\phi(x)^Ta+f(x)\leq t\end{subarray}} t,
\end{equation}
which contains two linear inequality constraint for each $x\in\tX$. The solution $\aLP$ of this LP is meant to be close to the best polynomial approximation. This LP can be solved by a simple call to a linear solver. Although attractive by its simplicity, this first step faces two difficulties related to the conditioning of the LP to be solved: high degree polynomials expressed in the canonical basis lead to badly conditioned LPs (similarly to Vandermonde matrices that are badly conditioned), and thin discretization also lead to badly conditioned LPs when the optimal polynomial is not strongly unique (this situation corresponding to the so-called singular approximation problems~\cite{Osborne1969,Reemtsen1990,Watson2000}). This may lead to inaccurate numerical solutions computed by linear solvers, but we don't handle this inaccuracy explicitly here and use the approximate solution. The usage of exact LP solvers may improve this situation.

Once the polynomial $\pLP(x)=\phi(x)^T\aLP$ is computed as the solution of the LP solved in the first step, we need to identify the approximate extremal points of the error $\pLP(x)-f(x)$. These approximate extremal points will be corrected by the Newton iteration in the second phase, and therefore must be correctly identified. There is no obvious way of doing so: while extremal points of the optimal polynomial all have the same error, approximate extremal points of $\pLP$ should correspond to exact extremal points of the optimal polynomial, but their errors are approximately equal. Computing all extremal points and comparing their errors leads to complicated tuning of thresholds. Instead, we identify active constraints of the discretized LP\footnote{Instead of watching constraint activity directly, i.e., $\pm\phi(x)^T\tilde a\mp f(x)=\tilde t$ for active indices $x\in\tX$, which requires again tuning some thresholds that depend on the accuracy of the LP solver, we watch non-zero dual variables, which correspond to active constraints.}, which correspond to values of $x\in \tX$ that are extremal points for the discretization. These active indices are shows on the first column of Figure~\ref{fig:Runge}. The number of active constraints for an optimal solution of a LP is generically equal to the number of variables, here $n+1$ because of the auxiliary variable $t$ used in the LP formulation. The number of active constraints in the LP is given in Table~\ref{tab:Runge} (column $\#\mathrm{act}$), where we can see that indeed it is equal to $n+1$ in all cases but degree $4$ and degree $5$ approximations of the 2D Runge function, with no clear explanation of the reason why the number of active constraint does not correspond to the generic situation in these two cases. When the optimal polynomial is not strongly unique, the number of extremal points is less than $n+1$. As a consequence, in this situation we expect to have more active indices than extremal points. This can be clearly observed in the first graphic of the third, fourth and fifth rows of Figure~\ref{fig:Runge}, where two active indices closely surround some local maximizers.
%\footnote{Note that the local maximization of the error is very sensitive in those situations where a large number of local maximizers are grouped in a small domain, while convergence to the nearest local maximizer is mandatory. The algorithm should be tuned so that it does not jump from one local maximizer to another.}
Approximate extremal points are computed by performing a local maximization of the error starting from all identified active indices, and deleting duplicates. The approximate extremal points are denoted by $\xLP_1,\xLP_2,\ldots,\xLP_k\in\R^m$ are shown in the second column of Figure~\ref{fig:Runge}, while the number of approximate extremal points is given in Table~\ref{tab:Runge} (column $\#\mathrm{ext}$).

The variables used for the Newton iteration are the coefficients of the polynomial $a\in\R^n$, the coordinates of the extremal points $x_1,x_2,\ldots,x_k\in\R^m$, where $k$ is the number of approximate extremal points identified in the first-step, and the entries $\lambda\in\R^k$ of the kernel vector. The system of equations to be solved encodes a local version of the kernel optimality condition: the extremal points are required to be local maximizers of the error
\begin{equation}\label{eq:Newton-KKT}
\left\{\begin{array}{cl}
x_{ij}=0 & \text{if } \xLP_{ij}=0 
\\ x_{ij}=1 & \text{if } \xLP_{ij}=1 
\\ \tfrac{d}{dy_j}(\phi(y)^Ta-f(y))\Bigr|_{y=x_i}=0 & \text{otherwise}
\end{array}\right.,\ i\in\{1,\ldots,k\},\ j\in\{1,\ldots,m\},
\end{equation}
where we make the practically meaningful assumption that extremal points lie on the same boundary portion than the corresponding approximate extremal points\footnote{One could use full KKT conditions, but would need to guess a value of the KKT multipliers, which would result in the same system.}, with the same error
\begin{equation}\label{eq:Newton-extreme-point-errors}
\bigl|\phi(x_i)^Ta-f(x_i)\bigr|=\bigl|\phi(x_{i+1})^Ta-f(x_{i+1})\bigr| \ ,\ i\in\{1,\ldots,k-1\}.
\end{equation}
Finally, the kernel condition is
\begin{equation}\label{eq:Newton-convex}
	\sum_{i=1}^{k}\sign\bigl(\phi(x_i)^Ta-f(x_i)\bigr)\,\lambda_i\,\phi(x_i)=0,
\end{equation}
where the kernel vector $\lambda$ can be normalized by the linear constraint
\begin{equation}\label{eq:Newton-normalization}
\sum_{i=1}^{k}\lambda_i=1
\end{equation}
because all components of the solution must be non-negative. Note that the absolute value in~\eqref{eq:Newton-extreme-point-errors} and the sign function in~\eqref{eq:Newton-convex}, which are respectively non-differentiable and discontinuous, do not affect the Newton iteration because the sign of the error is expected to remain constant during the iteration. In fact, the approximation computed in the first step must be accurate enough so that the sign of the error is expected to remain constant during the Newton's iteration.

We have defined a system \eqref{eq:Newton-KKT}--\eqref{eq:Newton-extreme-point-errors}--\eqref{eq:Newton-convex}--\eqref{eq:Newton-normalization} of $k\,m+(k-1)+n+1=k\,m+k+n$ equations with $k\,m+k+n$ unknowns, which can be solved using Newton's iteration. To this end, we use the initial iterates $\aLP$ and $\xLP_i$ but need an initial iterate $\lLP$ as well. When $k=n+1$, the approximate subgradient matrix
\begin{equation}
	\SLP=\begin{pmatrix}\sLP_1\phi(\xLP_1)&\sLP_2\phi(\xLP_2)&\cdots&\sLP_k\phi(\xLP_k)\end{pmatrix}\in\R^{n\times k}
\end{equation}
is expected to have a kernel of dimension one containing a normalized vector having positive components, which can be used as initial kernel vector $\lLP$. When $k<n+1$, the optimal subgradient matrix has a kernel but the approximate subgradient matrix does not. Therefore we need to compute an approximate kernel vector, by dropping the least singular value of $\SLP$ to zero. This approach worked well for all presented examples. Other cases, where $k>n+1$ or the kernel has a dimension greater than one, have not been met in the examples presented here, and require more investigation.

We are now in position to run a Newton iteration with initial iterate $\aLP\in\R^n$, $\xLP_1,\ldots,\xLP_k\in\R^m$ and $\lLP\in\R^m$. In all test-cases, we have observed a quadratic convergence toward a solution $\aNEW\in\R^n$, $\xNEW_1,\ldots,\xNEW_k\in\R^m$ and $\lNEW\in\R^m$, which satisfies all constraints. Two checks must be performed a posteriori: first the error at the extremal points of the Newton iteration (reported in the column Newton error of Table~\ref{tab:Runge}) must agree with the global error of the computed polynomial $\bp(x)=\phi(x)^T\ba$ (whose value is computed using a branch-and-bound algorithm and is reported in the column global error of Table~\ref{tab:Runge}). In all test-cases, both errors agreed. Secondly, the components of the kernel vector $\lNEW$ must be checked to be non-negative: this is true in all test-cases excepted for the 2D Runge function approximated by a degree five polynomial. Therefore, excepted for this latter failed-case, the optimal polynomial approximation was successfully computed in all other test-cases. The global error of the optimal polynomial is shown in the third column of Figure~\ref{fig:Runge} together with its extremal points. For the failed-case, the actual optimal polynomial approximation has been computed using a Remez-like algorithm. Its error, depicted in the last row of Figure~\ref{fig:Runge}, shows that two approximate extremal points were wrongly selected after the first phase, illustrating the sensitivity of this simple approach.

We end this section watching the strong uniqueness of the computed polynomial approximations. Smarzewski’s condition (Theorem~\ref{thm:Smarzewski}) and the subgradient condition (Theorem~\ref{thm:convex-Chebyshev-approximation}) require that all components of $\lNEW$ are strictly positive and the matrix $\SNEW$ is full rank. The number of zero components in $\lNEW$ is given in Table~\ref{tab:Runge}. In all test-cases, the kernel of $\SNEW$ has dimension one (if it has a greater dimension, there would exist infinitely many solutions to the system and the Newton iteration could not converge quadratically) therefore the matrix $\SNEW$ is full rank if and only if there are $n+1$ extremal points, i.e., the matrix has one more column than rows. We conclude from Table~\ref{tab:Runge} that all affine optimal approximations are strongly unique, as well as the quadratic approximation of the 2D Runge function. All other computed approximations are seen to be not strongly unique.

%%%%%%%%%%%%%%%%%%%%%%%%%%%%%%%%%%%%%%%%%%%%%%%%%%%%%%%%%%%%%%%%%%%%%%
%%%%%%%%%%%%%%%%%%%%%%%%%%%%%%%%%%%%%%%%%%%%%%%%%%%%%%%%%%%%%%%%%%%%%%
\subsection{Uniform approximation of the DexTAR inverse geometric model}\label{ss:DexTAR}

The DexTAR parallel robot~\cite{Campos2010,Koessler2020} seen in the left picture of Figure~\ref{fig:DexTAR} is schematically represented in the upper left diagramme of Figure~\ref{fig:DexTAR-XP}: it is made of four rigid bars connected with revolute joints, two of them being actuated with joint coordinates $\theta_1$ and $\theta_2$. Its geometric model relates the coordinates of the end-effector $x=(x_1,x_2)^T$ and the joint coordinates of the activated revolute joints $\theta=(\theta_1,\theta_2)^T$ through a system of two equations $h(x,\theta)=0\in\R^2$
\begin{equation}\label{eq:geometric-model}
	h(x,\theta)=\begin{pmatrix}\bigl(-l+L\cos(\theta_1)-x_1\bigr)^2+\bigl(L\sin(\theta_1)-x_2\bigr)^2-L^2 \\ \bigl(\phantom{+}l+L\cos(\theta_{2})-x_1\bigr)^2+\bigl(L\sin(\theta_{2})-x_2\bigr)^2-L^2\end{pmatrix}.
\end{equation}
We choose the specifications of the actual DexTAR robot $l=59\,\mathrm{mm}$ and $L=90\,\mathrm{mm}$. We aim at approximating the inverse geometric model, which associate joint coordinates $\theta_1$ and $\theta_2$ to the pose coordinates $x_1$ and $x_2$. We choose a workspace $X=[0,40]\times[2,42]$ and a home configuration $x\approx(0,22)$ and $\theta=(\tfrac{\pi}{2},\tfrac{\pi}{2})$, such that $f(x,\theta)=0$, for which continuous inverse geometric models $\theta_1:X\rightarrow\R$ and $\theta_2:X\rightarrow\R$ are well defined.

The two-step approach is adapted to the approximation of these inverse geometric models. Each $\theta_l(x)$, $l\in\{1,2\}$, is approximated independently, so we fix it to one of the two values. The initial discretization $\tX$ is a $36\times36$ regular grid, and the inverse model is evaluated on the grid by solving the inverse geometric problem using a standard Newton iteration. Solving the corresponding LP leads to the approximation solution $\pLP(x)=\phi(x)^T\aLP$ computed in the first step. Same statistics and contour plots as in previous section are given in Table~\ref{tab:DexTAR} and Figure~\ref{fig:DexTAR-XP}.

For the second step, we need an explicit expression suited to the Newton iteration, so we cannot rely on the evaluation of the inverse models through an inner Newton iteration. Therefore we need to include explicitly the geometric model~\eqref{eq:geometric-model} in the system and avoid using explicitly the expression of the function to be approximated $f(x)=\theta_l(x)$, where $l\in\{1,2\}$ is the coordinate fixed from the begining: firstly, for each extremal point variable $x_i\in X$ we add a new variable $\theta_i\in\R^2$ and the equation
\begin{equation}
h(x_i,\theta_i)=0\in\R^2,
\end{equation}
thus keeping the system of equations square. Secondly, Equation~\eqref{eq:Newton-extreme-point-errors} cannot be used as is because it involves explicitly $f(x)=\theta_l(x)$. Enforcing extremal points to have the same error is now achieved by
\begin{equation}\label{eq:Newton-extreme-point-errors-DexTAR}
	|\phi(x_i)^Ta-\theta_{il}|=|\phi(x_{i+1})^Ta-\theta_{(i+1)\,l}|.
\end{equation}
Thirdly, Equation~\eqref{eq:Newton-KKT} involves the gradient of the error, which again cannot be used explicitly. Therefore, we use the implicit function theorem and replace $\tfrac{d}{dy}f(y)$ in Equation~\eqref{eq:Newton-KKT} by
\begin{equation}\label{eq:Newton-KKT-DexTAR}
\frac{d}{dy}\theta_l(y) = -\left(\Bigl(\frac{d}{d\theta}h(y,\theta)\Bigr)^{-1}\Bigl(\frac{d}{dy}h(y,\theta)\Bigr)\right)_{l:},
\end{equation}
where the subscript $l\!:$ means the $l^\mathrm{th}$ row. Equation~\eqref{eq:Newton-KKT-DexTAR} can be computed formally in our example. We finally obtain a system of equations that involves the geometric model $h(x,\theta)$ instead of the inverse geometric model $\theta_l(x)$. The Newton iteration was ran successfully for degrees ranging from $1$ to $4$, showing a quadratic convergence to the solution. The kernel vector was checked to be positive and the error at the extremal points of this solution was checked to equal the global error (see Table~\ref{tab:DexTAR}), hence showing that the limit of the Newton iteration satisfies the optimality condition. We see again in Table~\ref{tab:DexTAR} that best affine approximations are strongly unique, as well as the quadratic best approximation of $\theta_2(x)$. We observe too once again on Figure~\ref{fig:DexTAR-XP} that non strongly unique approximation have cluster of several active indices (here two) surrounding approximate extremal points.

\begin{table}[tbhp]
\footnotesize
\caption{Data for the approximation of the DexTAR inverse geometric model: the first and second line of each group of lines correspond to the approximation of $\theta_1(x)$ and $\theta_2(x)$ respectively. The columns are the same as in Table~\ref{tab:Runge}.}\label{tab:DexTAR}
\begin{center}
\begin{tabular}{|c|c|c|c|c|c|c|c|} \hline
deg & \hspace{0.2cm}$n$\hspace{0.2cm} & \#act & \#ext & \#zero & discrete error & Newton error & global error
\\\hline
   $1$ & $3$ & $4$ & $4$ & $0$ & $0.049361$ & $0.049362$ & $0.049362$
\\ $1$ & $3$ & $4$ & $4$ & $0$ & $0.150606$ & $0.150664$ & $0.150664$
\\\hline
$2$ & $6$ & $7$ & $6$ & $0$ & $0.007901$ & $0.007904$ & $0.007904$
\\ $2$ & $6$ & $7$ & $7$ & $0$ & $0.054600$ & $0.054646$ & $0.054646$
\\\hline
$3$ & $10$ & $11$ & $9$ & $0$ & $0.001379$ & $0.001380$ & $0.001380$
\\$3$ & $10$ & $11$ & $9$ & $0$ & $0.017893$ & $0.017913$ & $0.017913$
\\\hline
$4$ & $15$ & $16$ & $13$ & $0$ & $0.000237$ & $0.000237$ & $0.000237$
\\$4$ & $15$ & $16$ & $13$ & $0$ & $0.006038$ & $0.006047$ & $0.006047$
\\\hline
\end{tabular}
\end{center}
\end{table}

\begin{figure}
\centering
\includegraphics[width=0.4\linewidth]{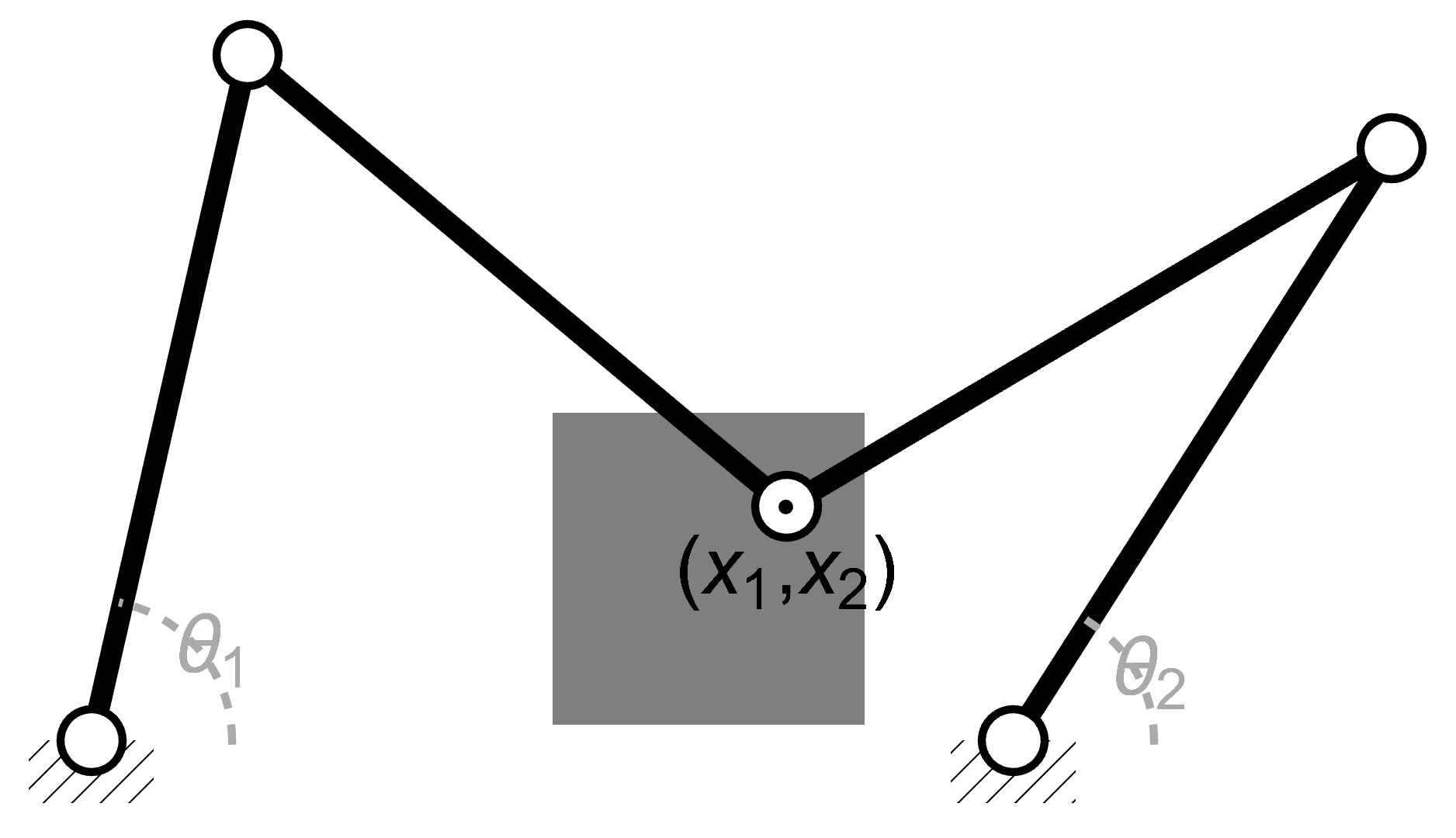}\hfill\includegraphics[width=0.29\linewidth]{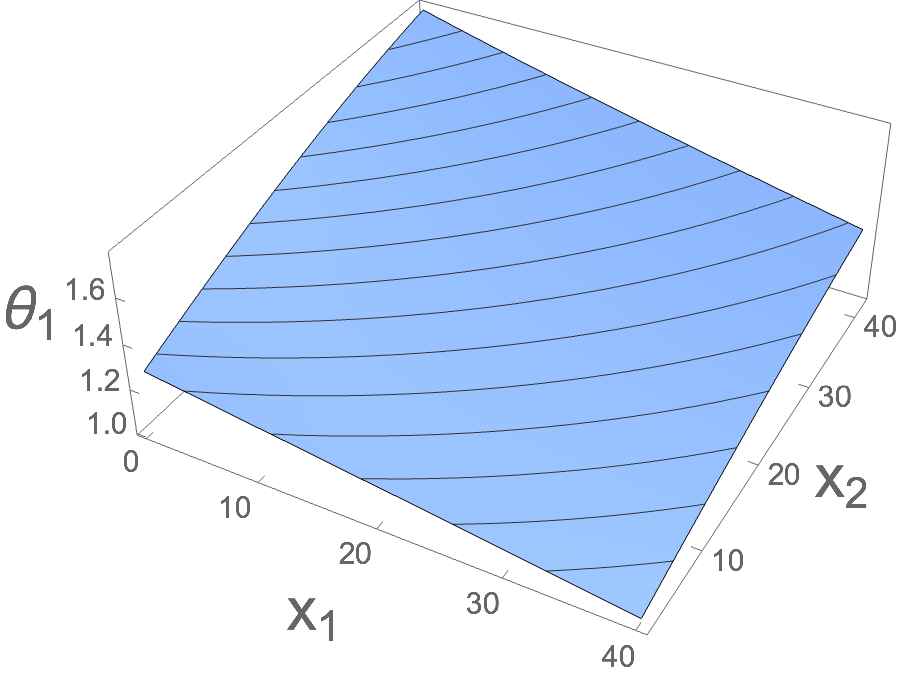}\hfill\includegraphics[width=0.29\linewidth]{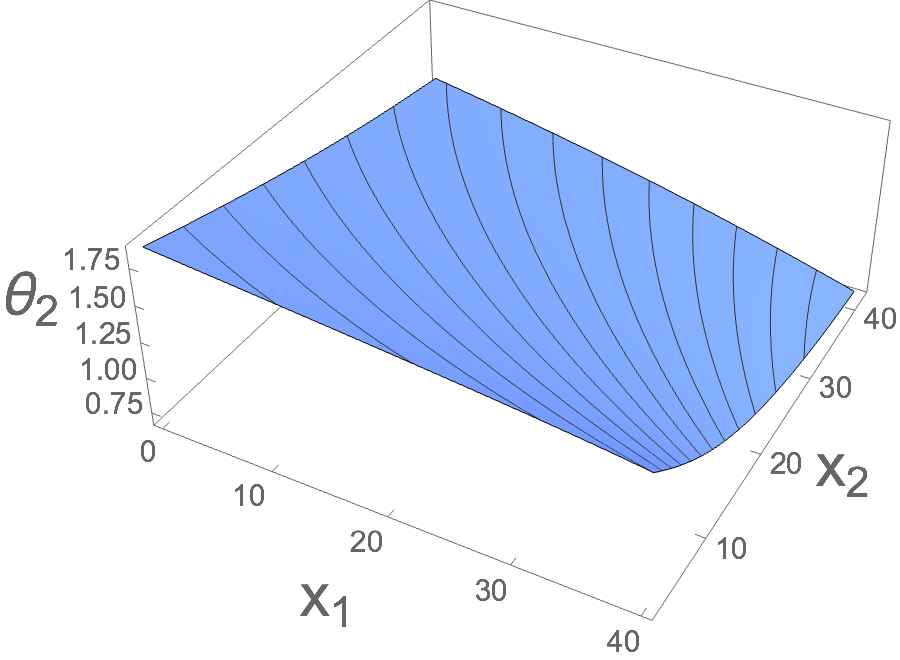}
\\\includegraphics[width=0.29\linewidth]{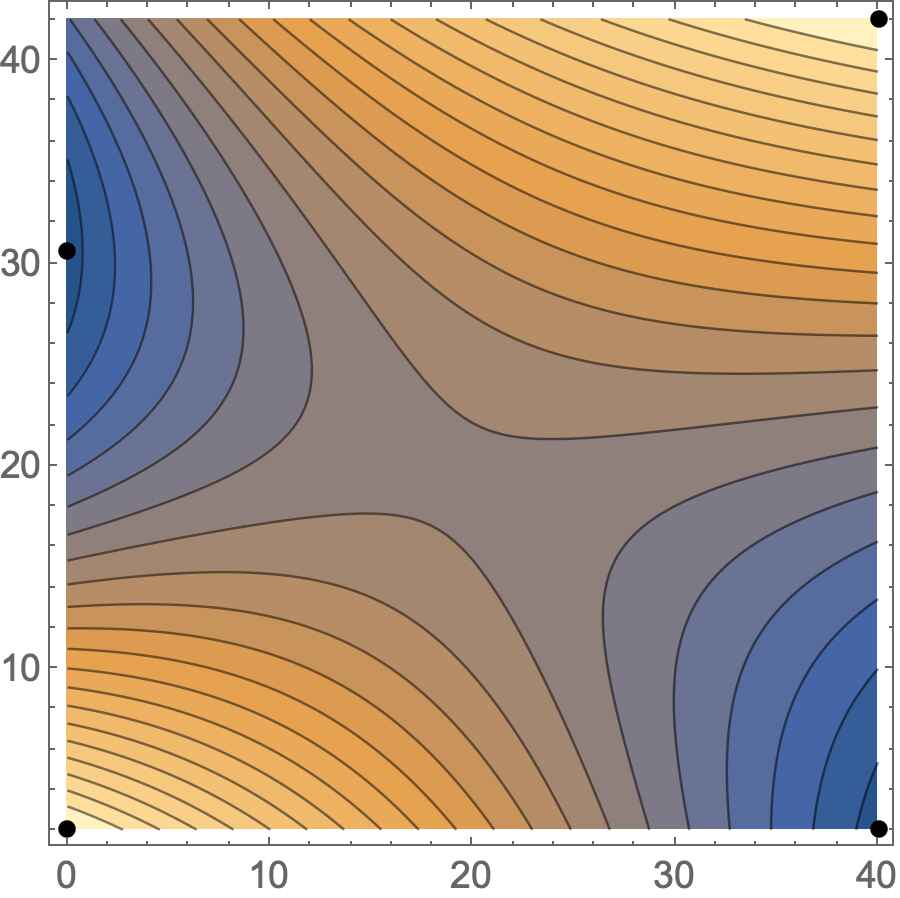}\hfill\includegraphics[width=0.29\linewidth]{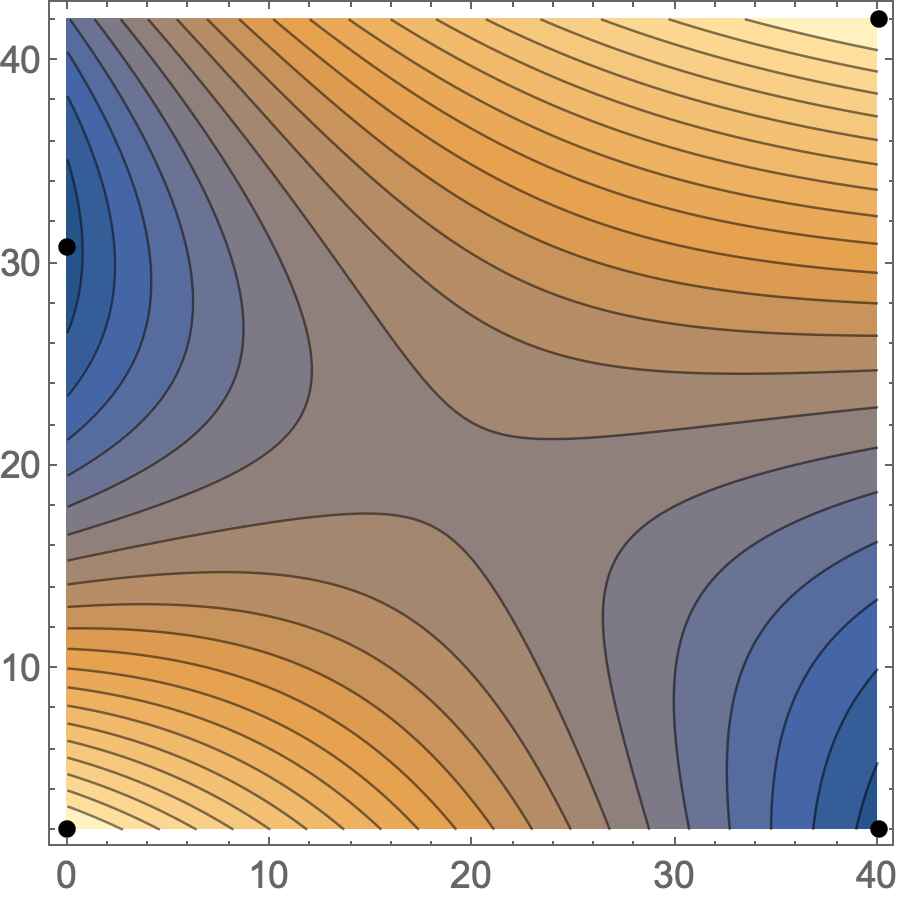}\hfill\includegraphics[width=0.4\linewidth]{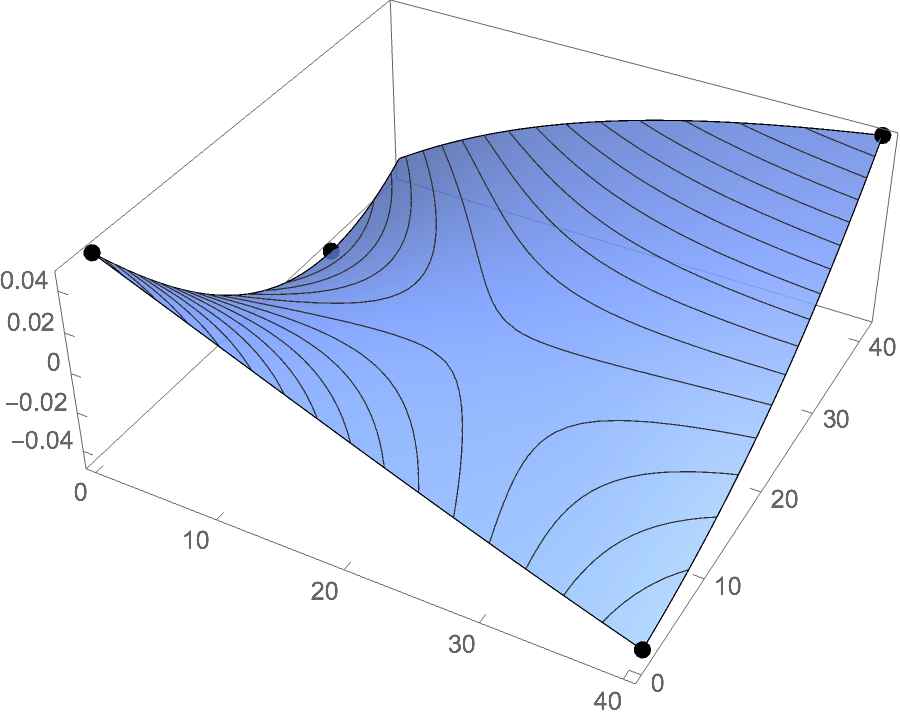}
\\\includegraphics[width=0.29\linewidth]{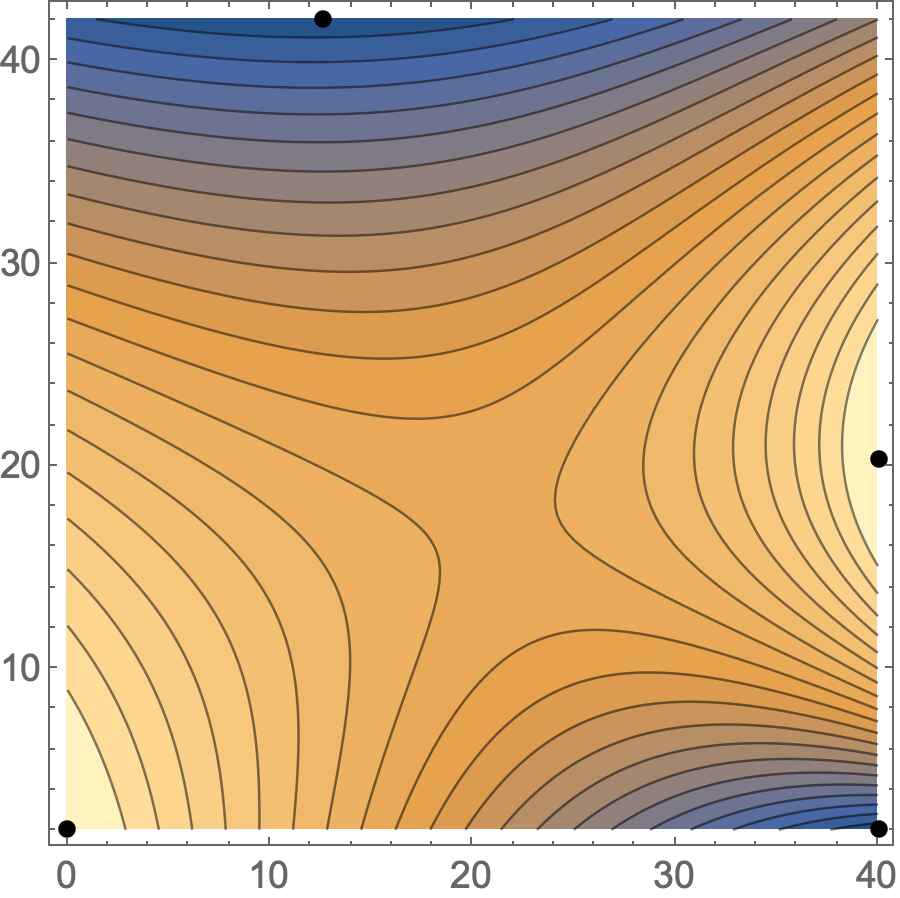}\hfill\includegraphics[width=0.29\linewidth]{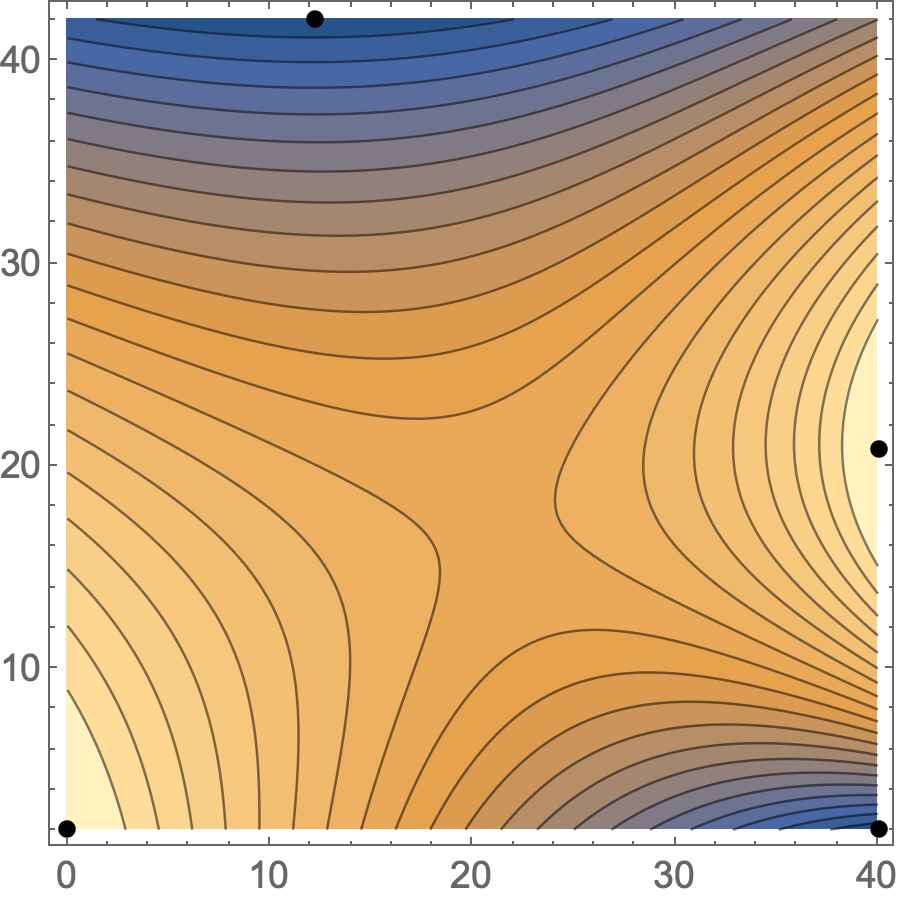}\hfill\includegraphics[width=0.4\linewidth]{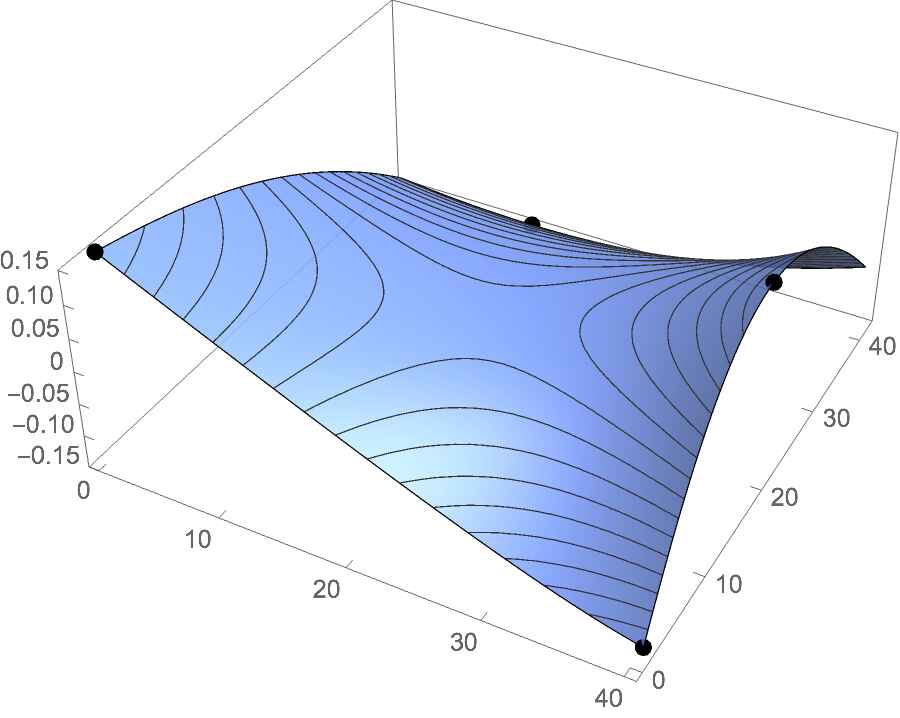}
\\\includegraphics[width=0.29\linewidth]{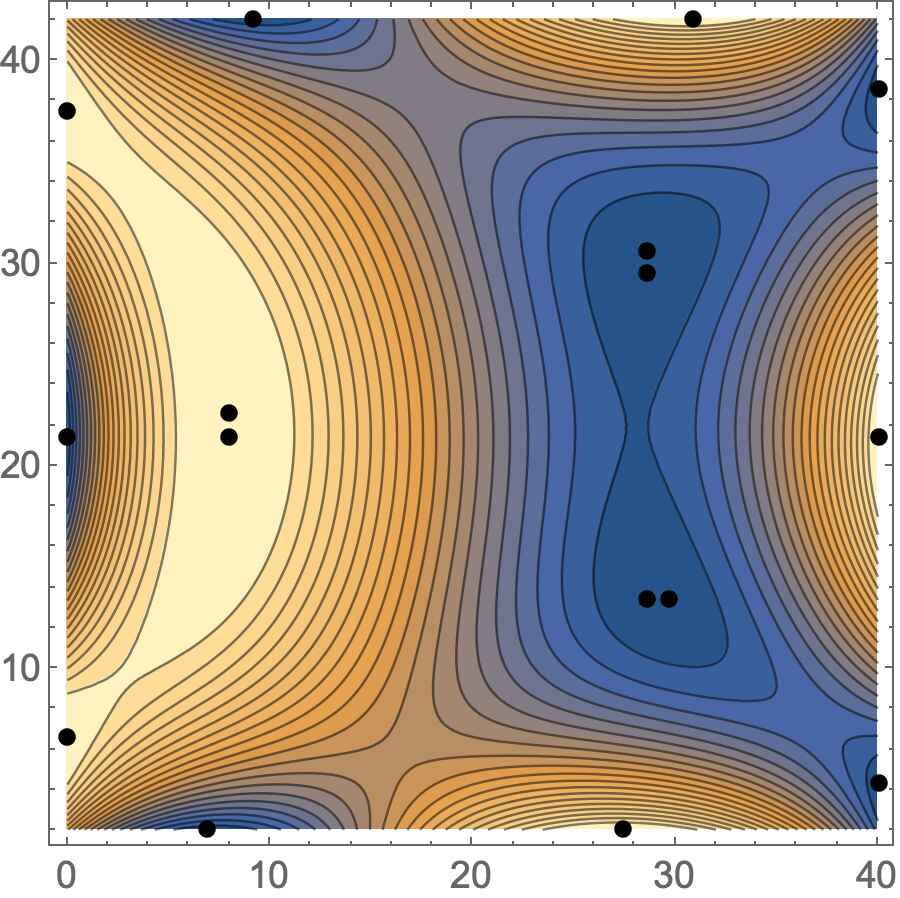}\hfill\includegraphics[width=0.29\linewidth]{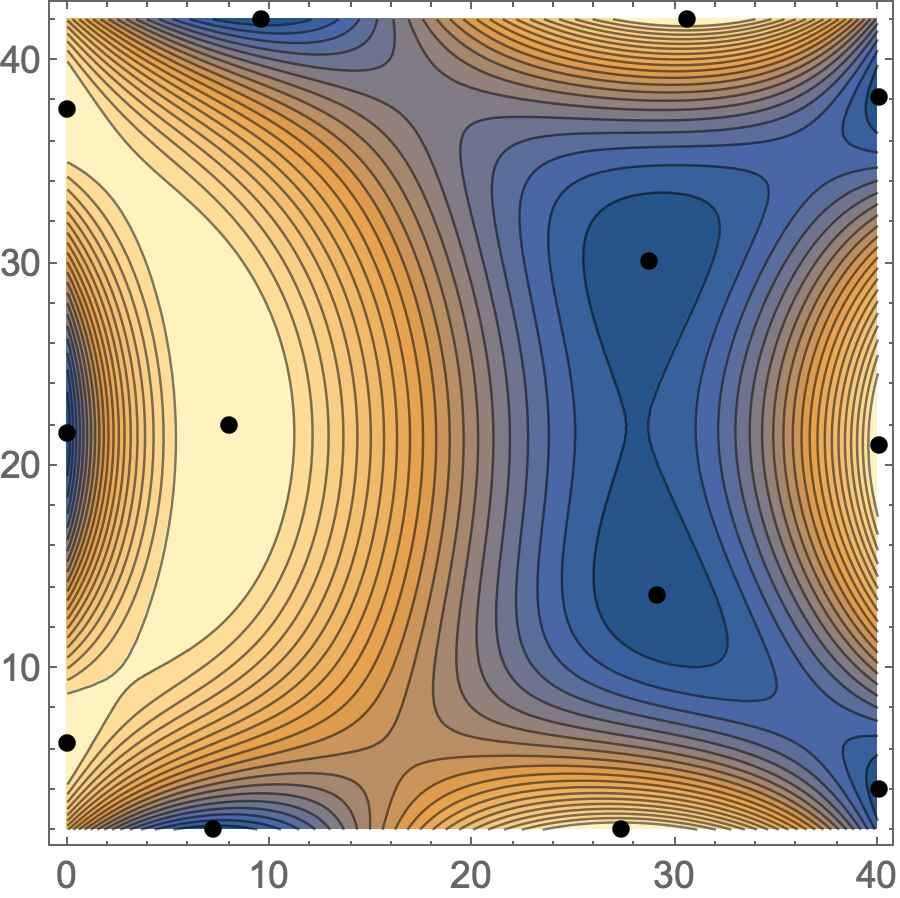}\hfill\includegraphics[width=0.4\linewidth]{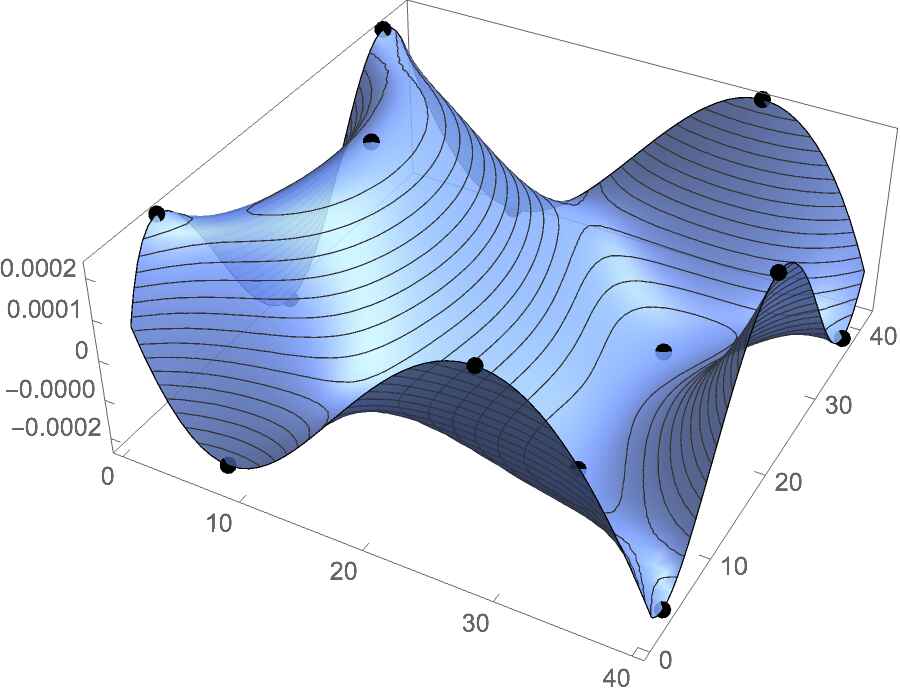}
\\\includegraphics[width=0.29\linewidth]{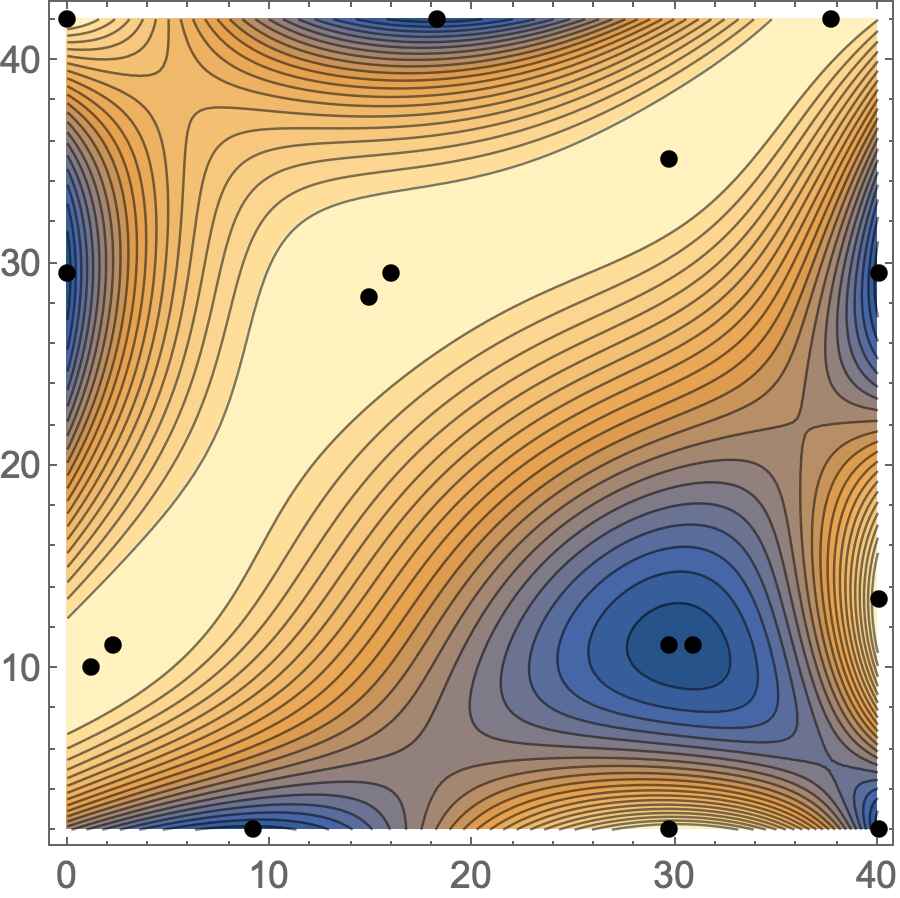}\hfill\includegraphics[width=0.29\linewidth]{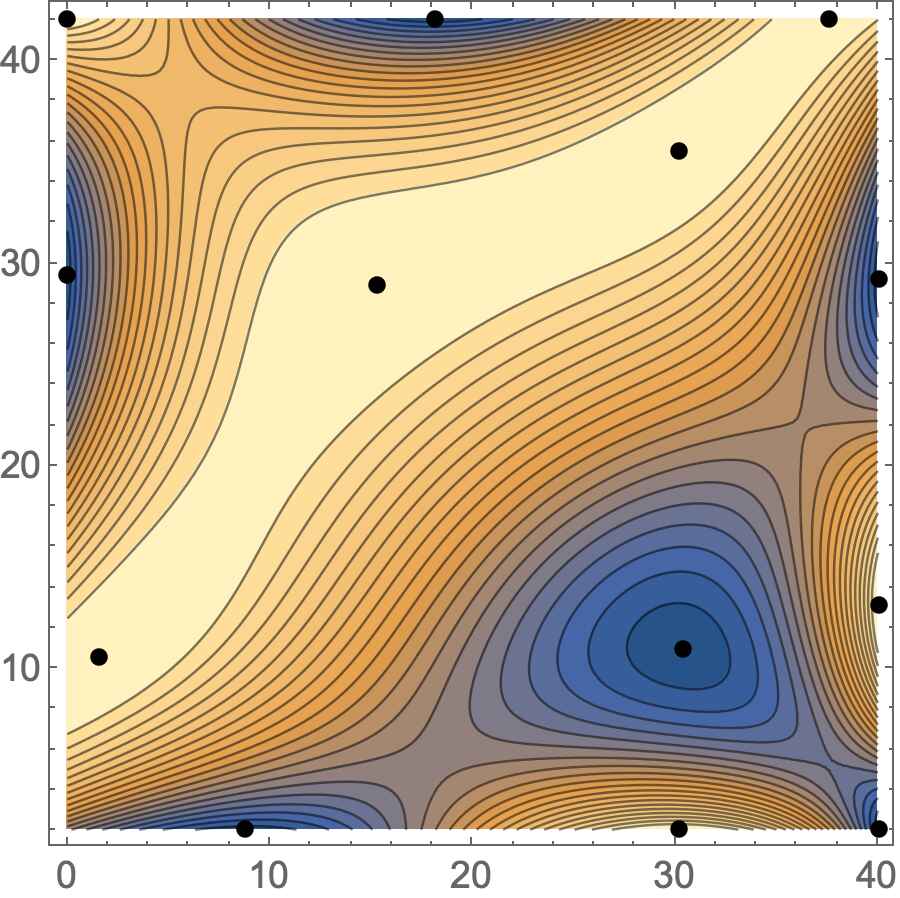}\hfill\includegraphics[width=0.4\linewidth]{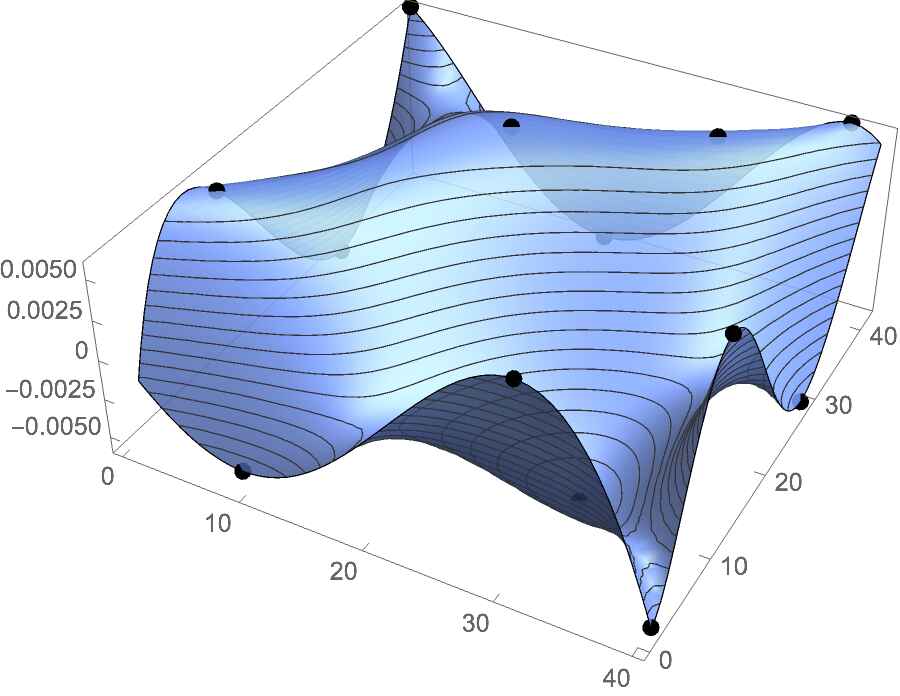}
\caption{The first row shows, from left to right, the diagramatic representation of the DexTAR, and the graphs of DexTAR inverse geometric models $\theta_1(x)$ and $\theta_2(x)$. The next four rows show the approximation of the two models by polynomials of degree $1$ and degree $4$.\label{fig:DexTAR-XP}}
\end{figure}

%%%%%%%%%%%%%%%%%%%%%%%%%%%%%%%%%%%%%%%%%%%%%%%%%%%%%%%%%%%%%%%%%%%%%%
%%%%%%%%%%%%%%%%%%%%%%%%%%%%%%%%%%%%%%%%%%%%%%%%%%%%%%%%%%%%%%%%%%%%%%
\subsection{Uniform approximation with polynomial evaluation error}\label{ss:approximation+evaluation}

In this section, we consider univariate polynomial approximation with $\phi:[\lx,\ux]\rightarrow\R^{n}$ and $\phi(x)=(1,x,x^2,\ldots,x^{n-1})$, so that $\phi(x)^Ta=\sum_{i=1}^{n}a_i\,x^{i-1}$. Arzelier, Bréhard and Joldes~\cite{Arzelier2025} proposed to minimize the evaluation error of a polynomial together with its approximation error. We restrict our attention to the evaluation of polynomials using the Horner form, whose worst case evaluation error can be approximated to first order by
\begin{equation}
u\,\sum_{j=1}^{n}c_j\bigl|\sum_{i=j}^na_i\,x^{i-1}\bigr|,
\end{equation}
where $u=2^{-p}$, $p\in\N$ being the precision of the floating point number format, i.e., number of bits of the mantissa~\cite{Muller2018}, and $c_1=c_n=1$ and $c_j=2$ otherwise. This can be written in matrix form $u\,\sum_{j=1}^{n}\,\bigl|(E_j\,\phi(x))^T\,a\bigr|$ with $E_j$ the diagonal matrix with zero on $j-1$ first diagonal entries, and $c_j$ in the other entries. For example, with $n=4$ we have
\begin{equation}
E_1=I, \ E_2=\begin{pmatrix}0&0&0&0\\0&2&0&0\\0&0&2&0\\0&0&0&2\end{pmatrix}, \ E_3=\begin{pmatrix}0&0&0&0\\0&0&0&0\\0&0&2&0\\0&0&0&2\end{pmatrix}, \ E_4=\begin{pmatrix}0&0&0&0\\0&0&0&0\\0&0&0&0\\0&0&0&1\end{pmatrix}.
\end{equation}
For writing convenience, let $e_0(a,x)=\phi(x)^Ta-f(x)$, and $e_i(a,x)=(E_i\,\phi(x))^T\,a$ for $i\in\{1,\ldots,n\}$, and $\vec e(a,x)=(e_0(a,x),\ldots,e_n(a,x))\in\R^{n+1}$, where the arrow emphasizes the vector valued nature of this error function. The approximation problem consists in minimizing worst case of the approximation error added to the linearized worst case evaluation error:
\begin{equation}
	m(a)=\max_{x\in[\lx,\ux]}\Bigl(|e_0(a,x)|+u\,\sum_{j=1}^{n}\,\bigl|e_j(a,x)|\Bigr).
\end{equation}
For writing convenience, we define $e(a,x)=|e_0(a,x)|+u\,\sum_{j=1}^{n}\,\bigl|e_j(a,x)|$. Extreme points of the error associated to the polynomial $p(x)=\phi(x)^Ta$ are now defined as $\ext(e_a)=\{x\in[\lx,\ux]:e(a,x)=m(a)\}$, where as previously $e_a(x)=e(a,x)$. In this context, some errors $e_i(a,x)$ may turn out to be zero for some extremal points. In order to have a convenient subgradient computation, we need to define signatures for the vector-valued error function $\sig(\vec e_a)\subseteq\R\times\R^{n+1}$ that include signs of each error $e_i(a,x)$, and so that if an error is zero then the corresponding extremal point appears twice with each sign:
\begin{equation}
	(x,s)\in\Sigma(\vec e_a) \iff x\in\ext(e_a) \text{ and } \forall i\in\{0,\ldots,n\},s_i\,e_i(a,x)\geq0.
\end{equation}
For example, if $x\in\ext(e_a)$ has $k$ zero errors then it will appear $2^k$ times inside $\sig(a)$, with different signs for each zero error. Finally, for $(x,s)\in\Sigma(\vec e_a)$ we define
\begin{align}
	\psi(x,s)&=\nabla\Bigl(s_0\,e_0(a,x)+u\,\sum_{j=1}^{n}\,s_j\,e_j(a,x)\Bigr)
	\\&=s_0\,\phi(x)+u\,\sum_{j=0}^{n}s_j\,E_j\,\phi(x).
	%=(s_0\,I+u\,\sum_{j=0}^{n}s_j\,E_j)\,\phi(x).
\end{align}
The following theorem is a kernel optimality condition for problem of minimizing the sum of the evaluation error of a polynomial and with its approximation error.

%, but for simplicity we will restrict our attention to cases where this does not occur, where we can define $\epsilon(a,x)=\sign e(a,x)$ and $\epsilon_i(a,x)=\sign e_i(a,x)$. The following theorem provides a kernel condition for the optimal solution of this problem. The proof is once again simple relying on the framework of convex analysis.

\begin{theorem}\label{thm:approximation+evaluation}
	The polynomial $\bp(x)=\phi(x)^T\ba$ is a minimizer of $m(a)$ if and only if there exists a finite signature $\{(x_1,s_1),\ldots,(x_m,s_m)\}\subseteq \sig(\vec e_\ba)$, where $s_i=(s_{i0},\ldots,s_{in})\in\{-1,1\}^{n+1}$, such that the matrix whose columns are $\psi(x_i,s_i)$ has a nonzero kernel vector with non-negative components.
\end{theorem}
\begin{proof}
The subdifferential of $m(\ba)$ is computed using standard rules:
\begin{align}
\partial m(\ba)&=\conv\bigl\{ \ \partial \bigl( \ |e_0(\ba,x)|+u\,\sum_{j=1}^{n}\,\bigl|e_j(\ba,x)| \ \bigr): \ x\in\ext(e_\ba) \ \bigr\}
\\ &=\conv\bigl\{ \ \partial|e_0(\ba,x)|+u\,\sum_{j=1}^{n}\,\partial\bigl|e_j(\ba,x)| \ : \ x\in\ext(e_\ba) \ \bigr\}
%\\ &=\conv\bigl\{ \ s_0\,\nabla e_0(\ba,x)+u\,\sum_{j=1}^{n}\,s_j\,\nabla e_i(\ba,x) \ : \ (x,s)\in\sig(\ba) \ \bigr\}.
\\ &=\conv\bigl\{ \ s_0\,\nabla e(a,x)+u\,\sum_{j=1}^{n}\,s_j\,\nabla e_j(a,x) \ : \ (x,s)\in\sig(\vec e_\ba) \ \bigr\}.
\end{align}
The last expression gives rise to $\partial m(\ba)=\conv\bigl\{\psi(x,s) : (x,s)\in\sig(\vec e_\ba) \bigr\}$. Finally, by Carathéodory's theorem, zero in the convex hull of this subdifferential is equivalent to zero in the convex hull of finitely many generators, which is the statement.
\end{proof}
We now use the two-step approach with the optimality condition of Theorem~\ref{thm:approximation+evaluation} for the Newton step. We use a local necessary condition of Theorem~\ref{thm:approximation+evaluation} in the form of a system of equations. Variables are $x_1,\ldots,x_k\in[\lx,\ux]$, $a_1,\ldots,a_n\in\R$ and $\lambda_1,\ldots,\lambda_k\in\R$, where $k$ is fixed to the number of thought extrema from the initial iterate. The first group of $k$ constraints encodes local extremality of each extremal point. For simplicity, we now assume that no error is zero so that local extremality can be characterized using derivatives (this assumption needs to be confirmed on the initial iterate). Local extremality is then expressed by
\begin{subequations}
\begin{align}
	x_1&=\lx \text{ or } \tfrac{\partial}{\partial x}e(x_1,a)=0
	\\\tfrac{\partial}{\partial x}e(x_i,a)&=0 \text{ for } i\in\{2,\ldots,k-1\}
	\\x_m&=\lx \text{ or } \tfrac{\partial}{\partial x}e(x_k,a)=0,
\end{align}
\end{subequations}
where the disjunctions for the first and last extremal points depends whether they are guessed to lie on the boundary or inside the domain (see Example~\ref{ex:approximation+evaluation} below). The second group of $k-1$ constraints encodes that extremal points need to have the same total error value:
\begin{equation}
	e(x_i,a)=e(x_{i+1},a)\text{ for } i\in\{1,\ldots,k-1\},
\end{equation}
which is differentiable with respect variables $x_i$, accordingly to the assumption that no error is zero at any extremal point. The third group of $n+1$ constraints encodes the kernel condition
\begin{equation}
\sum_{i=1}^k \lambda_i\,\psi(x_i,\epsilon(a,x_i))=0 \text{ and }\sum_{i=1}^k \lambda_i=1,
\end{equation}
where $\epsilon(a,x_i)=(\sign{e_0(a,x)},\ldots,\sign{e_n(a,x)})\in\{-1,1\}^{n+1}$, a linear normalization being used because kernel vectors are expected to be non-negative. We have finally constructed a square system of dimension $2k+n$, which is a local version of Theorem~\ref{thm:approximation+evaluation}.

The two-step approach is illustrated on the case of Example~3 of~\cite{Arzelier2025}.

\begin{figure*}[t]
\centering
\subfloat[Approximation error of $p_0(x)$.\label{fig:p0-approximation}]{
	\centering\includegraphics[width=0.3\textwidth]{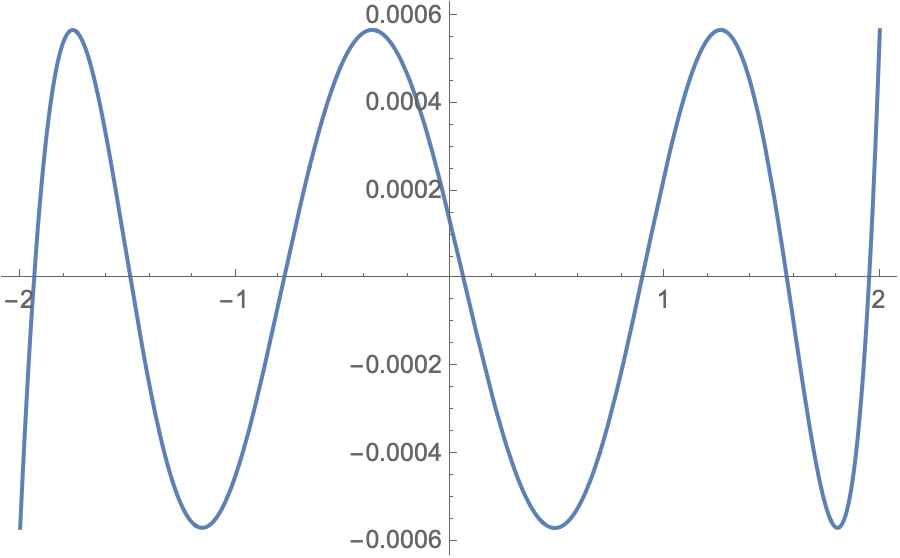}
}
\hspace{0.01\textwidth}
\subfloat[Worst case total error for $p_0(x)$ (linearized worst case evaluation error alone in orange).\label{fig:p0-approximation+evaluation}]{
	\centering\includegraphics[width=0.3\textwidth]{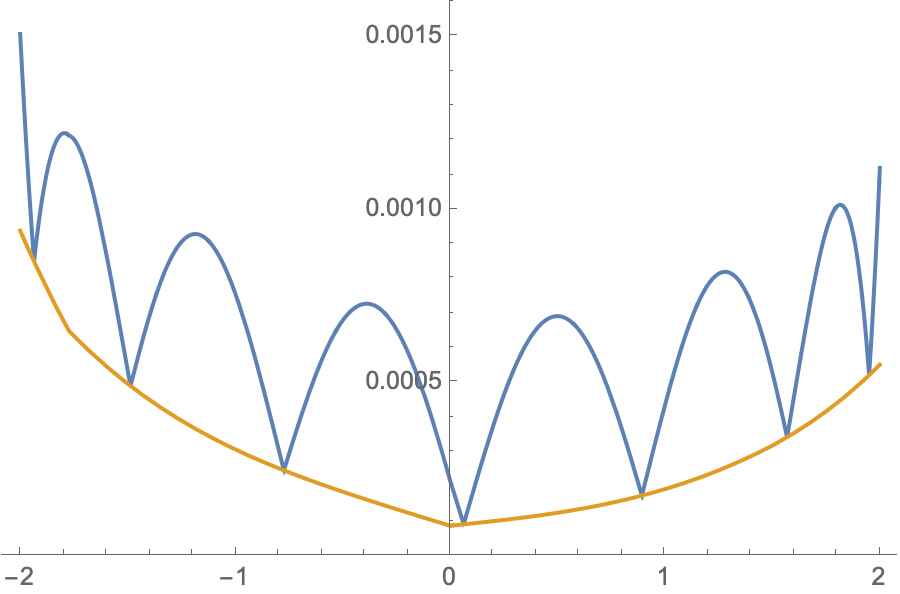}
}
\hspace{0.01\textwidth}
\subfloat[Worst case total error for $p_1(x)$ (linearized worst case evaluation error alone in orange).\label{fig:p0-approximation+evaluation-optimal}]{
	\centering\includegraphics[width=0.3\textwidth]{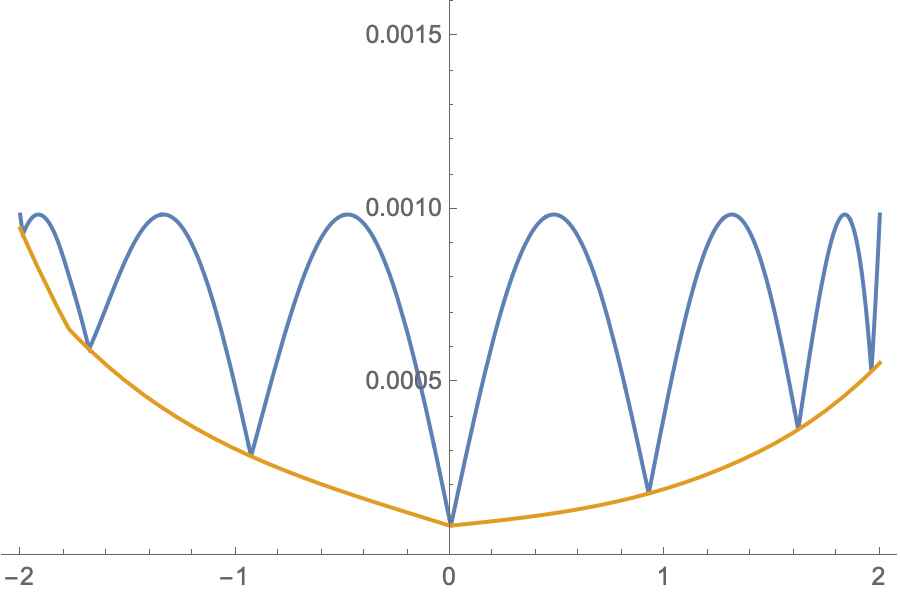}
}
\caption{Approximation and evaluation errors for different polynomial approximations.}
\end{figure*}

\begin{example}\label{ex:approximation+evaluation}
	The Airy function is to be approximated by a polynomial of degree~$6$ on the interval $[-1,1]$. We first approximate it in the usual Chebyshev sense. Using a sample of $81$ equidistant points (with sample distance $0.05$) and solving the corresponding finite linear problem, we obtain the polynomial
	\begin{multline}
	p_0(x)=0.00173x^6-0.0026 x^5-0.02068 x^4\\+0.06367 x^3-0.00088 x^2-0.26085 x+0.35516,
	\end{multline}
	where coefficients are rounded to $10^{-5}$. The error function $p_0(x)-f(x)$ is show in Figure~\ref{fig:p0-approximation}, where it is seen to approximatly equioscillate. Figure~\ref{fig:p0-approximation+evaluation} shows the linearized worst case evaluation error in orange, and the sum of the two. In order to apply Newton's method, we need initial guesses for the polynomial, the extremal points and the kernel vector. For the first two, we use $p_0(x)$ and approximations of its error extremal points given Table~\ref{table:extreme-points}.
\begin{table}
\centering
\begin{tabular}[t]{|c|c|c|c|c|c|c|c|c|}
\hline
$x_1$&$x_2$&$x_3$&$x_4$&$x_5$&$x_6$&$x_7$&$x_8$
\\\hline$-2.$& $-1.7943$& $-1.1847$& $-0.3875$& $0.4998$& $1.2803$& $1.8159$& $2.$
\\\hline
\end{tabular}
\caption{List of approximate extremal points for $p_0(x)-f(x)$.\label{table:extreme-points}}
\end{table}
	We clearly see that the first and last extremal points lie on the boundary of the interval, and we guess so for the optimal solution. The corresponding approximate subgradient matrix is
\begingroup
\everymath{\scriptstyle}
	\begin{equation}
\left(
\begin{array}{cccccccc}
 -0.9998 & 1.0002 & -0.9998 & 1.0002 & -0.9998 & 1.0002 & -0.9998 & 1.0002 \\
 2.0005 & -1.7939 & 1.1838 & -0.3878 & -0.4999 & 1.28 & -1.8163 & 1.9995 \\
 -4.0029 & 3.2173 & -1.4032 & 0.1502 & -0.2497 & 1.6395 & -3.2966 & 4.001 \\
 8.0098 & -5.7702 & 1.6631 & -0.0582 & -0.1247 & 2.1 & -5.9832 & 8.0059 \\
 -16.0273 & 10.3487 & -1.9712 & 0.0225 & -0.0624 & 2.6872 & -10.8701 & 16.0039 \\
 32.0391 & -18.5782 & 2.3342 & -0.0087 & -0.0312 & 3.4387 & -19.729 & 32.0234 \\
 -64.0625 & 33.3439 & -2.7646 & 0.0034 & -0.0156 & 4.4035 & -35.8165 & 64.0625 \\
\end{array}
\right),
	\end{equation}
\endgroup
	whose kernel vector $\lambda\approx(0.0597, 0.1188, 0.128, 0.14, 0.1476, 0.1563, 0.1648, 0.0848)$ can be used as an initial iterate for the Newton method. With these initial guesses, the Newton method converges to the polynomial
	\begin{multline}
	p_1(x)=0.0018 x^6-0.00277 x^5-0.02113 x^4\\+0.06447 x^3-0.00027 x^2-0.26164 x+0.35504,
	\end{multline}
	where coefficients are rounded to $10^{-5}$, in $5$ iterations with residual norm $10^{-14}$. Figure~\ref{fig:p0-approximation+evaluation-optimal} shows the sum of the worst case errors for this new polynomial: the approximation error is increased in the middle of the interval, where the evaluation error is small, and decreased near the bounds of the interval, where the evaluation error is large. All figures are consistent with the results from~\cite{Arzelier2025}.
\end{example}

%%%%%%%%%%%%%%%%%%%%%%%%%%%%%%%%%%%%%%%%%%%%%%%%%%%%%%%%%%%%%%%%%%%%%%
%%%%%%%%%%%%%%%%%%%%%%%%%%%%%%%%%%%%%%%%%%%%%%%%%%%%%%%%%%%%%%%%%%%%%%
%%%%%%%%%%%%%%%%%%%%%%%%%%%%%%%%%%%%%%%%%%%%%%%%%%%%%%%%%%%%%%%%%%%%%%
%%%%%%%%%%%%%%%%%%%%%%%%%%%%%%%%%%%%%%%%%%%%%%%%%%%%%%%%%%%%%%%%%%%%%%
\section{Conclusion}

%\blue{Restate here the main claim that optimality conditions for convex programming bring some new understanding and homogeneity to optimality conditions of real polynomial approximation, etc.}

Algorithms for computing multivariate best uniform approximations face difficulties that do not arise in the context of univariate approximation: firstly, the difficulty of computing and/or identifying error extremal points increases with the number of variables, which impacts both Remez-like algorithms and two-phase algorithms. Two-phase algorithms are simple but need a fine enough discretization, which becomes less and less tractable as the number of variables increases. Remez-like algorithms improve this situation by selecting samples corresponding to the worst error, but are impacted by the lack of strong uniqueness, which is common in the multivariate context. Some, attempts have been made to improve the situation, e.g.,~\cite{Osborne1969,Reemtsen1990,Sukhorukova2022}, and hybridizing Remez-like algorithms with Newton method applied to optimality conditions may offer some robust algorithmic framework for multivariate approximation.

The optimality condition most suited to algorithmic implementation are the kernel conditions. Nevertheless, all optimality conditions bring some insights. Although restricted to real-valued approximation, the subgradient conditions allowed homogenizing uniform approximation and relative Chebyshev centers through set-valued uniform approximation. On the other hand, subdifferential calculus and optimality conditions offer simple proofs of statements, which are more accurate in some cases. Furthermore, they present the advantage that they are related to affine underestimators, which bring some additional understanding and can be useful in practical algorithms, e.g., for finding descent directions or in subgradients algorithms. Interestingly, subgradient algorithms with memory, like Kelley's Method~\cite[Section 3.2.2 page 226]{Nesterov2018}, applied to the minimization of $m(a)$ are closely related to Remez's algorithm\footnote{In fact, the introduction cutting plane algorithms for convex programming in~\cite{Cheney1959,Kelley1960} was inspired by Remez's algorithm.}. In particular, the inner loop of the first Remez algorithm, which consists in maximizing $|e(a,x)|$ for a fixed $a$, actually turns out to being similar to computing a subgradient.

The key ingredient in this approach is the formula for the subdifferential of pointwise supremum functions. Extensions of this formula may lead to extensions of the optimality conditions in several directions:
\begin{itemize}
    \item The Formula for functions $m:\R^n\rightarrow\R\cup\{+\infty\}$ can handle constraint and can lead to optimality conditions for Chebyshev approximation problems with convex constraints on the coefficients. For example, finding the best uniform polynomial approximation with the constraint that the quadratic part of the polynomial is definite positive seems a challenge today.
    \item Modern formulas for subdifferential of pointwise supremum fonctions~\cite{Correa2023}, which don't require continuity nor compactness, can lead to extensions of approximation. For example, they may also help generalizing optimality conditions of relative Chebyshev centers to cases sets of function that are not totally complete.
    \item Clark generalized gradient applies to non-convex problem and enjoys a similar formula for pointwise supremum functions~\cite[Theorem 2.1 page 251]{Clarke1975}. This may lead to optimality conditions for nonlinear uniform approximation problems, e.g., low rank approximation approximations of multivariate functions~\cite{
%Townsend2013,Hashemi2017,
Zamarashkin2022}, multivariate generalized rational approximations~\cite{Millan2022}, which are quasiconvex, or other classes of nonlinear approximation~\cite{Malachivskyy2019}.
	\item As pointed out in Remark~\ref{rem:Kolmogorov}, using the pointwise supremum function $\tfrac{1}{2}\,\max_{x\in X}(\bp(x)-f(x))^2$ may allow applying optimality conditions for convex optimization to the uniform approximation of functions in the field of complex numbers.
   %\item Standard numerical algorithms behave badly when there are infinitely many minimizers (e.g., multivariate Chebyshev polynomials). The extension of sharpness to infinite sets of minimizers (J.V. Burke, M.C. Ferris (1993). Weak sharp minima in mathematical programming. SIAM Journal on Control and Optimization, 31, 1340–1359.) may allow using quadratic regularization to improve the performances of numerical algorithms.% \blue{\st{We conjecture that the set of minimizers of uniform approximation problems are weak sharp}}. If true, this may have interesting consequences, both theoretical (uniqueness would be equivalent to strong uniqueness) and practical (quadratic penalization).
\end{itemize}

%%%%%%%%%%%%%%%%%%%%%%%%%%%%%%%%%%%%%%%%%%%%%%%%%%%%%%%%%%%%%%%%%%%%%%
%%%%%%%%%%%%%%%%%%%%%%%%%%%%%%%%%%%%%%%%%%%%%%%%%%%%%%%%%%%%%%%%%%%%%%
%%%%%%%%%%%%%%%%%%%%%%%%%%%%%%%%%%%%%%%%%%%%%%%%%%%%%%%%%%%%%%%%%%%%%%
%%%%%%%%%%%%%%%%%%%%%%%%%%%%%%%%%%%%%%%%%%%%%%%%%%%%%%%%%%%%%%%%%%%%%%

%\appendix

\section*{Acknowledgments}

The author was funded by the CNRS grant Emergence AMEGO.

\bibliographystyle{siamplain}
\bibliography{SIREV_approximation,refs-goldsztejn}
\end{document}